\documentclass[letter,11pt]{article}

\usepackage[margin=1in]{geometry}
\usepackage[useregional]{datetime2}
\usepackage[english]{babel}
\usepackage[utf8]{inputenc}
\usepackage{xcolor, graphicx,changepage,xcolor, float, array, amssymb, amsmath, mathrsfs, amsfonts, hyperref, url, amsthm, tikz-cd, fancyhdr, enumerate, bbm, tikz-cd}
\makeatletter
\newcommand*{\rom}[1]{\expandafter\@slowromancap\romannumeral #1@}
\makeatother

\makeatletter
\renewcommand*\env@matrix[1][*\c@MaxMatrixCols c]{%
	\hskip -\arraycolsep
	\let\@ifnextchar\new@ifnextchar
	\array{#1}}
\makeatother

\theoremstyle{definition}
\newtheorem{theorem}{Theorem}[section]
\newtheorem{lemma}[theorem]{Lemma}
\newtheorem{definition}[theorem]{Definition}
\newtheorem{conjecture}[theorem]{Conjecture}
\newtheorem{question}[theorem]{Question}
\newtheorem{corollary}[theorem]{Corollary}
\newtheorem{example}[theorem]{Example}

\theoremstyle{remark}
\newtheorem{remark}[theorem]{Remark}

\usepackage{setspace}
\newcommand{\labeq}[1]{\label{eq:#1}}

\newcommand{\labt}[1]{\label{thm:#1}}
\newcommand{\reft}[1]{Theorem~\ref{thm:#1}}

\newcommand{\refs}[1]{Section~\ref{section:#1}}
\newcommand{\labs}[1]{\label{section:#1}}

\newcommand{\labq}[1]{\label{question:#1}}
\newcommand{\refq}[1]{Question~\ref{question:#1}}

\newcommand{\N}{\mathbb{N}}
\newcommand{\Q}{\mathbb{Q}}
\newcommand{\Z}{\mathbb{Z}}

\newcommand{\NQ}{\mathcal{N}(Q)}

\newcommand{\DNQ}{\mathcal{DN}(Q)}
 
\newcommand{\RNQ}{\mathcal{RN}(Q)}

\newcommand{\bp}{\boldsymbol{\Pi}}
\newcommand{\bP}{\boldsymbol{\Pi}}
\newcommand{\bS}{\boldsymbol{\Sigma}}

\newcommand{\pzt}{{\bP_3^0}}

\newcommand{\dimh}[1]{\hbox{$\dim_{\hbox{H}}$}\left( #1\right)}

\title{The Theory of Normality for Dynamically Generated Cantor Series Expansions}
\author{Sohail Farhangi\footnote{Department of Mathematics and Informatics, University of Adam Mickiewicz, Uniwersytetu Poznańskiego 4 Street, 61-614 Poznań, Poland. Email: sohail.farhangi@gmail.com}\ \footnote{Beijing Institute of Mathematical Sciences and Applications, Beijing, P.R.C. 101408.}\ , Bill Mance\footnote{Department of Mathematics and Informatics, University of Adam Mickiewicz, Uniwersytetu Poznańskiego 4 Street, 61-614 Poznań, Poland. Email: william.mance@amu.edu.pl}}

\begin{document}
\maketitle

\begin{abstract}
    The theory of normality for base $g$ expansions of real numbers in $[0,1)$ is rich and well developed.
    Similar theories have been developed for many other numeration systems, such as the regular continued fraction expansion, $\beta$-expansions, and L\"uroth series expansions.
    
    Let $Q=(q_n)_{n \in \mathbb{N}}$ be a sequence of integers greater than or equal to 2. The \textbf{$Q$-Cantor series expansion} of $x \in [0,1)$ is the unique sum of the form $x=\sum_{n=1}^\infty \frac{x_n}{q_1q_2\cdots q_n}$, where $x_n \neq q_n-1$ infinitely often.
    For the Cantor series expansions, most of the literature thus far considers $Q$ where the theory of normality differs drastically from that of the base $g$ expansions. 
    We introduce the class of \textbf{dynamically generated Cantor series expansions}, which is a large class of Cantor series expansions for which much of the classical theory of base $g$ expansions can be developed in parallel. This class includes many examples such as the Thue-Morse sequence on $\{2,3\}$ and translated Champernowne numbers.
    
    A special case of our main results is that if $Q$ is a bounded basic sequence that is dynamically generated by an ergodic system having zero entropy, then normality base $Q$ coincides with distribution normality base $Q$, and $Q$ possesses a Hot Spot Theorem.
\end{abstract}
\section{Introduction}
\subsection{Normal numbers}

  A real number $x \in [0,1)$  is \textbf{normal  in base $g$} if for all natural numbers $k$, all blocks of digits of length $k$ in base $g$
  occur with relative frequency $g^{-k}$ in the $g$-ary expansion of $x$.\footnote{Every result in this paper is true and can be stated if we consider real numbers instead of members of $[0,1)$.}
  We denote this set by $\mathcal{N}_g$.
  There is an enormous literature on normal numbers and the interested reader is referred to the book of Bugeaud \cite{BugeaudBook}. 
  Furthermore, many of the concepts and theorems describing normal numbers have been fully or partially extended to normal numbers in other numeration systems such as continued fraction expansions, $\beta$-expansions, L\"uroth series expansions, and so on.

We wish to mention one of the most fundamental and important results relating to normal numbers in base $g$.
The following is due to D. D. Wall in his Ph.D. dissertation \cite{Wall}.

\begin{theorem}[D. D. Wall]\labt{wall}
A real number $x$ is normal in base $g$ if and only if the sequence $(g^nx)_{n = 1}^\infty$ is uniformly distributed mod $1$.
\end{theorem}

While it is not difficult to prove \reft{wall}, its importance in the theory of normal numbers can not be overstated.
Large portions of the theory of normal numbers in base $g$ make use of \reft{wall}. 
For example, it is easy to show that rational addition preserves normality in base $g$ by invoking \reft{wall}. That is, $q+\mathcal{N}(g) \subseteq \mathcal{N}(g)$ for every rational $q$. 
We remark that the set of all real numbers that preserve normality was fully characterized by Rauzy \cite{NormalityPreservationByAddition}. 
See also Bergelson and Downarowicz \cite{BergDownNormalityPreservation}.
Another example of a classical result in the theory of normal numbers that makes use of \reft{wall} is the Kamae-Weiss selection rule \cite{SelectionRules1,SelectionRules2}.

In addition to \reft{wall}, one can consider variations of the (Classical) Hot Spot Theorem to be as fundamental as \reft{wall}. 
Originally due to Pyatetskii-Shapiro \cite{OriginalHotSpot}, the Classical Hot Spot Theorem  simplifies the task of determining whether or not a given $x \in [0,1)$ is normal base-$g$ and is considered an essential tool for both proving theorems about normal numbers and providing constructions. 
We give two equivalent formulations.
\begin{theorem}\labt{HotSpot1}
    Given $g \in \mathbb{N}_{\ge 2}$ and $y \in [0,1)$, $y$ is normal base-$g$ if there exists a constant $C \ge 1$ such that for any $0 \le a < b \le 1$ we have

    \begin{equation}
        \limsup_{N\rightarrow\infty}\frac{1}{N}\left|\left\{1 \le n \le N\ |\ \{g^ny\} \in (a,b)\right\}\right| \le C(b-a),
    \end{equation}.
\end{theorem}

\begin{theorem}\labt{HotSpot2}
    Let $y \in [0,1)$ have base-$g$ expansion $y = 0.y_1y_2\cdots y_n\cdots$.
    If there exists $C \ge 1$ such that for every $w = (w_1,\cdots,w_\ell) \in \{0,1,\cdots,g-1\}^\ell$ we have

    \begin{equation}
        \limsup_{N\rightarrow\infty}\frac{1}{N}|\{1 \le n \le N\ |\ w = (y_n,y_{n+1},\cdots,y_{n+\ell-1})\}| \le \frac{C}{g^\ell},
    \end{equation}
    then $x$ is normal base-$g$.
\end{theorem}

In this paper, we are interested in a class of expansions known as the $Q$-Cantor series expansions
that includes the $g$-ary expansions as a special case, but do not always admit an extension of \reft{wall}, \reft{HotSpot1}, or \reft{HotSpot2} .
The study of normal numbers and other statistical properties of real numbers with respect to
large classes of Cantor series expansions was  first done by P.\ Erd\H{o}s
and A.\ R\'{e}nyi in \cite{ErdosRenyiConvergent} and \cite{ErdosRenyiFurther} and by
A.\ R\'{e}nyi in \cite{RenyiProbability}, \cite{Renyi}, and \cite{RenyiSurvey} and by P.\ Tur\'{a}n in \cite{Turan}.

The $Q$-Cantor series expansions, first studied by G.\ Cantor in \cite{Cantor},
are a natural generalization of the $b$-ary expansions.
G.\ Cantor's motivation to study the Cantor series expansions was to extend the
well-known proof of the irrationality of the number $e=\sum 1/n!$ to a larger class of numbers.
Results along these lines may be found in the monograph of J.\ Galambos \cite{Galambos}. 

We begin with definitions and notation that we will be using. A sequence $Q = (q_n)_{n = 1}^\infty \in \mathbb{N}_{\ge 2}^\mathbb{N}$ is called a \textbf{basic sequence}. We let $\tilde{Q} = \prod_{n = 1}^\infty[0,q_n-1]$ denote the space of $Q$-Cantor series of the elements of $[0,1)$ under the correspondence given by

\begin{equation}
    x = \sum_{n = 1}^\infty\frac{x_n}{q_1q_2\cdots q_n},
\end{equation}
where we assume without loss of generality that $x_n \neq q_n-1$ infinitely often. 
We may abbreviate this correspondence by writing $x = 0.x_1x_2\cdots x_n\cdots_Q$, and we let $\phi_Q:\tilde{Q}\rightarrow[0,1)$ be the map given by $\phi_Q(x_n)_{n = 1}^\infty = x$. 
We say that $(x_n)_{n = 1}^\infty \in \tilde{Q}$ is \textbf{$Q$-distribution normal} if the sequence $(q_nq_{n-1}\cdots q_1\phi_Q(x_m)_{m = 1}^\infty)_{n = 1}^\infty$ is uniformly distributed mod $1$, and we let $\mathcal{DN}(Q)$ denote the collection of $Q$-distribution normal sequences. 
The sequence $(x_n)_{n = 1}^\infty$ is \textbf{$Q$-normal} if all blocks of digits occur with the expected frequency, which we will now define precisely.
For a block $D = (d_1,d_2,\cdots,d_\ell)$ of potential digits and a $j \in \mathbb{N}$, define

\begin{equation}
    \mathcal{I}_{Q,j}(D) = \begin{cases}
                            1 &\text{if }d_1 < q_j, d_2 < q_{j+1},\cdots,d_\ell < q_{j+\ell-1}\\
                            0 &\text{otherwise}
                           \end{cases}\text{, and }Q_n(D) = \sum_{j = 1}^n\frac{\mathcal{I}_{Q,j}(D)}{q_jq_{j+1}\cdots q_{j+\ell-1}}.
\end{equation}
Intuitively speaking, $\mathcal{I}_{Q,j}(D)$ determines whether or not is is possible to see $D$ as a block of digits starting at position $j$, and $Q_n(D)$ calculates the expected number of occurrences of $D$ in the first $n+\ell-1$ digits in the base $Q$ expansion of a typical $x \in [0,1)$.
Letting $x = \phi_Q(x_n)_{n = 1}^\infty$, we define

\begin{equation}
    N_n^Q(D,x) = \#\{i \le n:\ d_j = x_{i+j},\ \forall\ 1 \le j \le \ell\}.
\end{equation}
The sequence $(x_n)_{n = 1}^\infty$ is \textbf{$Q$-normal} if for all blocks $D$ for which $\displaystyle\lim_{n\rightarrow\infty}Q_n(D) = \infty$ we have

\begin{equation}
    \lim_{n\rightarrow\infty}\frac{N_n^Q(D,x)}{Q_n(D)} = 1,
\end{equation}
and we denote the collection of $Q$-normal sequences by $\mathcal{N}(Q)$. Similarly, $(x_n)_{n = 1}^\infty$ is \textbf{$Q$-ratio normal}  (here we write $x \in \RNQ$)
if for all blocks $D_1$ and $D_2$ of equal length  such that 
\\ $\displaystyle\lim_{n \to \infty} \min(Q_n(D_1),Q_n(D_2)) = \infty$, we have
\begin{equation}\labeq{RNQdef}
\lim_{n \to \infty} \frac {N_n^Q (B_1,x)/Q_n(B_1)} {N_n^Q (B_2,x)/Q_n(B_2)}=1.
\end{equation}

When $Q = (g)_{n = 1}^\infty$, we recover the classical notion of a base-$g$ normal number, in which case $\mathcal{N}(Q) = \RNQ = \mathcal{DN}(Q)$. 
This is just a restatement of \reft{wall}.
However, for a general basic sequence $Q$ we do not need to have $\mathcal{N}(Q) = \mathcal{DN}(Q)$ or $\NQ = \RNQ$. 
Interestingly, if $Q$ is a periodic basic sequence, then we once again have that $\NQ=\RNQ=\DNQ$ as shown by Airey and Mance \cite{NormalEquivalencesForEventuallyPeriodicSequences}.
In particular if $q_n=2$ for $n$ odd and $q_n=3$ for $n$ even, then $\NQ=\RNQ=\DNQ=\mathcal{N}_6$.


A significant amount of the theory of normality in base-$g$ follows from Wall's equivalence and versions of the hotspot lemma.
Moreover, much of the standard theory of normal numbers appears to badly fall apart when $\mathcal{N}(Q) \neq \mathcal{DN}(Q)$ \cite{AireyManceNormalOrders}.

\begin{equation}
    \begin{tikzcd}
 & \mathcal{N}\cap\mathcal{DN} \arrow[dr] \arrow[dl] & &\\%
\mathcal{N} \arrow[dr] &  & \mathcal{DN}\cap\mathcal{RN} \arrow[dr] \arrow[dl] &\\
& \mathcal{RN} & & \mathcal{DN}
\end{tikzcd}
\end{equation}

An important case where this happens is where $q_n \to \infty$ and $\sum \frac{1}{q_n}=\infty$ is discussed in \cite{AireyJacksonManceComplexityCantorSeries,AireyManceHDDifference, AireyManceVandehey, AireyManceNormalPreserves, AireyManceNormalOrders} and other papers.
The above diagram shows the relationships between these notions of normality for this case, where an arrow denotes inclusion. 
All non-empty difference  sets are known to be $D_2(\bP^0_3)$-complete \cite{ AireyJacksonManceComplexityCantorSeries}.
Informally, this may be thought of as these notions having maximal logical independence. See \cite{Kechris} for definitions involving the Borel hierarchy.
Furthermore, all non-empty difference sets are known to have full Hausdorff dimension except for $\NQ \backslash \DNQ$ whose Hausdorff dimension has only been shown to be $1$ for a subset of this class of basic sequences \cite{AireyManceHDDifference}.

 Thus, as part of determining what is the ``right'' generality to extend results about base-$g$ normal numbers, we would like to know for which basic sequences $Q$ we have $\mathcal{DN}(Q) = \mathcal{N}(Q)$, and that $Q$ admits something like a Hot Spot Theorem.
 A natural candidate for consideration is the class of dynamically generated $Q$, which we will define precisely in \refs{MainResults}.

It is possible that there will always be sporadic cases of $Q$ where $\NQ=\DNQ$ that don't appear to fall under any particular umbrella. Currently, we don't have a good understanding of which $Q$ may satisfy this property. Thus, we ask the following question.

\begin{question}\labq{Borel}
Consider the subset of the Baire space
$$
S=\{Q \in \mathbb{N}_{\ge 2}^\mathbb{N} : \NQ=\DNQ\}.
$$
Clearly, $S$ is co-analytic. Is $S$ a Borel set?
\end{question}

If the answer to \refq{Borel} is yes, then it makes sense to search for a general Borel condition on $Q$ that suggests where to extend familiar results on normality. If the answer is no, then it is not clear which basic sequences $Q$ behave similar to a base $g$ expansion when considering normality. Thus, one may only look for more special cases where this happens.

Furthermore, we can also ask similar questions to \refq{Borel} when discussing extensions of the Hot Spot Theorems. There are several different levels of these questions, depending on the strength of the Hot Spot Theorem in question. 


\subsection{Dynamically generated basic sequences}\labs{DefinitionOfDGBSSubsection}

We must define a few notions to state our main results. Examples demonstrating our main theorems will be given in \refs{Examples}. Moreover, we will provide counterexamples in \refs{Counterexamples} that demonstrate the necessity of many of our conditions.

Every notion introduced in this section will be explored in greater depth later in the paper. We first introduce the idea of a dynamically generated basic sequence. We also remark that a similar notion of random Cantor series was studied by Kifer \cite{Kifer} who studied fractal properties relating to randomly generated basic sequences $Q$. In the case of a randomly generated basic sequence, one may prove results about typical $Q$, but we wish to provide results for specific $Q$. 

We will always use $X$ to denote a complete separable metric space, $\mathscr{B}$ the Borel $\sigma$-algebra on $X$, $\mu$ a probability measure on $(X,\mathscr{B})$ that assigns positive measure to all non-empty open sets, and $T:X\rightarrow X$ a continuous map that preserves $\mu$. We call $\mathcal{X} := (X,\mathscr{B},\mu,T)$ a \textbf{continuous measure preserving system (c.m.p.s.)}. More generally, we let $(Y,\mathscr{A},\nu)$ denote a standard probability space and $S:Y\rightarrow Y$ a measurable map that preserves $\nu$, and we call $\mathcal{Y} := (Y,\mathscr{A},\nu,S)$ a \textbf{measure preserving system (m.p.s.)}. If $f:Y\rightarrow\mathbb{N}_{\ge 2}$ is measurable, then $y \in Y$ is a \textbf{generic point for $f$} if for all $w = (w_1,w_2,\cdots,w_\ell) \in \mathbb{N}_{\ge 2}^\ell$ and

\begin{equation}
    E_w := \bigcap_{i = 1}^\ell T^{-i}f^{-1}(w_i)\text{, we have }\lim_{N\rightarrow\infty}\frac{1}{N}\sum_{n = 1}^N\mathbbm{1}_{E_w}(T^ny) = \mu(E_w).
\end{equation}
We will often consider a block of bases $B = (b_1,\cdots,b_\ell) \in \mathbb{N}_{\ge 2}^\ell$, so we may also use the notation $E_B$ instead of $E_w$ in such situations. 
A point $x \in X$ is a \textbf{generic point} if it is a generic point for all $f \in C(X)\cap L^1(X,\mu)$, i.e., if for all such $f$ (taking values in $\mathbb{C}$, not just $\mathbb{N}_{\ge 2}$) we have

\begin{equation}
    \lim_{N\rightarrow\infty}\frac{1}{N}\sum_{n = 1}^Nf(T^nx) = \int_Xfd\mu.
\end{equation}
Given a m.p.s. $(Y,\mathscr{A},\nu,S)$ and some measurable $f:Y\rightarrow \mathbb{N}_{\ge 2}$, the $\sigma$-algebra generated by $f$ is the smallest $S$-invariant $\sigma$-algebra $\mathscr{A}_f \subseteq \mathscr{A}$ with respect to which $f$ is measurable.
We will also be assuming throughout this paper that if $E_w \neq \emptyset$, then $\nu(E_w) > 0$.
If $\mathcal{X}$ is a c.m.p.s. and $f \in C(X)$, then the topology generated by $f$ is the smallest $T$-invariant collection of open sets with respect to which $f$ is still continuous.
A basic sequence $Q = (q_n)_{n = 1}^\infty$ is \textbf{dynamically generated} if there exists a c.m.p.s. $(X,\mathscr{B},\mu,T)$, a continuous function $f:X\rightarrow\mathbb{N}_{\ge 2}$ that generates the topology of $X$, and a generic point $x \in X$ for which $q_n = f\left(T^nx\right)$, and we say that $Q$ is generated by $(X,\mathscr{B},\mu,T,f,x)$.

In Lemma \ref{WhenIsASequenceDynamicallyGenerated} we will give an alternative characterization of which basic sequences $Q$ are dynamically generated.
We see that $(q_1,\cdots,q_\ell) = (w_1,\cdots,w_\ell)$ if and only if $x \in E_w$.
It is worth observing that for $g \in \mathbb{N}_{\ge 2}$, the classical theory of base $g$ normality corresponds to the theory of the basic sequence $Q$ that is dynamically generated by the trivial one point system and the function $f \equiv g$.

We will now discuss two natural strengthenings of the notions of $Q$-normality and $Q$-distribution normality for dynamically generated $Q$. 
Let us fix a block $D = (d_1,\cdots,d_\ell)$ of potential digits and a block $B = (b_1,\cdots,b_\ell)$ of potential bases. Let $S_B = \{j \in \mathbb{N}\ |\ (q_j,q_{j+1},\cdots,q_{j+\ell-1}) = (b_1,b_2,\cdots,b_\ell)\}$, and define

\begin{alignat*}{2}
    &Q_n(D,B) = \sum_{j \in S_B\cap[1,n]}\frac{\mathcal{I}_{Q,j}(D)}{b_1b_2\cdots b_K}\text{, and for }z = 0.z_1z_2\cdots z_n\cdots_Q\text{ define}\\
    &N_n^Q(D,B,z) = \#\{j \in S_B\cap[1,n]:\ d_i = z_{j+i},\ \forall\ 0 \le i < k\}.
\end{alignat*}
Intuitively speaking, $Q_n(D,B)$ is the expected number of occurrences of the block of digits $D$ paired with the block of bases $B$ in the first $n+K-1$ digits of the base $Q$ expansion of a generic $x \in [0,1)$.
Similarly, $N_n^Q(D,B,z)$ counts the number of occurrences of the block of digits $D$ paired with the block of bases $B$ in the base $Q$ expansion of $z$.
We say that $z \in [0,1)$ is \textbf{uniformly $Q$-normal} if for all blocks $D$ and $B$ with $\displaystyle\lim_{n\rightarrow\infty}Q_n(D,B) = \infty$, we have

\begin{equation}
    \lim_{n\rightarrow\infty}\frac{N_n^Q(D,B,z)}{Q_n(D,B)} = 1,
\end{equation}
and we denote the collection of such sequences by $\mathcal{UN}(Q)$.
For $D = (d_1,\cdots,d_\ell)$ and $B = (b_1,\cdots,b_\ell)$, we write $D < B$ if $d_i < b_i$ for all $1 \le i \le \ell$.
It is worth noting that when $Q$ is dynamically generated, if $D$ and $B$ are blocks for which $D < B$ and $B$ appears at least once in $Q$, then $\displaystyle\lim_{n\rightarrow\infty}Q_n(D,B) = \infty$.
Now suppose that $Q$ is generated by $(X,\mathscr{B},\mu,T,f,x)$. 

A sequence $(x_n)_{n = 1}^\infty \subseteq X\times[0,1)$ is \textbf{uniformly distributed (with respect to $\mu$)}\footnote{We remark that our definition of uniform distribution is motivated by that of \cite[Definition 3.1.1]{Kuipers&Niederreiter}.}  if for any $F \in C(X\times[0,1))\cap L^1(X\times[0,1),\mu\times m)$ we have

\begin{equation}
    \lim_{N\rightarrow\infty}\frac{1}{N}\sum_{n = 1}^NF(x_n) = \int_{X\times[0,1)}Fd\mu\times m.
\end{equation}
We say that $z \in [0,1)$ is \textbf{uniformly $Q$-distribution normal} if $(x,z)$ has a uniformly distributed orbit in the skew product system $\mathcal{X}^f := (X\times[0,1), \mathscr{B}\times\mathscr{L},\mu\times m,T\rtimes M)$, where 
$M_nz = nz\pmod{1}$, $T\rtimes M(x,z) = (Tx,M_{f(x)}z)$, $\mathscr{L}$ is the Lebesgue $\sigma$-algebra on $[0,1)$ and $m$ is the Lebesgue measure.
We denote the collection of such sequences by $\mathcal{UDN}(Q)$. 
It is worth noting that a priori, the definition of $\mathcal{UDN}(Q)$ depends on the generating tuple $(X,\mathscr{B},\mu,T,f,x)$. 
However, we will prove in Theorem \ref{UN=UDN} that $\mathcal{UN}(Q) = \mathcal{UDN}(Q)$, so a posteriori we see that the definition of $\mathcal{UDN}(Q)$ is independent of the generating tuple.

Here, we will see that if $Q$ is dynamically generated, then there is at least some connection between blocks of digits and distributional properties.

\begin{theorem}\label{UN=UDN}
If $Q = (q_n)_{n = 1}^\infty$ is a dynamically generated basic sequence, then $\mathcal{UN}(Q) = \mathcal{UDN}(Q)$.
\end{theorem}

\subsection{Main Results}\labs{MainResults}

Returning to our original goal of finding $Q$ for which $\mathcal{N}(Q) = \mathcal{DN}(Q)$, we arrive at our first main theorem.

\begin{theorem}\label{MainCorollary}
    If $Q$ is a basic sequence generated by $(X,\mathscr{B},\mu,T,f,x)$ with $(X,\mathscr{B},\allowbreak\mu,\allowbreak T)$ ergodic and having zero entropy, and $\int_X\log(f)d\mu < \infty$, then $\mathcal{N}(Q) = \mathcal{UN}(Q) = \mathcal{UDN}(Q) = \mathcal{DN}(Q)$.
\end{theorem}

Since periodic basic sequences are generated by rotation on a finite set, we see that Theorem \ref{MainCorollary} gives a much larger class of $Q$ for which $\mathcal{N}(Q) = \mathcal{DN}(Q)$ than the periodic $Q$ given by \cite{PeriodicBasicSequences}.

Furthermore, in \refs{RatioNormality} we will connect ratio normality and a new notion of uniform ratio normality (denoted $\mathcal{URN}(Q)$) to arrive at the following diagram describing the implications for general dynamically generated basic sequences $Q$.

\begin{equation}
    \begin{tikzcd}
\mathcal{UN}  \arrow[r, leftrightarrow] \arrow[d] & \mathcal{URN} \arrow[r, leftrightarrow] \arrow[d] & \mathcal{UDN} \arrow[d]\\%
\mathcal{N} & \mathcal{RN} \arrow[l, leftrightarrow] & \mathcal{DN}
\end{tikzcd}
\end{equation}

For $g \in \mathbb{N}_{\ge 2}$ we let  $M_g:[0,1)\rightarrow[0,1)$ is given by $M_g(y) = gy\pmod{1}$. 
A sequence $(q_n)_{n = 1}^\infty \subseteq \mathbb{N}_{\ge 2}^\mathbb{N}$ is a \textbf{$g$-power sequence} if $q_n = g^{a_n}$ with $a_n \in \mathbb{N}$. Lafer \cite{Lafer} first considered $g$-power sequences in his Ph.D. dissertation. 
We will explore $g$-power sequences in greater detail in \refs{gpower} where we will prove the following result.

\begin{theorem}\label{MainTheoremForPowersOf2}
    Let $g \in \mathbb{N}_{\ge 2}$ and let $Q = (q_n)_{n = 1}^\infty \in \left(\{g^n\}_{n = 1}^\infty\right)^{\mathbb{N}}$ be a basic sequence generated by $(X,\mathscr{B},\mu,T,f,x)$ with $(X,\mathscr{B},\mu,T)$ ergodic and having zero entropy, and $\int_X\log(f)d\mu < \infty$.
    We have $\mathcal{N}(Q) = \mathcal{DN}(Q) = \mathcal{UN}(Q) = \mathcal{UDN}(Q) = \mathcal{N}_g$.
\end{theorem} 


We now turn to the Hot Spot Theorems.
The classic Hot Spot Theorem was improved by Postnikov \cite{AStrongerHotSpot}, Bailey and Misiurewicz \cite{AStrongHotSpot}, once again by Pyatetskii-Shapiro \cite{StrongestHotSpotTheorem}, and then by Bergelson and Vandehey \cite{HotSpotProofOfGeneralizedWallTheorem}.
See Moshchevitin and Shkredov \cite{MoshchevitinShkredov} as well as Airey and Mance \cite{MoshchevitinShkredovCorrection} for Hot Spot Theorems for other numeration systems.
In order to prove Theorem \ref{MainCorollary}, we will need a variant of the strongest Hot Spot Theorem, so we state it below. 

\begin{theorem}[{Bergelson-Vandehey \cite[Theorem 6]{HotSpotProofOfGeneralizedWallTheorem}}]\label{StrongestHotSpot}
    Let $y \in [0,1)$ have base-$g$
expansion $0.y_1y_2y_3\cdots$. Suppose that for any $\sigma > 0$, there exists arbitrarily large $k$, subsets $\mathcal{N}_{\sigma,k} \subseteq \mathbb{N}$ of natural upper density at most $\sigma$, and a constant $C > 0$ (not dependent on $\sigma$ or $k$) such that the following holds: for every word $s = [d_1,d_2,\cdots,d_k]$ of length $k$ we have

\begin{equation}
    \limsup_{N\rightarrow\infty}\frac{\nu_s\left(y,N;\mathcal{N}_{\sigma,k}\right)}{N} \le \frac{Cb^{\sigma k}}{g^k},
\end{equation}
where $\nu_s(y,N;\mathcal{N}_{\sigma,k}) = \#\{i \in [1,N]\setminus\mathcal{N}_{\sigma,k}\ :\ y_{i+j} = d_j,\ \forall\ 1 \le j \le k\}$. Then $y$ is normal in base-$g$.
\end{theorem}

For any dynamically generated basic sequence $Q$ and any block of digits $D\in \mathbb{N}_0^\ell$, let 
    \begin{equation}\labeq{PD}
        P_D := \lim_{n\rightarrow\infty}\frac{Q_n(D)}{n}.
    \end{equation}
It will be shown in Lemma~\ref{LimitsExistForDGBS} that $P_D$ exists.
The following is the second main result of the paper (proven in Section~\ref{DynamicallyGeneratedBasicSequencesSection}) and provides us with a Hot Spot Theorem for some deterministic basic sequences.

\begin{theorem}[Hot Spot]\label{StrongestHotSpotForDDGBasicSequences}
    Let $Q = (q_n)_{n = 1}^\infty$ be a basic sequence generated by $(X,\mathscr{B},\mu,T,f,x)$ with $(X,\mathscr{B},\mu,T)$ having zero entropy and $\int_X\log(f)d\mu < \infty$.
    \begin{enumerate}[(i)]
        \item Let $y \in [0,1)$ be such that for any $\sigma \in (0,1)$, there exists a $\delta > 0$ and a subset $\mathcal{N}_{\sigma} \subseteq \mathbb{N}$ of natural upper density at most $1-\sigma$, and a constant $C > 0$ (not dependent on $\sigma$) such that the following holds: for every $0 \le a < b \le \text{min}(a+\delta,1)$ we have

\begin{equation}\label{GeneralizedDDGHotspotAssumption}
    \limsup_{N\rightarrow\infty}\frac{\nu_{(a,b)}\left(y,N,Q;\mathcal{N}_{\sigma}\right)}{N} < C(b-a)^\sigma,
\end{equation}
where $\nu_{(a,b)}\left(y,N,(a_n)_{n = 1}^\infty;\mathcal{N}_\sigma\right) = \nu_{(a,b)}\left(y,N,(a_n)_{n = 1}^N;\mathcal{N}_\sigma\right) = \#\{n \in [0,N-1]\setminus\mathcal{N}_{\sigma}\ :\ \prod_{i = 1}^na_iy\allowbreak \in (a,b)\}$. Then $y \in \mathcal{UDN}(Q)$.

\item Suppose that $(X,\mathscr{B},\mu,T)$ is ergodic. For $A \subseteq \mathbb{N}$, a block of digits $D \in \mathbb{N}_0^\ell$, and a $y = 0.y_1y_2\cdots y_i\cdots_Q \in [0,1)$, let 

    \begin{equation}
        N_n^Q(D,y;A) = \#\{i \in [1,n]\setminus A\ |\ y_{i+j} = d_j\ \forall\ 1\le j \le \ell\}.
    \end{equation}
    Fix $y \in [0,1)$ and $C > 0$. If for every $\sigma \in (0,1)$ there exists $\ell_\sigma \in \mathbb{N}$ and $\mathcal{N}_\sigma \subseteq \mathbb{N}$ with upper density at most $1-\sigma$ such that for every $\ell \ge \ell_\sigma$ and every $D \in \mathbb{N}_0^\ell$  we have

    \begin{equation}\label{GeneralizedDDGHotspotAssumption2}
        \limsup_{n\rightarrow\infty}\frac{N_n^Q(D,y;\mathcal{N}_\sigma)}{n} \le CP_D,
    \end{equation}
    then $y \in \mathcal{UN}(Q)$.
    \end{enumerate}
\end{theorem}

Recalling \refq{Borel}, we would like to have version of Theorem~\ref{MainCorollary} that fully characterizes those $Q$ for which $\NQ=\DNQ$ with a Borel condition. Similarly, we would like an improvement of Theorem~\ref{StrongestHotSpotForDDGBasicSequences} that fully characterizes all $Q$ which admit a Hot Spot Theorem.
We note that the set of $Q$ that are described in Theorem~\ref{MainCorollary} and Theorem~\ref{StrongestHotSpotForDDGBasicSequences} is a Borel set. 
The set of $Q$ satisfying conditions (i) and (iii) of Lemma~\ref{WhenIsASequenceDynamicallyGenerated} is a $\bP_3^0$-complete set by Theorem~6 in the paper of Airey, Jackson, Kwietniak, and Mance \cite{AireyJacksonKwietniakManceSpecificationComplexity}.
While it is not explicitly stated, the reduction given to prove this theorem shows that the set of all $Q$ satisfying every condition of Lemma~\ref{WhenIsASequenceDynamicallyGenerated} is a $\bP_3^0$-complete set. 
So the set of dynamically generated basic sequences is a $\bP_3^0$-complete set.

Further, we note that the set of dynamically generated $Q$ with entropy zero is trivially a $\bp_3^0$ set although it is likely that this set can be proven to be $\bP_3^0$-complete by combining the techniques in \cite{AireyJacksonKwietniakManceSpecificationComplexity} and \cite{AireyJacksonManceNoise}.
The condition that $\int_X\log(f)d\mu < \infty$ is a boundedness condition and thus only $\bS_2^0$.
It can be shown that the ergodicity condition is $\bP_4^0$.
Thus, the set of $Q$ that satisfy the assumptions of Theorem~\ref{MainCorollary} and Theorem~\ref{StrongestHotSpotForDDGBasicSequences} is a Borel set and is $\bP_4^0$ although it is not clear if it is $\bP_4^0$-complete. 
If the assumptions about ergodicity and $\int_X\log(f)d\mu < \infty$ are removed, then we would only have to deal with a $\bP_3^0$-complete condition in Theorem~\ref{MainCorollary}.

Lastly, we provide several questions and conjectures in \refs{Conjectures}

\textbf{Acknowledgements:} Both authors acknowledge being supported by grant
2019/34/E/ST1/00082 for the project “Set theoretic methods in dynamics and number theory,” NCN (The
National Science Centre of Poland).
We would like to thank Yuval Peres for the proof of Lemma \ref{LemmaUsingMcdiarmid}, and Tomasz Downarowicz for a discussion that led to the proof of Theorem \ref{NormalityPreservationUnderRationalMultiplication}.
We would also like to thank Jean-Paul Thouvenot and Fran\c{c}ois Ledrappier for helpful discussions regarding entropy theory.
Furthermore, we would like to thank Shigeki Akiyama for helpful discussions regarding Example \ref{ExampleWithSubstitutions}.\\ \

\section{Examples}\labs{Examples}
In this section we will give various concrete examples of dynamically generated basic sequences to which the theory developed thusfar applies.
In light of Theorem \ref{StrongestHotSpotForDDGBasicSequences} and Theorem~\ref{MainCorollary}, we will begin with 4 examples of $Q = (q_n)_{n = 1}^\infty$ that are generated by an ergodic zero entropy system $\mathcal{X} := (X,\mathscr{B},\mu,T)$, a function $f:X\rightarrow\mathbb{N}_{\ge 2}$ satisfying $\int_X\log(f)d\mu < \infty$, and a generic point $x \in X$.
We will also consider two examples in which $\mathcal{X}$ is ergodic and has positive entropy, and one example in which $\mathcal{X}$ is not ergodic.
The fact that our examples are indeed dynamically generated basic sequences can be checked directly from the definition for the examples that are constructed on a symbolic dynamical system, and using Corollary \ref{DGBSOnTori} for the rest.
For all of the following examples, we use $a,b,c,$ and $d$ to denote distinct elements of $\mathbb{N}_{\ge 2}$.

\begin{example}[Periodic sequences]
    If there exists $m \in \mathbb{N}$ for which $q_n = q_{n+m}$ for all $n \in \mathbb{N}$, then $Q$ is dynamically generated by rotation on a finite set. 
    It is clear that $Q$ is bounded and deterministic in this case.
    Let $g = q_1q_2\cdots q_{m-1}$.
    The fact that $\mathcal{DN}(Q) = \mathbb{N}(Q) = \mathcal{N}_g$ was first observed in \cite{PeriodicBasicSequences}.
    While $\mathcal{UN}(Q)$ and $\mathcal{DN}(Q)$ were not yet defined in $\cite{PeriodicBasicSequences}$, the fact that $\mathcal{UDN}(Q) = \mathcal{UN}(Q) = \mathcal{N}_g$ can be proven directly in this case, and is arguably easier to prove than the fact that $\mathcal{N}(Q) = \mathcal{DN}(Q)$.
\end{example}

\begin{example}[Almost Periodic sequences]
    Let $\alpha \in \mathbb{R}\setminus\mathbb{Q}$ be arbitrary and consider the system $\mathcal{X} := (\mathbb{T},\mathscr{L},m,T)$ given by $T(x) = x+\alpha$.
    Using the natural identification between $\mathbb{T}$ and $[0,1)$, we consider the function $f$ given by $f(x) = a$ if $x \in [0,0.5)$ and $f(x) = b$ if $x \in [0.5,1)$.
    Since $\mathcal{X}$ is uniquely ergodic and $f$ is Riemann integrable, we see that for any $x \in [0,1)$ the sequence $Q = (f(T^nx))_{n = 1}^\infty$ is dynamically generated, bounded, and deterministic.
    The determinism of $Q$ follows from the fact that $Q$ is almost periodic.

    This example may be generalized by considering rotations on higher dimensional tori, as this will always produce a uniquely ergodic system on the orbit closure of the origin, and by replacing $f$ with any other (potentially unbounded) Jordan measurable function.
\end{example}

\begin{example}[Nilsequences]
    Rotations on a nilmanifold are a natural class of systems that generalize rotation on a finite dimensional torus. 
    For a detailed discussion about nilmanifolds and their connections to ergodic theory, we refer the reader to \cite{host2018nilpotent}.
    For now we only record some relevant properties.
    
    A rotation on a nilmanifold can always be modeled as a c.m.p.s. $\mathcal{X} := \left(X,\mathscr{B},\mu,T\right)$ that has zero entropy topological (and hence measurable) entropy.
    If $\mathcal{X}$ possesses a transitive point, then it will be uniquely ergodic.
    If $\mathcal{X}$ does not possess a transitive point, then we can take any $x \in X$, and restrict to the orbit closure of $x$, which will be a new nilmanifold in which $x$ is a transitive point, hence the new system will be uniquely ergodic.

    For a concrete example, let $\alpha \in \mathbb{R}\setminus\mathbb{Q}$ be arbitrary and consider the system $\left(\mathbb{T}^2,\mathscr{L}^2,m^2,T\right)$ with $T(x,y) = (x+\alpha,y+2x+\alpha)$.
    The system is uniquely ergodic, hence $(0,0)$ is a generic point.
    We see that $T^n(0,0) = (n\alpha,n^2\alpha)$.
    Using the natural identification between $\mathbb{T}^2$ and $[0,1)^2$, we consider the function $f$ given by $f(x,y) = a$ if $x \ge y$ and $f(x,y) = b$ if $x < y$.
    This produces the dynamically generated basic sequence $Q = (q_n)_{n = 1}^\infty$ given by $q_n = a$ if $\left\{n\alpha\right\} \ge \left\{n^2\alpha\right\}$ and $q_n = b$ else.
    
\end{example}

\begin{example}[Substitution sequences]\label{ExampleWithSubstitutions}
For a detailed discussion about dynamical systems associated with substitutions, we refer the reader to \cite{queffelec1987substitution}.
Here we only give a quick specialized introduction that will allow for the next 3 examples.

Suppose that we have a finite alpha $\mathcal{A} = \{a_1,\cdots,a_k\}$ and a function $\psi:\mathcal{A}\rightarrow\mathcal{A}^*$, where $\mathcal{A}^*$ is the collection of finite words over $\mathcal{A}$.
The function $\psi$ is called a \textbf{substitution}.
Given a word $w \in \mathcal{A}^*$, we define $\psi(w)$ by applying $\psi$ to each constituent letter of $w$ and then concatenating the results together in order, which yields a new word in $\mathcal{A}^*$.
The substitution $\psi$ is \textbf{primitive} if there exists $t \in \mathbb{N}$ such that for each $a \in \mathcal{A}$, the word $\psi^t(a)$ contains every letter of $\mathcal{A}$ at least once.
We say that the substitution $\psi$ has property (H1) if for each $a \in \mathcal{A}$ we have $\lim_{t\rightarrow\infty}|\psi^t(a)| = \infty$, and $\psi$ has property (H2) if $\psi(a_1)$ is a word that start with the letter $a_1$.
If $\psi$ satisfies properties (H1) and (H2), then there exists an infinite word $\psi^\infty(a_1) \in \mathcal{A}^\mathbb{N}$ whose initial words agree with $\psi^t(a_1)$ for every $t \in \mathbb{N}$.
We obtain a topological dynamical system $(X,T)$ by letting $T:\mathcal{A}^\mathbb{N}\rightarrow\mathcal{A}^\mathbb{N}$ be the left shift map and $X$ the orbit closure of $\psi^\infty(a_1)$.
If $\psi$ is primitive, then $(X,T)$ will be minimal, uniquely ergodic, and have zero entropy.

Our next 3 concrete examples will involve a primitive substitution $\psi:\mathcal{A}\rightarrow\mathcal{A}^*$ that satisfies properties (H1) and (H2).
In all 3 examples we consider the function $f:\{a,b\}^\mathbb{N}\rightarrow\{a,b\}$ given by $f((x_n)_{n = 1}^\infty) = x_1$ and its restriction to $X$.
It is clear that $f$ is continuous and that it generates the topology of $X$.
    \begin{enumerate}[(a)]
        \item \textbf{The Fibonacci Sequence}. 
        Let $\mathcal{A} = \{a,b\}$.
        The substitution $\Psi$ given by $\Psi(a) = b$ and $\Psi(b) = ab$ generates the Fibonacci sequence.
        While $\Psi$ is primitive, and satisfies property (H2), it does not satisfy (H1), so we instead consider $\psi = \Psi^2$.
        We see that $\psi(a) = ab$ and $\psi(b) = bab$, so $\psi$ has all of the desired properties.
        The infinite word $\psi^\infty(a)$ is naturally identified with a dynamically generated basic sequence $Q$.
        To see the first 29 letters of $\psi^\infty(a)$, we write $\psi^t(a)$ for $t \in [0,4]$.

        \begin{equation}
            a, ab, abbab, abbabbababbab, abbabbababbabbababbabbababbab
        \end{equation}

        \item \textbf{The Thue-Morse Sequence}. 
        Let $\mathcal{A} = \{a,b\}$.
        The substitution $\psi$ given by $\psi(a) = ab$ and $\psi(b) = ba$ generates the Thue-Morse sequence and is seen to satisfy all of our desired properties.
        As before, we identify the infinite word $\psi^\infty(a)$ with a dynamically generated basic sequence $Q$.
        To see the first 32 letters of $\psi^\infty(a)$, we write $\psi^t(a)$ for $t \in [0,5]$.

        \begin{equation}
            a, ab, abba, abbabaab, abbabaabbaababba, abbabaabbaababbabaababbaabbabaab
        \end{equation}

        \item \textbf{The Rudin-Shapiro Sequence}.
        Let $\mathcal{A} = \{a,b,c,d\}$.
        The substitution $\psi$ given by $\psi(a) = ab$, $\psi(b) = ac$, $\psi(c) = db$, and $\psi(d) = dc$ generates the Rudin-Shapiro sequence and is seen to satisfy all of our desired properties.
        We once again identify the infinite word $\psi^\infty(a)$ with a dynamically generated basic sequence $Q$.
        To see the first 32 letters of $\psi^\infty(a)$, we write $\psi^t(a)$ for $t \in [0,5]$.

        \begin{equation}
            a, ab, abac, abacabdc, abacabdcabacdcdb, abacabdcabacdcdbabacabdbdcdbdcac
        \end{equation}
    \end{enumerate}
\end{example}

\begin{example}(Bernoulli random sequences)
    Fix some (potentially infinite) $\mathcal{A} \subseteq \mathbb{N}_{\ge 2}$ and a probability measure $\mu$ on $\mathcal{A}$ with full support.
    Consider the m.p.s. $\mathcal{X} := \left(\mathcal{A}^\mathbb{N},\mathscr{B},T,\mu^\mathbb{N}\right)$ where $\mathscr{B}$ is the Borel $\sigma$-algebra, and $T$ is the left shift.
    The system $\mathcal{X}$ is a \textbf{Bernoulli shift}, and it is a system with completely positive entropy, hence disjoint from every system with zero entropy.
    We define a \textbf{$(\mathcal{A},\mu)$-normal sequence} to be a generic point for $T$ in $\mathcal{X}$.
    We see that the function $f:\mathcal{A}^\mathbb{N}\rightarrow\mathcal{A}$ given by $f((x_n)_{n = 1}^\infty) = x_1$ is continuous and generates the topology of $\mathcal{A}^\mathbb{N}$, so any $(\mathcal{A},\mu)$-normal sequence is also a dynamically generated sequence.

    When $\mathcal{A} = [2,11]$ and $\mu$ is the uniform probability measure on $\mathcal{A}$, we have the following 3 concrete examples of $(\mathcal{A},\mu)$-normal sequences.
    All three examples are motivated by examples of numbers that are normal base 10.
    However, normality base 10 is a phenomenon that occurs with the digit set $[0,9]$, so we increase all of the digits by 2 in order to allow for examples of basic sequences.
    \begin{enumerate}[(a)]
        \item \textbf{The Champernowne Number.} Champernowne \cite{champernowne1933construction} showed that if you concatenate the base 10 expansions of the natural numbers in order, then the corresponding number $C = 0.c_1\cdots c_n\cdots \in [0,1)$ is normal base $10$.
        Since

        \begin{equation}
            C = 0.1234567891011121314151617181920212223242526272829...,
        \end{equation}
        we see that the corresponding example of a $(\mathcal{A},\mu)$-normal basic sequence is

        \begin{equation*}
            Q = \bigg(3,4,5,6,7,8,9,(10),(11),3,2,3,3,3,4,3,5,3,6,3,7,3,8,3,9,3,(10),3,(11),4,2,4,3,\cdots
        \end{equation*}

        \item \textbf{Davenport-Erd\H{o}s Constructions.} 
        Davenport and Erd\H{o}s \cite{davenport1952note} proved a polynomial version of Champernowne's result.
        A special case of their result, originally due to Besicovitch in \cite{BesicovitchSquares}, is that if we concatenate the base 10 expansions of the squares in order, then the corresponding number $C_2 \in [0,1)$ is normal.
        Since 

        \begin{equation}
            C_2 = 0.149162536496481100121144169196225256...,
        \end{equation}
        we see that the corresponding example of a $(\mathcal{A},\mu)$-normal basic sequence is 

        \begin{equation*}
            Q = \bigg(3,6,(11),3,8,4,7,5,8,6,(11),8,6,(10),3,3,2,2,3,4,3,3,6,6,3,8,(11),3,(11),8,4,4,...
        \end{equation*}

        \item \textbf{The Copeland-Erd\H{o}s Number.}
        Copeland and Erd\H{o}s \cite{copeland1946note} showed, as a special case of a more general result, that if you concatenate the base 10 expansions of the primes in order, then the corresponding number $C_3 \in [0,1)$ is normal.
        Since

        \begin{equation}
            C_3 = 0.2357111317192329313741434753596167717379838997101...,
        \end{equation}
        we see that the corresponding example of a $(\mathcal{A},\mu)$-normal basic sequence is

        \begin{equation*}
            Q = \bigg(4,5,7,9,3,3,3,5,3,9,3,(11),4,5,4,(11),5,3,5,9,6,3,6,5,6,9,7,5,7,(11),8,3,8,9,...
        \end{equation*}
    \end{enumerate}
\end{example}

\begin{example}[The Adler-Keane-Smorodinsky Sequence]
    Let $\mu$ be the probability measure on $((0,1),\mathscr{L})$ given by $d\mu = (\log(2)(1+x))^{-1}dx$ and let $T_G:(0,1)\rightarrow(0,1)$ be given by $T_Gx = \left\{\frac{1}{x}\right\}$.
    The m.p.s. $((0,1),\mathscr{L},\mu,T_G)$ is the dynamical system that is naturally associated to the study of continued fractions, and we mention that this system also has completely positive entropy even though it is not isomorphic to any Bernoulli shift.
    There is a homeomorphism $\phi:(0,1)\setminus\mathbb{Q}\rightarrow\mathbb{N}^\mathbb{N}$ satisfying $T\circ\phi = \phi\circ T_G$, where $T:\mathbb{N}^\mathbb{N}\rightarrow\mathbb{N}^\mathbb{N}$ is the left shift map.
    In particular, $\phi$ is a bijection between the generic points of $T_G$ and those of $T$.
    Since the function $f:\mathbb{N}^\mathbb{N}\rightarrow\mathbb{N}_{\ge 2}$ given by $f\left((x_n)_{n = 1}^\infty\right) = x_1+1$ is continuous and generates the topology of $\mathbb{N}^\mathbb{N}$, we see that continued fraction normal numbers are naturally identified with dynamically generated basic sequences.
    
    One of the earliest examples of a number $C \in [0,1)$ that is continued fraction normal was given by Adler, Keane, and Smorodinsky \cite{adler1981construction}.
    Their example is to first take an ordering of $\mathbb{Q}\cap(0,1)$ by $\frac{1}{2},\frac{1}{3},\frac{2}{3},\frac{1}{4},\frac{2}{4},\frac{3}{4},\frac{1}{5},\frac{2}{5},\frac{3}{5},\frac{4}{5},\cdots$, and then concatenate their continued fraction expansions (with the convention that the last digit is not 1) together.
    We see that the first few digits of the continued fraction expansion of $C$ are given by

    \begin{equation}
        C = [2,3,1,2,1,2,1,3,5,2,2,1,1,2,1,4,\cdots.
    \end{equation}
    The corresponding dynamically generated basic sequence is given by

    \begin{equation}
        Q = (3,4,2,3,2,3,2,4,6,3,3,2,2,3,2,5,\cdots
    \end{equation}
    For other examples of continued fraction normal numbers, see \cite{vandehey2016new}.
\end{example}

\begin{example}[A non-ergodic sequence]\label{NonErgodicExample}
    Given a finite word $w$ and a $n \in \mathbb{N}\cup\{\infty\}$, let $w^n$ denote the concatenation of $n$ copies of $w$.
    Now consider the basic sequence $Q$ corresponding to the infinite word $abcbac(abc)^2(bac)^2(abc)^3(bac)^3\cdots(abc)^n(bac)^n\cdots$.
    It is readily verified that $Q$ satisfies the conditions of Lemma \ref{WhenIsASequenceDynamicallyGenerated}, hence it is a dynamically generated basic sequence. 
    Let $X$ denote the orbit closure of $Q$ in $\{a,b,c\}^\mathbb{N}$ and $T:X\rightarrow X$ the left shift map. 
    For $1 \le i \le 6$, let $x_i \in X$ be given by $x_1 = (abc)^\infty$, $x_2 = (bca)^\infty$, $x_3 = (cab)^\infty$, $x_4 = (bac)^\infty$, $x_5 = (acb)^\infty$, and $x_6 = (cba)^\infty$.
    Let $\mu = \lim_{N\rightarrow\infty}\frac{1}{N}\sum_{n = 1}^N\delta_{T^nQ}$ with convergence taking place in the weak$^*$ topology, and observe that $\mu = \frac{1}{6}\sum_{i = 1}^6\delta_{x_i}$.
    Since $Q$ is generated by $(X,\mathscr{B},\mu,T)$, and $\{x_1,x_2,x_3\}$ is a nontrivial $T$-invariant set, we see that $Q$ is generated by a non-ergodic system.
\end{example}

\begin{remark}\label{RemarkAboutPotentialGeneralization}
    In Example \ref{NonErgodicExample}, it seems like it would be easier to use the sequence $Q'$ corresponding to $aba^2b^2a^3b^3\cdots a^nb^n\cdots$.
    Unfortunately, $Q'$ does not satisfy condition (ii) of Lemma \ref{WhenIsASequenceDynamicallyGenerated} since the words $ab$ and $ba$ appear with 0 density, so it is not dynamically generated.
    Nonetheless, it can be checked that $Q'$ still satisfies versions of Theorems \ref{StrongestHotSpotForDDGBasicSequences}, Theorem~ \ref{MainCorollary}, and Theorems \ref{RN=N} and \ref{URN=UN} when the various definitions of uniform normality are suitably modified.
    The important property that distinguishes $Q'$ from the pathological example given in Remark \ref{RemarkJustifyingFullSupportAssumption} is that for any finite word $w$ that appears in $Q'$ with $d(w) = 0$, there exists a finite word $w'$ that has the same length as $w$ but is lexicographically larger than $w$ and for which $d(w')$ exists and is positive.
    What this means is that the words $w$ that appear with $0$ density cannot perturb the expected frequency of any block of digits.
    Consequently, all of the theory developed in the previous sections could be applied to this more general class of basic sequences that are also generated by reasonable dynamical systems.
    We chose not to pursue this extra level of generality for the sake of presentation, as much of the discussion is greatly simplified in our current set up.
\end{remark}
\section{Counterexamples}\labs{Counterexamples}
In this Section we given various counterexamples to show that many of the assumptions of our main results are in fact necessary.
We observe that for $0 \le a \le b \le 1$ and $c \in [0,1)$, we have $2c\pmod{1} \in (a,b)$ if and only if $c \in \left(\frac{1}{2}a,\frac{1}{2}b\right)\bigcup\left(\frac{1}{2}a+\frac{1}{2},\frac{1}{2}b+\frac{1}{2}\right)$.
Consequently, we will adopt the notation $2^{-1}(a,b) := \left(\frac{1}{2}a,\frac{1}{2}b\right)\bigcup\left(\frac{1}{2}a+\frac{1}{2},\frac{1}{2}b+\frac{1}{2}\right)$.
It is worth mentioning that the first 3 counterexamples in this section involve Bernoulli random sequences as the basic sequence, and the last example in this section involves Bernoulli random sequences as the sequence of digits.

\begin{example}\label{DN(Q)ButNOtN(Q)}
There exists a non-deterministic dynamically generated sequence $Q = (q_n)_{n = 1}^\infty \in \{2,4\}^\mathbb{N}$ for which $(\mathcal{N}_2\cap\mathcal{DN}(Q))\setminus\mathcal{N}(Q) \neq \emptyset$.
\end{example}

\begin{proof}
    Let $y \in [0,1)$ be normal base 4. We will now inductively the basic sequence $Q$. 
    At the first step we define $q_1 = 4$ if $y \in [0,0.5)$ and $q_1 = q_2 = 2$ if $y \in [0.5,1)$. 
    For the $n^\text{th}$ step of the induction, we assume that $q_i$ is defined for $1 \le i \le M = M(n)$, and we define $q_{M+1} = 4$ if $4^{n-1}y \in [0,0.5)$ and $q_{M+1} = q_{M+2} = 2$ if $4^{n-1}y \in [0.5,1)$.
    We see by the construction of $Q$ that the digits $2$ and $3$ never appear in the base $Q$ expansion of $y$, so $y$ is not $Q$-normal. 
    It remains to check that $y$ is $Q$-distribution normal. 
    Letting $M(n)$ be as above, we see that $M(n+1)-M(n) = 1,2$ and $M(n) \approx \frac{3}{2}n$.
    Letting $0 \le a < b \le 1$ be arbitrary, we see that

    \begin{alignat*}{2}
        & \lim_{N\rightarrow\infty}\frac{1}{N}\left|\{1 \le n \le N\ |\ q_n\cdots q_1y \in (a,b)\}\right| = \lim_{N\rightarrow\infty}\frac{1}{M(N)}\left|\{1 \le n \le M(N)\ |\ q_n\cdots q_1y \in (a,b)\}\right|\\
        = & \lim_{N\rightarrow\infty}\frac{1}{M(N)}\left|\left\{1 \le n \le N\ |\ 4^ny \in (a,b)\}\cup\{1 \le n \le N\ |\ 4^ny \in [0.5,1)]\ \&\ 2\cdot4^ny \in (a,b)]\right\}\right|\\
        = & \lim_{N\rightarrow\infty}\frac{2}{3N}\left|\left\{1 \le n \le N\ |\ 4^ny \in (a,b)\}\cup\{1 \le n \le N\ |\ 4^ny \in [0.5,1)\cap2^{-1}(a,b)\right\}\right|\\
        = & \frac{2}{3}\left(m(a,b)+m\left([0.5,1)\cap2^{-1}(a,b)\right)\right) = b-a.
    \end{alignat*}
\end{proof}

\begin{example}\label{QWithoutHotSpot}
    There exists a non-deterministic dynamically generated sequences $Q = (q_n)_{n = 1}^\infty \in \{2,4\}^\mathbb{N}$ and a $y \in \mathcal{N}_2$ for which $\left(\prod_{j = 1}^nq_jy\right)_{n = 1}^\infty$ converges in distribution to the probability measure $\mu$ given by
    
    \begin{equation}
        \mu(a,b) = \begin{cases}
                    \frac{4}{5}(b-a)&\text{if }0 \le a < b \le \frac{1}{2}\\
                    \frac{6}{5}(b-a)&\text{if }\frac{1}{2}\le a < b \le 1.
                   \end{cases}
    \end{equation}
    In particular, this $Q$ does not admit a Hot Spot Theorem.
\end{example}

\begin{proof}
    Let $y \in [0,1)$ be normal base 4. 
    We will now inductively construct a basic sequence $Q := (q_n)_{n = 1}^\infty \in \{2,4\}^{\mathbb{N}}$.
    At the first step we define $q_1 = 4$ if $y \notin [0.25,0.5)$ and $q_1 = q_2 = 2$ if $y \in [0.25,0.5)$. 
    For the $n^\text{th}$ step of the induction, we assume that $q_i$ is defined for $1 \le i \le M = M(n)$, and we define $q_{M+1} = 4$ if $4^{n-1}y \notin [0.25,0.5)$ and $q_{M+1} = q_{M+2} = 2$ if $4^{n-1}y \in [0.25,0.5)$. 
    Letting $M(n)$ be as above, we see that $M(n+1)-M(n) = 1,2$ and $M(n) \approx \frac{5}{4}n$. 
    Letting $0 \le a < b \le 1$ be arbitrary, we see that

    \begin{alignat*}{2}
        & \lim_{N\rightarrow\infty}\frac{1}{N}\left|\{1 \le n \le N\ |\ q_n\cdots q_1y \in (a,b)\}\right| = \lim_{N\rightarrow\infty}\frac{1}{M(N)}\left|\{1 \le n \le M(N)\ |\ q_n\cdots q_1y \in (a,b)\}\right|\\
        = & \lim_{N\rightarrow\infty}\frac{1}{M(N)}\left|\{1 \le n \le N\ |\ 4^ny \in (a,b)\}\cup\{1 \le n \le N\ |\ 4^ny \in [0.25,0.5)\ \&\ 2\cdot4^ny \in (a,b)]\}\right|\\
        = & \lim_{N\rightarrow\infty}\frac{4}{5N}\left|\{1 \le n \le N\ |\ 4^ny \in (a,b)\}\cup\{1 \le n \le N\ |\ 4^ny \in [0.25,0.5)\cap2^{-1}(a,b)\}\right|\\
        = & \begin{cases}
            \frac{4}{5}m(a,b) = \frac{4}{5}(b-a)&\text{if }0 \le a < b \le \frac{1}{2}\\
            \frac{4}{5}\left(m(a,b)+m\left([0.25,0.5)\cap2^{-1}(a,b)\right)\right) = \frac{6}{5}(b-a)&\text{if }\frac{1}{2} \le a < b \le 1.
        \end{cases}
    \end{alignat*}
\end{proof}

\begin{remark}\label{RemarkOnImprovingTheConstruction}
    Given $C \in \left(1,\frac{5}{4}\right)$, we can modify the previous construction to let $q_{M+1} = q_{M+2} = 2$ if and only if $4^{n-1}y \in [0.25,0.5)$ and $M(n) < Cn$.
    Doing so will result in the measure $\mu$ given by

    \begin{equation}
        \mu(a,b) = \begin{cases}
                    \frac{1}{C}(b-a)&\text{if }0 \le a < b \le \frac{1}{2}\\
                    \left(2-\frac{1}{C}\right)(b-a)&\text{if }\frac{1}{2}\le a < b \le 1.
                   \end{cases}
    \end{equation}
    It is worth noting that in Lemma \ref{DistributionNormalImpliesgNormal} it is shown that $\mathcal{N}(Q) \subseteq \mathcal{N}_2$, so this example also shows that Lemma \ref{DistributionNormalImpliesgNormal} does not have a converse.
\end{remark}

\begin{example}\label{N(Q)ButNotDN(Q)}
    For $a,b \in \mathbb{N}_{\ge 2}$ satisfying $a < b$ and $a|b$, there exists a non-deterministic dynamically generated basic sequence $Q = (q_n)_{n = 1}^\infty \in \{a,b\}^\mathbb{N}$ for which $\mathcal{N}(Q)\setminus\mathcal{DN}(Q) \neq \emptyset$.
\end{example}

\begin{proof}
    Let $(\Omega,\mathbb{P})$ be a probability space and $(q_n(\omega))_{n = 1}^\infty$ a sequences of i.i.d. random variables taking values in $\{a,b\}$ with $\mathbb{P}(q_n = a) = \mathbb{P}(q_n = b) = \frac{1}{2}$. 
    Let $\epsilon \in \big(0,\frac{1}{b}\big]$ be arbitrary and $(y_n(\omega))_{n = 1}^\infty$ be another sequence of independent random variables defined by

    \begin{equation}
        y_n(\omega) = \begin{cases}
                        0&\text{ with probability }\frac{1}{a}+\epsilon\text{ if }q_n(\omega) = a\\
                        1&\text{ with probability }\frac{1}{a}-\epsilon\text{ if }q_n(\omega) = a\\
                        i&\text{ with probability }\frac{1}{a}\hskip 6.5mm\text{ if }q_n(\omega) = a\text{ and }2 \le i < a\\
                        0&\text{ with probability }\frac{1}{b}-\epsilon\text{ if }q_n(\omega) = b\\
                        1&\text{ with probability }\frac{1}{b}+\epsilon\text{ if }q_n(\omega) = b\\
                        i&\text{ with probability }\frac{1}{b}\hskip 6.5mm\text{ if }q_n(\omega) = b\text{ and }2 \le i < b.
                      \end{cases}
    \end{equation}
    For the sake of concreteness, we mention that $y_n(\omega)$ and $q_m(\omega)$ are independent when $n \neq m$.
    Letting $Q(\omega) = (q_n(\omega))_{n = 1}^\infty$, we will show that we have $y(\omega) := 0.y_1(\omega)\cdots y_n(\omega)\cdots_{Q(\omega)} \in \mathcal{N}(Q(\omega))\setminus\mathcal{DN}(Q(\omega))$ with probability $1$.
    Since the i.i.d. sequence $(q_n(\omega))_{n = 1}^\infty$ over $\Omega$ can be modeled as $\left(q_1\left(T^n\omega\right)\right)_{n = 1}^\infty$ for some Bernoulli measure preserving transformation $T:\Omega\rightarrow\Omega$, we can take $\omega$ to be a generic point to see that $(q_n(\omega))_{n = 1}^\infty$ is dynamically generated for a.e. $\omega \in \Omega$, which will yield the desired result.

    To this end, we begin by showing that with probability $1$ we have $y(\omega) \in \mathcal{N}(Q(\omega))$. 
    We will show that all blocks of digits $D = (d_1,\cdots,d_\ell)$ occur with the correct frequency by induction on the length $\ell$ of the block.
    We begin with the base case of $\ell = 1$ of our induction.
    We remark that all calculations in the rest of this proof occur with probability $1$ with respect to $\omega$.

    \begin{alignat}{2}
     &\frac{1}{n}Q(\omega)_n(\{i\}) = \frac{1}{n}\sum_{j = 1}^n\frac{1}{q_j(\omega)} \underset{n\rightarrow\infty}{\longrightarrow} \frac{1}{2}\left(\frac{1}{a}+\frac{1}{b}\right) = \frac{a+b}{2ab}\text{, for }0 \le i < a,\\
     &\frac{1}{n}Q(\omega)_n(\{i\}) = \frac{1}{n}\sum_{j = 1}^n\frac{1}{q_j(\omega)}\mathbbm{1}_{q_j(\omega) = b} \underset{n\rightarrow\infty}{\longrightarrow}\frac{1}{2b}\text{, for }a \le i < b,\\
     &\lim_{n\rightarrow\infty}\frac{1}{n}N_n^{Q(\omega)}(0,y(\omega)) = \underbrace{\frac{1}{2}}_{\mathbb{P}(q_n = a)}\cdot\underbrace{\left(\frac{1}{a}+\epsilon\right)}_{\mathbb{P}(y_n = 0)}+\underbrace{\frac{1}{2}}_{\mathbb{P}(q_n = b)}\cdot\underbrace{\left(\frac{1}{b}-\epsilon\right)}_{\mathbb{P}(y_n = 0)} = \frac{a+b}{2ab},\\
     &\lim_{n\rightarrow\infty}\frac{1}{n}N_n^{Q(\omega)}(1,y(\omega)) = \underbrace{\frac{1}{2}}_{\mathbb{P}(q_n = a)}\cdot\underbrace{\left(\frac{1}{a}-\epsilon\right)}_{\mathbb{P}(y_n = 1)}+\underbrace{\frac{1}{2}}_{q_n(\omega) = b}\cdot\underbrace{\left(\frac{1}{b}+\epsilon\right)}_{\mathbb{P}(y_n = 1)} = \frac{a+b}{2ab},
     \end{alignat}
     \begin{alignat}{2}
     &\lim_{n\rightarrow\infty}\frac{1}{n}N_n^{Q(\omega)}(i,y(\omega)) = \underbrace{\frac{1}{2}}_{\mathbb{P}(q_n = a)}\cdot\underbrace{\frac{1}{a}}_{\mathbb{P}(y_n = i)}+\underbrace{\frac{1}{2}}_{\mathbb{P}(q_n = b)}\cdot\underbrace{\frac{1}{b}}_{\mathbb{P}(y_n = i)} = \frac{a+b}{2ab}\text{, for }2 \le i < a\\
     &\lim_{n\rightarrow\infty}\frac{1}{n}N_n^{Q(\omega)}(i,y(\omega)) = \underbrace{\frac{1}{2}}_{\mathbb{P}(q_n = a)}\cdot\underbrace{0}_{\mathbb{P}(y_n = i)}+\underbrace{\frac{1}{2}}_{\mathbb{P}(q_n = b)}\cdot\underbrace{\frac{1}{b}}_{\mathbb{P}(y_n = i)} = \frac{1}{2b}\text{, for }a \le i < b.
    \end{alignat}
    Now that we have completed the base case of $\ell = 1$, we proceed to the inductive step and assume that the desired result holds for all blocks of length $\ell = L$ so that we can show that the result also holds for blocks of length $\ell = L+1$.
    Let $D = (d_1,\cdots,d_L,d_{L+1}) \in \{0,\cdots,b-1\}^{L+1}$ be a block of digits.
    Recalling that $(q_n(\omega))_{n = 1}^\infty$ and $(y_n(\omega))_{n = 1}^\infty$ are i.i.d. sequences, we see that if for some $j \in \mathbb{N}$ we consider the set $\Omega_{j,B} \subseteq \Omega$ given by
    
    \begin{equation}
        \Omega_{j,B} = \{\omega \in \Omega\ |\ \left(y_j(\omega),y_{j+1}(\omega),\cdots,y_{j+L-1}(\omega)\right) = (b_1,b_2,\cdots,b_L)\},
    \end{equation} 
    then we still have that $\mathbb{P}(q_{j+\ell} = a\ |\ \omega \in \Omega_{j,B}) = \mathbb{P}(q_{j+\ell} = b\ |\ \omega \in \Omega_{j,B}) = \frac{1}{2}$, and the distribution of $y_{j+L}(\omega)$ also does not change when restricted to $\Omega_{j,B}$.
    Consequently, we see that if $D' = (d_1,d_2,\cdots,d_L)$, then

    \begin{alignat}{2}
        &\lim_{n\rightarrow\infty}\frac{1}{n}Q_n(\omega)(D) = \left(\lim_{n\rightarrow\infty}\frac{1}{n}Q(\omega)_n\left(D'\right)\right)\cdot\left(\lim_{n\rightarrow\infty}\frac{1}{n}Q(\omega)_n(\{b_{L+1}\})\right)\text{, and}\\
        &\lim_{n\rightarrow\infty}\frac{1}{n}N_n^{Q(\omega)}(D,y(\omega)) = \left(\lim_{n\rightarrow\infty}\frac{1}{n}N_n^{Q(\omega)}\left(D',y(\omega)\right)\right)\left(\lim_{n\rightarrow\infty}\frac{1}{n}N_n^{Q(\omega)}(d_{L+1},y(\omega))\right),
    \end{alignat}
    so the desired result follows from the induction hypothesis.

    Now we show that $y(\omega) \notin \mathcal{DN}(Q(\omega))$ almost surely.
    Let $c = \frac{b}{a}$.
    We observe that $q_n(\omega)\cdots q_1(\omega)\allowbreak y(\omega) \in \big[0,\frac{1}{a}\big)$ if and only if $(q_{n+1}(\omega),x_{n+1}(\omega)) \in \{(a,0),(b,0),(b,1),\cdots,(b,c-1)\}$, so 
    
    \begin{equation}
        \left(\mathbbm{1}_{\big[0,\frac{1}{a}\big)}(q_n(\omega)\cdots q_1(\omega)y(\omega))\right)_{n = 1}^\infty
    \end{equation}
    is an i.i.d. sequence of random variables.
    Consequently, we can use the strong law of large numbers to see that

    \begin{alignat*}{2}
        &d\left(\left\{n \in \mathbb{N}\ |\ q_1(\omega)q_2(\omega)\cdots q_n(\omega)y(\omega) \in \bigg[0,\frac{1}{a}\bigg)\right\}\right) = \lim_{N\rightarrow\infty}\frac{1}{N}\sum_{n = 1}^N\mathbbm{1}_{\big[0,\frac{1}{a}\big)}(q_n(\omega)\cdots q_1(\omega)y(\omega))\\
        = &\mathbb{P}\left(y(\omega) \in \bigg[0,\frac{1}{a}\bigg)\right) = \underbrace{\frac{1}{2}}_{\mathbb{P}(q_1 = a)}\cdot\underbrace{\left(\frac{1}{a}+\epsilon\right)}_{\mathbb{P}(y_1 = 0)}+\underbrace{\frac{1}{2}}_{\mathbb{P}(q_1 = b)}\cdot\underbrace{\frac{c}{b}}_{\mathbb{P}(y_1 < c)} = \frac{1}{a}+\frac{1}{2}\epsilon > \frac{1}{a}.
    \end{alignat*}
\end{proof}

\begin{remark}
    In Example \ref{N(Q)ButNotDN(Q)}, the assumption that $a|b$ was used to simplify the proof that $y(\omega) \notin\mathcal{DN}(Q(\omega))$ almost surely.
    We leave it as an exercise to show that Example \ref{N(Q)ButNotDN(Q)} holds for any $a \neq b \in \mathbb{N}_{\ge 2}$.
\end{remark}

\begin{example}\label{NecessityOfFiniteIntegral}
    Let $g \in \mathbb{N}_{\ge 2}$ and $Q = (q_n)_{n = 1}^\infty \in \left(\{g^n\}_{n = 1}^\infty\right)^\mathbb{N}$ be a basic sequence generated by $(X,\mathscr{B},\mu,T,f,x)$ with $\int_X\log(f)d\mu = \infty$.
    \begin{enumerate}[(i)]
        \item $\mathcal{N}_g\setminus(\mathcal{N}(Q)\cup\mathcal{DN}(Q)) \neq \emptyset$.

        \item $\mathcal{UN}(Q)\setminus\mathcal{N}_g \neq \emptyset$. 
    \end{enumerate}
\end{example}

\begin{proof}[Proof of (i)]
    Let $q_n = g^{a_n}$, let $s_n = \sum_{i = 1}^na_i$, let $(Y_n')_{n = 1}^\infty$ be an i.i.d. sequence of random variables whose values are uniformly distributed over $\{0,1,\cdots,g-1\}$, and consider the sequence of independent random variables $(Y_n)_{n = 1}^\infty$ given by
    
    \begin{equation}
        Y_n = \begin{cases}
                0&\text{if }a_m = a_1\text{ and }s_{m-1} < n \le s_m\\
                Y_n'&\text{ otherwise}.
              \end{cases}
    \end{equation}
    We will show that $y = 0.Y_1Y_2\cdots Y_n\cdots_g$ is almost surely a $g$-normal number that is not $Q$-normal or $Q$-distributional normal. To this end, we begin by observing that 

    \begin{equation}
        \lim_{N\rightarrow\infty}\frac{\sum_{n = 1}^Na_1\mathbbm{1}_{a_n = a_1}}{\sum_{n = 1}^Na_n} = \lim_{N\rightarrow\infty}\frac{\frac{1}{N}\sum_{n = 1}^Na_1\mathbbm{1}_{a_n = a_1}}{\frac{1}{N}\sum_{n = 1}^Na_n} 
 \le \lim_{N\rightarrow\infty}\frac{a_1}{\frac{1}{N}\sum_{n = 1}^N\log(q_n)} = 0.
    \end{equation}
    It follows that the digits of the base $g$ expansion of $y$ are generated by an i.i.d. sequence of random variables along a subset of full density, so $y$ is almost surely normal base $g$. 
    To see that $y$ is never normal base $Q$, it suffices to observe that the digits corresponding to a base $q_n \neq q_1$ are generated uniformly at random, so they all occur with the expected frequency, but the digits corresponding to a base of $q_n = q_1$ are always $0$, so the frequency of the digit $0$ will be too high.
    Furthermore, to see that $y$ is almost surely not $Q$-distribution normal, we let $v_n = q_n\cdots q_2q_1y\pmod{1}$ and observe that if $y$ was $Q$-distribution normal, then we would have

    \begin{equation}
        V := \lim_{N\rightarrow\infty}\frac{1}{N}\left|\left\{1 \le n \le N\ |\ v_n \in \bigg[0,\frac{1}{g}\bigg)\right\}\right| = \frac{1}{g}.
    \end{equation}
    However, letting $y = 0.y_1y_2\cdots y_n\cdots_Q$ be the base $Q$ expansion of $y$ and $d_m := \displaystyle\lim_{N\rightarrow\infty}\frac{1}{N}\sum_{n = 1}^N\mathbbm{1}_{q_n = g^m}$, we see that

    \begin{equation}
        V = \lim_{N\rightarrow\infty}\frac{1}{N}\sum_{n = 1}^N\mathbb{P}\left(y_n \in \big[0,g^{a_n-1}\big)\right) = d_1+\sum_{n = 2}^\infty \frac{1}{g}d_n = \frac{1}{g}+\left(1-\frac{1}{g}\right)d_1 > \frac{1}{g}.
    \end{equation}
\end{proof}

\begin{proof}[Proof of (ii)]
    For each $m \in \mathbb{N}$ let $B_m = \{n \in \mathbb{N}\ |\ q_n = g^m\}$.
    We observe that if $y \in [0,1)$ is such that all of its digits in its base $Q$ expansion are randomly generated, then $y$ is almost surely uniformly $Q$-normal.
    Furthermore, if we take a $y \in [0,1)$, and for each $m \in \mathbb{N}$ we alter finitely many digits of the base $Q$ expansion of $y$ along the set $B_m$ to produce a number $y'$, then $y' \in \mathcal{UN}(Q)$ if and only if $y \in \mathcal{UN}(Q)$.
    Using this idea, we will construct a $y \in [0,1)$ that is not normal base $g$.
    We inductively construct the base $Q$ digits of $y$. 
    For the base case, let the first digit of $y$ be arbitrary and let $N_1 = 1$. 
    Now assume that the first $N_m$ digits of $y$ have been constructed. 
    Let $A_m$ denote the collection of bases that do not appear in $(q_n)_{n = 1}^{N_m}$ and observe that

    \begin{equation}
        \lim_{N\rightarrow\infty}\frac{\frac{1}{N}\sum_{n = 1}^Na_n\mathbbm{1}_{A_m^c}(q_n)}{\frac{1}{N}\sum_{n = 1}^Na_n} \le \lim_{N\rightarrow\infty}\frac{\text{max}(A_m^c)}{\frac{1}{N}\sum_{n = 1}^N\log(q_n)} = 0\text{, hence }\lim_{N\rightarrow\infty}\frac{\frac{1}{N}\sum_{n = 1}^Na_n\mathbbm{1}_{A_m^c}(q_n)}{\frac{1}{N}\sum_{n = 1}^Na_n\mathbbm{1}_{A_m}(q_n)} = 0.
    \end{equation}
    Let $N_{m+1}'$ be such that

    \begin{equation}\label{TooMany0sEquation}
        \frac{\sum_{n = 1}^{N_{m+1}'}a_n\mathbbm{1}_{A_m^c}(q_n)}{\sum_{n = 1}^{N_{m+1}'}a_n\mathbbm{1}_{A_m}(q_n)} < \frac{1}{m}.
    \end{equation}
    If $q_n \in A_m$ for some $N_m < n \le N_{m+1}'$, then set $y_n = 0$, and if $q_n \in A_m$ for some $n > N_{m+1}'$, then let $y_n$ be picked uniformly at random from all possible digits. 
    Equation \eqref{TooMany0sEquation} tells us that the first $N_{m+1}'$ digits of the base $g$ expansion of $y$ have at least $m$ times as many $0$s as all other digits combined, which shows us that $y$ is not normal base $g$. 
\end{proof}

\section{Properties of Dynamically Generated Basic Sequences}\label{DynamicallyGeneratedBasicSequencesSection}

We may now begin to prove the theorems in the paper and develop our notions of dynamically generated basic sequences and uniform normality and distribution normality (defined in \refs{DefinitionOfDGBSSubsection}).

\begin{lemma}\label{WhenIsASequenceDynamicallyGenerated}
    A basic sequence $Q = (q_n)_{n = 1}^\infty$ is dynamically generated if and only if the following conditions hold:
    \begin{enumerate}[(i)]
        \item For any $w := (w_1,w_2,\cdots,w_k) \in \mathbb{N}_{\ge 2}^k$ the following limit exists

    \begin{equation}\label{Assumption1ForDSS}
        d(w) := \lim_{N\rightarrow\infty}\frac{1}{N}\left|\left\{1 \le n \le N\ |\ q_{n+i} = w_i\ \forall\ 0 \le i < k\right\}\right|.
    \end{equation}

    \item If $d(w) = 0$, then $w \neq [q_n,q_{n+1},\cdots,q_{n+k-1}]$ for any $n \in \mathbb{N}$.

    \item For any $k \in \mathbb{N}$, we have

    \begin{equation}\label{Assumption2ForDSS}
        \sum_{|w| = k}d(w) = 1.
    \end{equation}
    \end{enumerate}
\end{lemma}

\begin{proof}
    For the first direction, let us assume that $Q$ satisfies (i)-(iii). Let $X' = \mathbb{N}_{\ge 2}^\mathbb{N}$ with the product topology, which is complete separable metric space under the metric $d:X'\times X'\rightarrow[0,1]$ given by 
    
    \begin{equation}
        d\left((x_n)_{n = 1}^\infty,(y_n)_{n = 1}^\infty\right) = \sum_{n = 1}^\infty\frac{1}{2^n}\delta_{x_n = y_n}.
    \end{equation}
    Let $T:X'\rightarrow X'$ be given by $T(x_n)_{n = 1}^\infty = (x_{n+1})_{n = 1}^\infty$, let $X = c\ell(\{T^nQ\}_{n = 1}^\infty)$, and let $f:X\rightarrow\mathbb{N}_{\ge 2}$ be given by $f((x)_{n = 1}^\infty) = x_1$. 
    We see that for each $w = (w_1,w_2,\cdots,w_k) \in \mathbb{N}_{\ge 2}^k$ we have
    
    \begin{equation}
        E_w = \left\{(x_n)_{n = 1}^\infty \in X\ |\ x_i = w_i\ \forall\ 1 \le i \le k\right\},
    \end{equation}
    which is a basic open set of $X$.
    We use Equations \eqref{Assumption1ForDSS} to see that

    \begin{equation}
        \lim_{N\rightarrow\infty}\frac{1}{N}\sum_{n = 1}^N\mathbbm{1}_{E_w}(T^nQ) = d(w).
    \end{equation}
    Equation \eqref{Assumption2ForDSS} allows us to invoke the Kolmogorov Consistency Theorem to see that  

    \begin{equation}
        \mu := \lim_{N\rightarrow\infty}\frac{1}{N}\sum_{n = 0}^{N-1}\delta_{T^nQ}
    \end{equation}
    is a well defined probability measure with convergence taking place in the weak$^*$ topology. We see that $Q$ is a generic point for $\mu$ by construction, condition (ii) ensures that each non-empty open set has positive $\mu$-measure, and it is immediate that $q_n = f(T^nQ)$ and that $f$ generates the topology of $X$.

    For the next direction let us assume  $Q$ is generated by $(X,\mathscr{B},\mu,T,f,x)$. 
    Property (i) follows from the fact that each $E_w$ is an open set and $x$ is a generic point. Property (ii) follows from the fact that all non-empty open sets have positive $\mu$-measure. Property (iii) follows from the fact that $\{E_w\ |\ w \in \mathbb{N}_{\ge 2}^k\}$ forms a partition of $X$.
\end{proof}

\begin{remark}\label{RemarkJustifyingFullSupportAssumption}
    The reason that we have condition (ii) in Lemma~\ref{WhenIsASequenceDynamicallyGenerated}, and analogously the reason that we assume $\mu$ assigns positive measure to open subsets of $X$ is seen through the following example.
    Let $(q_n)_{n = 1}^\infty$ be given by $q_n = 2$ if $n$ is not a perfect square, and $q_n = 3$ if $n$ is a perfect square.
    Let $y = 0.y_1y_2\cdots y_n\cdots_2$ be a normal number base $2$, and let $y' = 0.y_1y_2\cdots y_n\cdots_Q$.
    It can be checked that $y' \in \mathcal{DN}(Q)\setminus\mathcal{N}(Q)$ even though $Q$ satisfies conditions (i) and (iii) of Theorem \ref{WhenIsASequenceDynamicallyGenerated}.
    The reason for this is that $\sum_{n = m^2}\frac{1}{3} = \infty$, so we still expect to see infinitely many occurrences of the digit $2$ in elements of $\mathcal{N}(Q)$ even though the density of the occurrences of $3$s among the $q_n$ is $0$. 
    In particular, we see that $d(w) = \mu(E_w)$, and the measurable dynamics that we do on $X$ will ignore the $E_w$ for which $\mu(E_w) = 0$, even though the notion of normality for a basic sequence does not ignore bases that appear with $0$ density.
    See also Remark \ref{RemarkAboutPotentialGeneralization}.
\end{remark}

\begin{lemma}\label{LimitsExistForDGBS}
    Let $Q = (q_n)_{n = 1}^\infty$ be a dynamically generated basic sequence.
    \begin{enumerate}[(i)]
        \item For any block of digits $D = (d_1,\cdots,d_\ell) \in \mathbb{N}_0^\ell$,

    \begin{equation}
        P_D = \lim_{n\rightarrow\infty}\frac{Q_n(D)}{n}\text{ exists.}
    \end{equation}

    \item For any block of digits $D = (d_1,\cdots,d_\ell) \in \mathbb{N}_0^\ell$ and any block of bases $B = (b_1,\cdots,b_\ell) \in \mathbb{N}_{\ge 2}^\ell$,

    \begin{equation}
        P_{D,B} := \lim_{n\rightarrow\infty}\frac{Q_n(D,B)}{n}\text{ exists.}
    \end{equation}

    \item For any $\ell \in \mathbb{N}$ we have

    \begin{equation}\label{SumOfProbabilitiesIs1Equation}
        \sum_{D \in \mathbb{N}_0^\ell}P_D = 1.
    \end{equation}

    \item For any $\ell \in \mathbb{N}$ and any block of digits $D \in \mathbb{N}_0^\ell$ we have

    \begin{equation}\label{SumOfProbabilitiesIs1Equation2}
        \sum_{B \in \mathbb{N}_{\ge 2}^\ell}P_{D,B} = P_D.
    \end{equation}
    \end{enumerate}
\end{lemma}

\begin{proof}
    Let $Q$ be generated by $(X,\mathscr{B},\mu,T,f,x_0)$. To show (i), let

    \begin{equation}\label{f_DDefiningEquation}
        f_D(x) = \begin{cases}
                    \left(\prod_{i = 1}^\ell f\left(T^ix\right)\right)^{-1}&\text{if }f\left(T^ix\right) > d_i\ \forall\ 1 \le i \le \ell\\
                    0&\text{else.}
              \end{cases}
    \end{equation}
    Since $f_D \in C(X)$ and $x_0 \in X$ is generic, we see that

    \begin{equation}
        P_D = \lim_{n\rightarrow\infty}\frac{Q_n(D)}{n} = \lim_{N\rightarrow\infty}\frac{1}{N}\sum_{n = 1}^Nf_D\left(T^nx_0\right) = \int_Xf_Dd\mu
    \end{equation}
    is well defined. To show (ii), let

    \begin{equation}
        f_{D,B}(x) = \begin{cases}
                    \left(\prod_{i = 1}^\ell b_i\right)^{-1}&\text{if }x \in \bigcap_{i = 1}^\ell T^{-i}f^{-1}(\{b_i\})\text{ and }b_i > d_i\text{ for all }1 \le i \le \ell\\
                    0&\text{else.}
              \end{cases}
    \end{equation}
    Since $f_{D,B} \in C(X)$ and $x_0 \in X$ is generic, we see that

    \begin{equation}
        P_{D,B} := \lim_{n\rightarrow\infty}\frac{Q_n(D,B)}{n} = \lim_{N\rightarrow\infty}\frac{1}{N}\sum_{n = 1}^Nf_{D,B}(T^nx_0) = \int_Xf_{D,B}d\mu
    \end{equation}
    is well defined. 
    
    We now proceed to show (iii). Firstly, we observe that

    \begin{equation}
        \sum_{D \in \mathbb{N}_0^\ell}P_D = \sum_{D \in \mathbb{N}_0^\ell}\int_Xf_Dd\mu = \int_X\sum_{D \in \mathbb{N}_0^\ell}f_Dd\mu,
    \end{equation}
    so to prove Equation \eqref{SumOfProbabilitiesIs1Equation} it suffices to show that $\displaystyle\sum_{w \in \mathbb{N}_0^\ell}f_w = 1$. To this end, let $x \in X$ be arbitrary and let $w = (w_1,\cdots,w_\ell) \in \mathbb{N}_{\ge 2}^\ell$ be such that $x \in E_w = \bigcap_{i = 1}^\ell T^{-i}f^{-1}(\{w_i\})$. We see that

    \begin{equation}
        \sum_{D \in \mathbb{N}_0^\ell}f_D(x) = \sum_{\underset{1 \le i \le \ell}{0 \le d_i < w_i}}\left(\prod_{i = 1}^\ell w_i\right)^{-1} = 1.
    \end{equation}
    Lastly, to show (iv) we observe that

    \begin{equation}
        \sum_{B \in \mathbb{N}_{\ge 2}^\ell}P_{D,B} = \sum_{B \in \mathbb{N}_{\ge 2}^\ell}\int_Xf_{D,B}d\mu = \int_X\sum_{B \in \mathbb{N}_{\ge 2}^\ell}f_{D,B}d\mu,
    \end{equation}
    so Equation \eqref{SumOfProbabilitiesIs1Equation2} follows from the observation that $\displaystyle\sum_{B \in \mathbb{N}_{\ge 2}^\ell}f_{D,B} = f_D$.
\end{proof}

\begin{remark}
In light of Lemma \ref{LimitsExistForDGBS}, we see that for a dynamically generated basic sequence $Q$ and a $z \in [0,1)$, we have $z \in \mathcal{N}(Q)$ if and only if for all $D \in \mathbb{N}_0^\ell$ with $P_D \neq 0$ we have

\begin{equation}\label{NormalityForDGBSEquation}
    \lim_{n\rightarrow\infty}\frac{N_n^Q(D,z)}{n} = P_D = \lim_{n\rightarrow\infty}\frac{Q_n(D)}{n}.
\end{equation}
Similarly, we see that $z \in \mathcal{UN}(Q)$ if and only if for all $D \in \mathbb{N}_0^\ell$ and $B \in \mathbb{N}_{\ge 2}^\ell$ with $P_{D,B} \neq 0$ we have

\begin{equation}\label{UniformNormalityForDGBSEquation}
    \lim_{n\rightarrow\infty}\frac{N_n^Q(D,B,z)}{n} = P_{D,B} = \lim_{n\rightarrow\infty}\frac{Q_n(D,B)}{n}.
\end{equation}
\end{remark}

It is worth remarking here that $P_{D,B} \neq 0$ if and only if $D < B$ and $\mu(E_B) > 0$.
In particular, $Q_n(D,B) \neq 0$ for some $n$, then $P_{D,B} \neq 0$. Our next goal will be to obtain dynamical reformulations of Equations \eqref{NormalityForDGBSEquation} and \eqref{UniformNormalityForDGBSEquation}.

Given a block of digits $D = (d_1,\cdots,d_\ell) \in \mathbb{N}_0^\ell$ and a block of bases $B = (b_1,\cdots,b_\ell) \in \mathbb{N}_{\ge 2}^\ell$ satisfying $D < B$, we define

\begin{alignat*}{2}
    I_{D,B}& = \Bigg[\displaystyle\sum_{i = 1}^\ell\frac{d_i}{b_1\cdots b_i},\displaystyle\sum_{i = 1}^\ell\frac{d_i}{b_1\cdots b_i}+\frac{1}{b_1\cdots b_\ell}\Bigg)\text{, and}\\
    E_{D,B}& = T(E_B)\times I_{D,B}.
\end{alignat*}
We observe that for $z = 0.z_1z_2\cdots z_n\cdots_Q$ and $n \in \mathbb{N}$ we have $((z_n,\cdots,z_{n+\ell-1}),(q_n,\cdots,q_{n+\ell-1})) = (D,B)$ if and only if $(T\rtimes M)^n(x,z) \in E_{D,B}$. 

We will now examine a concrete example. Consider the basic sequence $Q$ generated by $([0,1),\mathscr{L}\allowbreak,\times2,m,f,x)$ where $f(y) = 2$ if $y < \frac{1}{2}$ and $f(y) = 3$ if $y \ge \frac{1}{2}$,\footnote{We see that we would have to work with $X = \{0,1\}^\mathbb{N}$ and the left shift map in order to have a continuous $f$, but we choose to work with $X = [0,1)$ in order to create Figures 1-4.} and $x$ is normal base $2$. 
We see that $E_2 = [0,\frac{1}{2})$ and $E_3 = [\frac{1}{2},1)$, and that for any $\ell \in \mathbb{N}$ and any $B \in \{2,3\}^\ell$ we have $E_B = [\frac{a}{2^\ell},\frac{a+1}{2^\ell})$ for some $0 \le a < 2^\ell$. 
The next 3 figures show the sets $E_B\times I_{D,B}$ as subsets of $[0,1)^2$ for all $D$ and $B$ with $B \in \{2,3\}^\ell$ and $1 \le \ell \le 3$. 
In these figures, the $x$-axis represents $[0,1)$, the sets $E_B$ are represented by $B$, and the sets $E_B\times I_{D,B}$ are the rectangles marked with $D$ that are over the set $E_B$.

\begin{figure}[!htb]
   \begin{minipage}{0.48\textwidth}
     \centering
     \includegraphics[width=\linewidth]{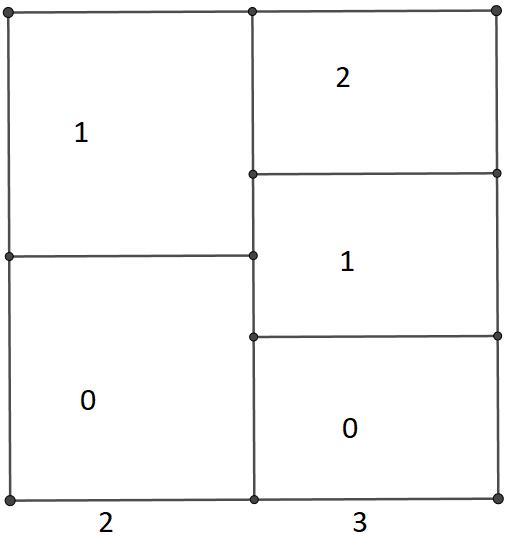}
     \caption{$E_B$ and $E_B\times I_{D,B}$ when $\ell = 1$.}\label{Level1Grid}
   \end{minipage}\hfill
   \begin{minipage}{0.48\textwidth}
     \centering
     \includegraphics[width=\linewidth]{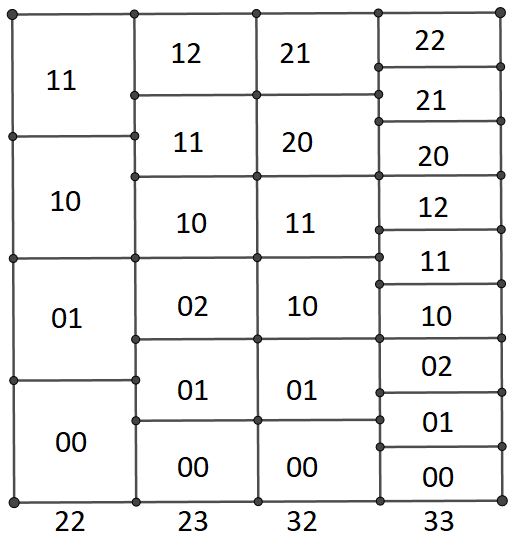}
     \caption{$E_B$ and $E_B\times I_{D,B}$ when $\ell = 2$.}\label{Level2Grid}
   \end{minipage}
\end{figure}

\pagebreak
\textcolor{white}{a}
\begin{figure}[H]
    \centering
    \includegraphics[width=\textwidth]{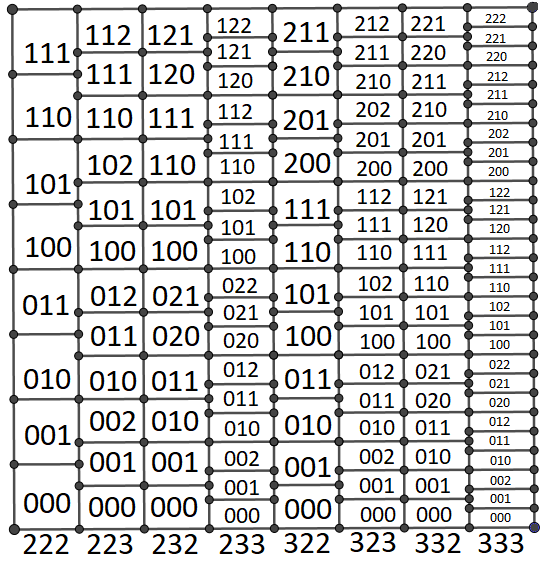}
    \caption{$E_B$ and $E_B\times I_{D,B}$ when $\ell = 3$.}
    \label{Level3Grid}
\end{figure}

\pagebreak
\textcolor{white}{a}
\begin{figure}[H]
    \centering
    \includegraphics[width=\textwidth]{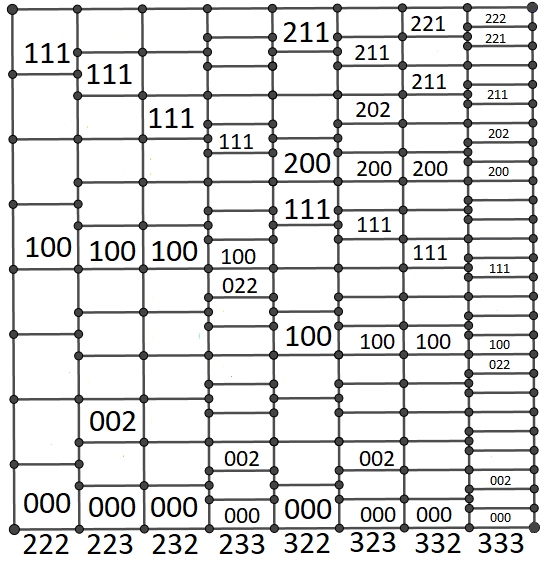}
    \caption{$E_B$ and $E_B\times I_{D,B}$ for selected values of $D$ when $\ell = 3$.}
    \label{PartiallyFilledLevel3Grid}
\end{figure}

\pagebreak
We observe that for any block of digits $D \in \mathbb{N}_0^\ell$, we have $P_D = \mu\times m(S_D)$, where $S_D := \bigcup_{B \in \mathbb{N}_{\ge 2}^\ell}E_{D,B}$. The reader is encouraged to revisit figures 1,2, and 4 in to see examples of the sets $S_D$. It is seen that the sets $S_D$ can have a wide variety of shapes, and it is not currently clear what structure, if any, they possess. Similarly, for any block of digits $D \in \mathbb{N}_0^\ell$ and any block of bases $B \in \mathbb{N}_{\ge 2}^\ell$ we have $P_{D,B} = \mu\times m(E_{D,B})$. We are now ready to restate the characterizations of $\mathcal{N}(Q)$ and $\mathcal{UN}(Q)$ given by Equations \eqref{NormalityForDGBSEquation} and \eqref{UniformNormalityForDGBSEquation} 
 respectively. 

 \begin{theorem}\label{SkewProductCharacterizationsOfNormalities}
     Let $Q$ be a basic sequence generated by $(X,\mathscr{B},\mu,T,f,x)$ and let $z \in [0,1)$.
     \begin{enumerate}[(i)]
         \item $z \in \mathcal{N}(Q)$ if and only if for any block of digits $D \in \mathbb{N}_0^\ell$ we have

         \begin{equation}
             \mu\times m(S_D) = \lim_{N\rightarrow\infty}\frac{1}{N}\sum_{n = 1}^N\mathbbm{1}_{S_D}\left((T\rtimes M)^n(x,z)\right).
         \end{equation}

         \item $z \in \mathcal{UN}(Q)$ if and only if for any block of digits $D \in \mathbb{N}_0^\ell$ and any block of bases $B \in \mathbb{N}_{\ge 2}^\ell$ we have

         \begin{equation}
             \mu\times m(E_{D,B}) = \lim_{N\rightarrow\infty}\frac{1}{N}\sum_{n = 1}^N\mathbbm{1}_{E_{D,B}}\left((T\rtimes M)^n(x,z)\right).
         \end{equation}

         \item $z \in \mathcal{DN}(Q)$ if and only if for any $0 \le a < b \le 1$ we have

         \begin{equation}
             b-a = \lim_{N\rightarrow\infty}\frac{1}{N}\sum_{n = 1}^N\mathbbm{1}_{X\times(a,b)}((T\rtimes M)^n(x,z)).
         \end{equation}
     \end{enumerate}
 \end{theorem}

\begin{proof}[Proof of Theorem~\ref{UN=UDN}]
Let $Q$ be generated by $(X,\mathscr{B},\mu,T,f,x)$. 
We will first show that $\mathcal{UDN}(Q) \subseteq \mathcal{UN}(Q)$, so let $z \in \mathcal{UDN}(Q)$ be arbitrary.
Now let $\ell \in \mathbb{N}$, $D \in \mathbb{N}_0^\ell$, $B \in \mathbb{N}_{\ge 2}^\ell$, and $\epsilon > 0$ all be arbitrary. 
Let $h_1,h_2 \in C_b([0,1))$ be such that $0 \le h_1 \le \mathbbm{1}_{I_{D,B}} \le h_2$ and $\int_0^1(h_2(y)-h_1(y))dy < \epsilon$.
Since $\mathbbm{1}_{E_B} \in C(X)$, we see that $F_1(x,y) := \mathbbm{1}_B(x)h_1(y)$ and $F_2(x,y) := \mathbbm{1}_B(x)h_2(y)$ satisfy $F_1,F_2 \in C(X\times[0,1))\cap L^1(X\times[0,1),\mu\times m)$ and $0 \le F_1 \le \mathbbm{1}_{E_{D,B}} \le F_2$. 
Since $z \in \mathcal{UDN}(Q)$, we see that

\begin{alignat*}{2}
\mu(E_B)\int_0^1h_1(y)dy& = \lim_{N\rightarrow\infty}\frac{1}{N}\sum_{n = 1}^NF_1((T\rtimes M)^n(x,z)) \le \lim_{N\rightarrow\infty}\frac{1}{N}\sum_{n = 1}^N\mathbbm{1}_{E_{D,B}}((T\rtimes M)^n(x,z))\\
&\le \lim_{N\rightarrow\infty}\frac{1}{N}\sum_{n = 1}^NF_2((T\rtimes M)^n(x,z)) = \mu(E_B)\int_0^1h_2(y)dy\text{, hence}\\
&\left|\lim_{N\rightarrow\infty}\frac{1}{N}\sum_{n = 1}^N\mathbbm{1}_{E_{D,B}}((T\rtimes M)^n(x,z))-\mu(E_{D,B})\right| < \epsilon\mu(E_B).
\end{alignat*}
Since $D,B,$ and $\epsilon$ were all arbitrary, Theorem \ref{SkewProductCharacterizationsOfNormalities}(ii) tells us that $z \in \mathcal{UN}(Q)$.

Now we will show that $\mathcal{UN}(Q) \subseteq \mathcal{UDN}(Q)$, so let $z \in \mathcal{UN}(Q)$ be arbitrary. 
Since the linear span of functions of the form $F(x,y) = g(x)h(y)$ with $g \in C_b(X)\cap L^1(X,\mu)$ and $h \in C_b([0,1))$ are dense in $C(X\times[0,1))\cap L^1(X\times[0,1),\mu\times m)$ with respect to the supremum norm, we see that $z \in \mathcal{UDN}(Q)$ if and only if

\begin{equation}
    \lim_{N\rightarrow\infty}\frac{1}{N}\sum_{n = 1}^NF((T\rtimes M)^n(x,z)) = \int_{X\times[0,1)}Fd\mu\times m,
\end{equation}
for all $F$ of the given form. 
Now let us fix some $g \in C_b(X)\cap L^1(X,\mu), h \in C_b([0,1)),$ and $\epsilon > 0$. 
Since $h$ is uniformly continuous, there exists a $\delta > 0$ such that for any partition $P = \{P_i\}_{i = 1}^{L}$ of $[0,1)$ into intervals of equal length with $L \ge \delta^{-1}$, there exists a step function $h' = \sum_{i = 1}^L c_i\mathbbm{1}_{P_i}$ for which $||h-h'||_\infty < \epsilon$. 
Now let $\ell = \lceil-\log_2(\delta)\rceil$, and observe that for $B = (b_1,\cdots,b_\ell) \in \mathbb{N}_{\ge 2}^\ell$ we have $m(I_{D,B}) = (b_1b_2\cdots b_\ell)^{-1} \le 2^{-\ell} \le \delta$. 
Since $\{E_B\ |\ B \in \mathbb{N}_{\ge 2}^\ell\}$ is a clopen cover of $X$, let $\{E_B\}_{B \in \mathcal{B}}$ be a finite collection for which $g' := g\mathbbm{1}_{X_B}$ with $X_{\mathcal{B}} = \bigcup_{B \in \mathcal{B}}E_B$ satisfies, $||g-g'||_\infty < \epsilon$. 
Observe that if $B_1,B_2$ are two blocks of bases for which $B_1$ is an initial segment of $B_2$, then $E_{B_2} \subseteq E_{B_1}$.
Next, observe that $\mathcal{A} := \{E_B\ |\ \ell \in \mathbb{N},\ B \in \mathbb{N}_{\ge 2}^\ell\ \&\ B\text{ extends some }B' \in \mathcal{B}\}$ is a basis of clopen sets for the topology of $X_{\mathcal{B}}$.
Now let $\mathcal{B}_1$ be a collection of blocks $B$ that each extend some $B' \in \mathcal{B}$ for which we have $||g'-\sum_{B \in \mathcal{B}_1}c_B\mathbbm{1}_{E_B}||_\infty < \epsilon$.
We see that for any $B \in \mathcal{B}_1$ and any block of digits $D < B$, we still have $|I_{D,B}| < \delta$, so we may pick $c_{D,B}$ such that $||h-\sum_{D < B}c_{D,B}\mathbbm{1}_{I_{D,B}}||_\infty < \epsilon$.
It follows that 

\begin{alignat}{2}
    &||g'(x)h(y)-\sum_{B \in \mathcal{B}'}\sum_{D < B}c_{D,B}\mathbbm{1}_{E_{D,B}}(x,y)||_\infty\\
    \le&||g'(x)h(y)-h(y)\sum_{B \in \mathcal{B}'}\mathbbm{1}_{E_B}(x)||_\infty+||h(y)\sum_{B \in \mathcal{B}'}\mathbbm{1}_{E_B}(x)-\sum_{B \in \mathcal{B}'}\sum_{D < B}c_{D,B}\mathbbm{1}_{E_{D,B}}(x,y)||_\infty\\
    \le&||h||_\infty||g-\sum_{B \in \mathcal{B}'}\mathbbm{1}_{E_B}(x)||_\infty+||\sum_{B \in \mathcal{B}'}\mathbbm{1}_{E_B}(x)\left(h(y)-\sum_{D < B}c_{D,B}\mathbbm{1}_{I_{D,B}}(y)\right)||_\infty \le 2\epsilon||h||_\infty+\epsilon.
\end{alignat}
Since $z \in \mathcal{UN}(Q)$, Theorem \ref{SkewProductCharacterizationsOfNormalities}(ii) tells us that for any $B \in \mathcal{B}'$ and $D < B$ we have

\begin{alignat}{2}
    &\lim_{N\rightarrow\infty}\frac{1}{N}\sum_{n = 1}^N\mathbbm{1}_{E_{D,B}}((T\rtimes M)^n(x,z)) = \mu\times m(E_{D,B})\text{, hence}\\
    &\left|\lim_{N\rightarrow\infty}\frac{1}{N}\sum_{n = 1}^NF((T\rtimes M)^n(x,z))-\int_{X\times[0,1)}Fd\mu\times m\right|\\
    \le&||F(x,y)-g'(x)h(y)||_\infty+||g'(x)h(y)-\sum_{B \in \mathcal{B}'}\sum_{D < B}c_{D,B}\mathbbm{1}_{E_{D,B}}||_\infty+||F(x,y)-g'(x)h(y)||_1\\
    &+\left|\lim_{N\rightarrow\infty}\frac{1}{N}\sum_{n = 1}^N\sum_{B \in \mathcal{B}'}\sum_{D < B}c_{D,B}\mathbbm{1}_{E_{D,B}}((T\rtimes M)^n(x,z))-\int_{X\times[0,1)}g'(x)h(y)d\mu\times m(x,y)\right|\\
    =&\epsilon+4\epsilon||h||_\infty+||\sum_{B \in \mathcal{B}'}\sum_{D < B}c_{D,B}\mathbbm{1}_{E_{D,B}}(x,y)-g'(x)h(y)||_1 < 2\epsilon+5\epsilon||h||_\infty.
\end{alignat}
\end{proof}

Next, we must define a key concept that we will need to  compare $\NQ$ and $\DNQ$. 
Our definition of determinism is motivated by \cite[Lemma 8.9]{SingleOrbitDynamics}.

\begin{definition}
    For a set $A \subseteq \mathbb{N}$, the \textbf{natural upper density} of $A$ is 
    \begin{equation}\label{NaturalUpperDensityEquation}
        \overline{d}(A) := \limsup_{n\rightarrow\infty}\frac{|A\cap[1,N]|}{N}.
    \end{equation}
    If the $\limsup$ in Equation \eqref{NaturalUpperDensityEquation} is a limit, then it is denoted by $d(A)$ and called the \textbf{natural density} of $A$.
    Given a sequence $\omega = (\omega_i)_{i = 1}^\infty \in \mathbb{N}_{\ge 2}^{\mathbb{N}}$, a length $k$, and a set $A \subseteq \mathbb{N}$, let

    \begin{equation}
        B_k(\omega;A) = ([\omega_i,\omega_{i+1},\cdots,\omega_{i+k-1}])_{i \in \mathbb{N}\setminus A},
    \end{equation}
    and let $p_{\epsilon}(k) = p_{\epsilon}(k,\omega)$ denote the (potentially infinite) infimum of the number of distinct elements in $B_k(\omega;A)$ as $A$ runs over all subsets of $\mathbb{N}$ with natural upper density at most $\epsilon$. The sequence $\omega$ is \textbf{deterministic} if the following conditions are satisfied:
    \begin{enumerate}[(i)]
        \item For every $\epsilon > 0$ there exists a $k_0 \in \mathbb{N}$ such that for all $k \ge k_0$ we have

    \begin{equation}
        \frac{\log(p_\epsilon(k,\omega))}{k} < \epsilon.
    \end{equation}

        \item For $E_n := \{i \in \mathbb{N}\ |\ \omega_i = n\}$ we have

        \begin{equation}
            \sum_{n = 2}^\infty-\overline{d}\left(E_n\right)\log\left(\overline{d}(E_n)\right) < \infty.
        \end{equation}
    \end{enumerate}
\end{definition}

We have the following useful characterization of deterministic basic sequences.

\begin{lemma}\label{DeterministicSystemsProduceDeterministicSequences}
    Let $Q = (q_n)_{n = 1}^\infty \in \mathbb{N}_{\ge 2}^\mathbb{N}$ be a basic sequence generated by $(X,\mathscr{B},\mu,T,f,x)$ with $h_\mu\left(\left\{f^{-1}(i)\right\}_{i = 2}^\infty\right) < \infty$. $Q$ is deterministic if and only if $\mathcal{X} := (X,\mathscr{B},\mu,T)$ has zero entropy.
\end{lemma}

\begin{proof}
    For the first direction, let us assume that $\mathcal{X}$ has zero entropy. Let $\epsilon > 0$ be arbitrary. Let $\xi = \left\{f^{-1}(i)\right\}_{i = 2}^\infty$ be the finite entropy partition of $Y$ induced by $f$.
    Let $\xi^n = \bigvee_{i = 0}^{n-1}T^{-i}\xi$ and let $I_\mu(\xi^n)(x') = -\log(\mu(\xi^n_i))$, where $\xi^n_i$ is the cell of $\xi^n$ containing $x'$.
    Since $\mathcal{X}$ has zero entropy, we may assume without loss of generality that all systems in the ergodic decomposition of $\mathcal{X}$ also have zero entropy, so the Shannon-McMillan-Breiman Theorem tells us that $\frac{1}{n}I_\mu(\xi^n)$ converges in measure to $0$ as $n\rightarrow\infty$.
    Let $N_0 \in \mathbb{N}$ be such that for all $N \ge N_0$ we have $\frac{1}{N}I_\mu\left(\xi^N\right) < \epsilon$ on a set $A' = A'(N)$ of measure at least $1-\epsilon$. 
    We see that if $\xi^N_i$ is a cell of $\xi^N$ for which $-\log\left(\mu\left(\xi^N_i\right)\right) < N\epsilon$, then $\mu\left(\xi^N_i\right) > e^{-N\epsilon}$.
    Since $\sum_i\mu\left(\xi^N_i\right) = 1$, we see that there are at most $e^{N\epsilon}$ values of $i$ for which $-\log\left(\mu\left(\xi^N_i\right)\right) < N\epsilon$. 
    Now consider $A := \{n \in \mathbb{N}\ |\ T^nx \notin A'\}$. 
    Since $x$ is a generic point for $f$, we see that $d(A) = \mu\left(X\setminus A'\right) < \epsilon$.
    Furthermore, we see that the only words $w := [w_1,w_2,\cdots,w_N]$ appearing in $B_N(\omega;A)$ are those for which

    \begin{equation}
        \xi^N_w := \bigcap_{i = 0}^{N-1}T^{-i}f^{-1}(w_{i+1}) = E_w
    \end{equation}
    satisfies $-\log\left(\mu\left(\xi^N_w\right)\right) < N\epsilon$, hence $\log(p_\epsilon(N,\omega)) < N\epsilon$, so $Q$ satisfies condition (i) of being deterministic. 
    To see that $Q$ also satisfies condition (ii), it suffices to observe that

    \begin{equation}
        d(E_n) = \lim_{N\rightarrow\infty}\frac{1}{N}\left|\left\{1 \le i \le N\ |\ f\left(T^iy\right) = n\}| = \mu\left(f^{-1}(\{m\})\right)\right\}\right|.
    \end{equation}
    
    For the next direction, let us assume that $Q$ is deterministic. 
    It suffices to show that the system $\mathcal{X}$ constructed in the proof of Theorem \ref{WhenIsASequenceDynamicallyGenerated} has zero entropy, so we will also continue using the same notation. Let $\epsilon > 0$ be arbitrary.
    Let $\xi = \left\{f^{-1}(i)\right\}_{i = 2}^\infty = \left\{C_{[i]}\right\}_{i = 2}^\infty$ and observe that $\xi^{n+1} = \bigvee_{i = 0}^nT^{-i}\xi = \{C_w\}_{w \in \mathbb{N}_{\ge 2}^{n+1}}$. 
    Let $\xi(n) = \left\{f^{-1}(i)\right\}_{i = 2}^n\cup\left\{f^{-1}(\mathbb{N}_{> n})\right\}$, and let $N \in \mathbb{N}$ be such that for $\eta = \xi(N)$ we have $H_\mu(\xi|\eta) < \epsilon$.
    Since $\xi$ is a generating partition for $\mathscr{B}$, we see that
    \begin{equation}
        h_\mu(T) = h_\mu(T,\xi) \le h_\mu(T,\eta)+H_\mu(\xi|\eta) < h_\mu(T,\eta)+\epsilon,
    \end{equation}
    so it suffices to show that $h_\mu(T,\eta) = 0$. 
    Let $\omega_i = q_i$ if $q_i \le N$ and $\omega_i = N+1$ otherwise. Since $B_k(\omega;A) \le B_k(Q;A)$ for all $A \subseteq \mathbb{N}$, we see that $\omega$ is a deterministic sequence. 
    For $2 \le i \le N$, let $C_{[i]}' = C_{[i]}$, let $C_{[N+1]}' = \bigcup_{i = N+1}^\infty C_{[i]}$, and for $w = [w_1,w_2,\cdots,w_k] \in [2,N+1]^k$, let

    \begin{equation}
        C_w' = \bigcap_{i = 1}^kC_{[w_i]}'.
    \end{equation}
    We see that $\eta^n = \bigvee_{i = 0}^{n-1}T^{-i}\eta = \left\{C_w'\right\}_{w \in [2,N+1]^n}$. 
    Let $\epsilon_2 > 0$ be arbitrary and let $k_0 \in \mathbb{N}$ be such that for $k \ge k_0$ we have $|B_k(\omega;A_k)| < e^{k\epsilon_2}$ for some $A_k \subseteq \mathbb{N}$ with $\overline{d}(A_k) < \epsilon_2$. 
    We see that

    \begin{alignat*}{2}
        &H\left(\eta^k\right) = \sum_{w \in [2,N+1]^k}-\mu\left(C_w'\right)\log\left(\mu\left(C_w'\right)\right)\\
        = &\sum_{w \in B_k(\omega;A)}-\mu\left(C_w'\right)\log\left(\mu\left(C_w'\right)\right)+\sum_{w \in B_k(\omega;A)^c}-\mu\left(C_w'\right)\log\left(\mu\left(C_w'\right)\right).
    \end{alignat*}
    Since $\mu(C_w') = d(\{n \in \mathbb{N}\ |\ (\omega_n,\omega_{n+1},\cdots,\omega_{n+k-1}) = w\})$, we see that
    \begin{equation}
        B = B(k) := \sum_{w \in B_k(\omega;A)}\mu(C_w') > 1-\epsilon_2.
    \end{equation}
    Recalling that $x\mapsto -x\log(x)$ is convex, we see that

    \begin{alignat*}{2}
        H\left(\eta^k\right) &\le B\log\left(\frac{|B_k(\omega;A)|}{B}\right)+(1-B)\log\left(\frac{(N+1)^k-|B_k(\omega;A)|}{1-B}\right)\\
        &\le B\log\left(\frac{e^{k\epsilon_2}}{B}\right)+(1-B)\log\left(\frac{(N+1)^k}{1-B}\right)\text{, hence}\\
        h_\mu(T,\eta) &= \liminf_{k\rightarrow\infty}\frac{1}{k}H\left(\eta^k\right) \le \liminf_{k\rightarrow\infty}\frac{1}{k}\left(B(k)\log\left(\frac{e^{k\epsilon_2}}{B(k)}\right)+(1-B(k))\log\left(\frac{(N+1)^k}{1-B(k)}\right)\right)\\
        &\le \epsilon_2+\lim_{B(k)\rightarrow1}\left((1-B(k))\log(N+1)-\frac{1}{k}(1-B(k))\log(1-B(k))\right) = \epsilon_2. 
    \end{alignat*}
\end{proof}

\begin{theorem}\label{MeasurablyProducingDGBS}
    If $\mathcal{Y} = (Y,\mathscr{A},\nu,S)$ is a m.p.s., $f:Y\rightarrow\mathbb{N}_{\ge 2}$ is measurable, and $y \in Y$ is a generic point for $f$, then the basic sequence $Q = (q_n)_{n = 1}^\infty$ given by $q_n = f(T^ny)$ is a dynamically generated basic sequence. Furthermore, if $h_\nu(\{f^{-1}(n)\}_{n = 2}^\infty) < \infty$ and $\mathcal{Y}$ has zero entropy, then $Q$ is deterministic.
\end{theorem}

\begin{proof}
    Let us first verify that we can apply Lemma \ref{WhenIsASequenceDynamicallyGenerated} to show that $Q$ is dynamically generated. To see that condition (i) holds, it suffices to observe that for any word $w = (w_1,w_2,\cdots,w_\ell) \in \mathbb{N}_{\ge 2}^\ell$ and $n \in \mathbb{N}$, we have
    \begin{alignat*}{2}
        &w_i = q_{n+i} = f\left(T^{n+i}y\right)\ \forall\ 1 \le i \le \ell \Leftrightarrow T^ny \in \bigcap_{i = 1}^\ell T^{-i}f^{-1}(w_i) = E_w\text{, hence}\\
        d(w) = &\lim_{N\rightarrow\infty}\frac{1}{N}\left|\left\{1 \le n \le N\ |\ w_i = q_{n+i} = f(T^{n+i}y)\ \forall\ 1 \le i < \ell\right\}\right|\\
        =&\lim_{N\rightarrow\infty}\frac{1}{N}\left|\left\{1 \le n \le N\ |\ T^ny \in E_w\right\}\right| = \nu(E_w).
    \end{alignat*}
    Condition (ii) holds due to the standing assumption that if $E_w \neq \emptyset$ then $\nu(E_w) > 0$. Condition (iii) holds due to the fact that $d(w) = \nu(E_w)$ and $\{E_w\ |\ w \in \mathbb{N}_{\ge 2}^\ell\}$ is a partition of $Y$. To show that $Q$ is deterministic when $\mathcal{Y}$ has zero entropy, we proceed as we did in the proof of Lemma \ref{DeterministicSystemsProduceDeterministicSequences}.
\end{proof}

\begin{corollary}\label{DGBSOnTori}
    Suppose that $\mathcal{X} := \left(X,\mathscr{B},\mu,T\right)$ is a c.m.p.s., $f:\mathbb{T}^d\rightarrow\mathbb{N}_{\ge 2}$ is Jordan measurable\footnote{This means that for any $n \in \mathbb{N}_{\ge 2}$, the boundary of the set $f^{-1}(\{n\})$ has 0 measure.}, and $x \in X$ is a generic point. 
    Then the sequence $Q = (q_n)_{n = 1}^\infty$ given by $q_n = f\left(T^nx\right)$ is dynamically generated.
\end{corollary}

While our next result will not be used later on, we record it because it is of independent interest.

\begin{theorem}\label{NormalityPreservationUnderRationalMultiplication}
    If $Q$ is a dynamically generated basic sequence generated by $(X,\mathscr{B},\mu,T,f,x)$ with $(X,\mathscr{B},\mu,T)$ having finite entropy, $\int_X\log(f)d\mu < \infty$, $\frac{r}{s}$ is a non-zero rational number, and $y \in \mathcal{UDN}(Q)$, then $\frac{r}{s}y \in \mathcal{UDN}(Q)$.
\end{theorem}

\begin{proof}
    We recall that for each $n \in \mathbb{N}_{\ge 2}$, $m$ is the unique measure of maximal entropy for the map $M_n:[0,1)\rightarrow[0,1)$. 
    Consequently, the Abramov-Rokhlin formula (see \cite{NonInvertibleAbramovRokhlinFormula}) tells us that $\mu\times m$ is the unique measure of maximal entropy for $T\rtimes M$ on $X\times[0,1)$ among all measures $\lambda$ that project to $\mu$ on $X$.
    Now let $\lambda$ be any weak$^*$ limit point of the sequence $\left\{\frac{1}{N}\sum_{n = 1}^N\delta_{(T\rtimes M)^n(x,\frac{r}{s}y)}\right\}_{N = 1}^\infty$, and observe that $\lambda$ projects to $\mu$ on $X$.
    Since $y \in \mathcal{UDN}(Q)$, $\mathcal{X}^f$ is a factor of the system $\mathcal{X}^f_1 := (X\times[0,1),\mathscr{B},\lambda,T\rtimes M)$ via the factor map $\pi(w_1,w_2) = (w_1,sw_2)$.
    Since the entropy of $\lambda$ is at least as large as that of $\mu\times m$, we see that $\lambda = \mu\times m$, which yields the desired result.
\end{proof}
\section{Hot Spot Theorems and equivalent notions of normality}\label{HotSpotEquivalentSection}

One of the goals of this article is to prove an analogue of Theorem \ref{StrongestHotSpot} for deterministic dynamically generated basic sequences. In the proof of Theorem \ref{StrongestHotSpot} Bergelson and Vandehey made use of Hoeffding's inequality, but in our more general setting involving basic sequences, we require McDiarmid's inequality, which is a generalization of Hoeffding's inequality. We now record the special case McDiarmid's Inequality \cite{McDiardmidsInequality} that we will need later on.

\begin{lemma}
    Let $f:\{0,1\}^n\rightarrow\mathbb{R}$ be such that for any $x_1,\cdots,x_{i-1},x_i,x_i',x_{i+1},\cdots,x_n \in \{0,1\}^n$ we have

    \begin{equation}
        \left|f(x_1,\cdots,x_{i-1},x_i,x_{i+1},\cdots,x_n)-f\left(x_1,\cdots,x_{i-1},x_i',x_{i+1},\cdots,x_n\right)\right| \le \frac{1}{n}.
    \end{equation}
    If $X_1,\cdots,X_n$ are random variables on a probability space $(\Omega,\mathbb{P})$ taking values in $\{0,1\}$, then 

    \begin{equation}
        \mathbb{P}(|f(X_1,\cdots,X_n)-\mathbb{E}[f(X_1,\cdots,X_n)]| > \epsilon) \le 2\text{exp}\left(-2n\epsilon^2\right).
    \end{equation}
\end{lemma}

We would once again like to thank Yuval Peres for providing us with the proof of Lemma \ref{LemmaUsingMcdiarmid}.

\begin{lemma}\label{LemmaUsingMcdiarmid}
    Fix $\epsilon > 0$ and $N \in \mathbb{N}$. Let $K = \lceil-\log_2(\epsilon)\rceil+2$ and let $(q_n)_{n = 1}^{KN} \in \mathbb{N}_{\ge 2}^{KN}$ be arbitrary. 
    For $1 \le n \le KN$ let $M_n = \prod_{m = 1}^nq_m$. For $\epsilon > 0$ and $(a,b) \subseteq [0,1)$, let

    \begin{equation*}
        E = E\left(\epsilon,a,b,(q_n)_{n = 1}^{KN}\right) = \left\{x \in [0,1)\ |\ \left|\frac{1}{KN}\sum_{n = 1}^{KN}\mathbbm{1}_{(a,b)}(M_nx)-(b-a)\right| > \epsilon\right\}.
    \end{equation*}
    There exists a set $E' \subseteq [0,1)$ for which $E \subseteq E'$, $E'$ is a union of intervals of length $M_{KN}^{-1}$, and

    \begin{equation*}
        m(E') \le 4K\text{exp}\left(-\frac{1}{2}N\epsilon^2\right).
    \end{equation*}
\end{lemma}

\begin{proof}
    For $n \in \mathbb{N}$ and $1 \le k \le K$, let $M_{n,k} = q_{Kn-K+k}q_{Kn-K+k+1}\cdots q_{Kn+k-1}$, let $P_{n,k}$ denote the partition of $[0,1)$ into consecutive half-open intervals of length $M_{n,k}^{-1} \le 2^{-K} \le \frac{\epsilon}{4}$, and let $A_{n,k},B_{n,k}$ be unions of elements of $P_{n,k}$ satisfying $A_{n,k} \subseteq F \subseteq B_{n,k}$ and $m(B_{n,k}\setminus A_{n,k}) < \frac{\epsilon}{2}$. 
    Consider the random variables $(X_{n,k})_{n = 1}^\infty$ and $(Y_{n,k})_{n = 1}^\infty$ given by $X_{n,k}(x) = \mathbbm{1}_{A_{n,k}}(M_{Kn-K+k}x)$ and $Y_{n,k}(x) = \mathbbm{1}_{B_{n,k}}(M_{Kn-K+k}x)$, so $X_{n,k}$ and $Y_{n,k}$ are each checking if the digit in position $Kn-K+k$ in the base $Q$ expansion of $x$ belongs to a given set. 
    We observe that for each $1 \le k \le K$, the sequences $(X_{n,k})_{n = 1}^N$ and $(Y_{n,k})_{n = 1}^N$ are independent sequences of random variables, and they each take values in $\{0,1\}$. Let $f_N(x_1,\cdots,x_N) = \frac{1}{N}\sum_{n = 1}^Nx_n$ and let
    
    \begin{alignat*}{2}
        E' = &\bigcup_{k = 1}^K\bigg(\left\{\left|f_N(X_{1,k},\cdots,X_{N,k})-\mathbb{E}[f(X_{1,k},\cdots,X_{N,k})]\right| > \frac{\epsilon}{2}\right\}\\
        &\ \ \qquad\cup\left\{\left|f_N(Y_{1,k},\cdots,Y_{N,k})-\mathbb{E}[f(Y_{1,k},\cdots,Y_{N,k})]\right| > \frac{\epsilon}{2}\right\}\bigg)
    \end{alignat*}
    We now use McDiarmid's inequality to see that for each $1 \le k \le K$ we have

    \begin{alignat*}{2}
        &m\left(\left|f_N(X_{1,k},\cdots,X_{N,k})-\mathbb{E}[f(X_{1,k},\cdots,X_{N,k})]\right| > \frac{\epsilon}{2}\right) < 2\text{exp}\left(-\frac{1}{2}N\epsilon^2\right),\\
        &m\left(\left|f_N(Y_{1,k},\cdots,Y_{N,k})-\mathbb{E}[f(Y_{1,k},\cdots,Y_{N,k})]\right| > \frac{\epsilon}{2}\right) < 2\text{exp}\left(-\frac{1}{2}N\epsilon^2\right),\\
        &f_N(X_{1,k}(x),\cdots,X_{N,k}(x)) \le \frac{1}{N}\sum_{n = 1}^N\mathbbm{1}_{(a,b)}(M_{Kn-K+k}x) \le f_N(Y_{1,k}(x),\cdots,Y_{N,k}(x)),\\
        &\mathbb{E}[f_N(Y_{1,k},\cdots,Y_{N,k})]-\mathbb{E}[f_N(X_{1,k},\cdots,X_{N,k})] < \frac{\epsilon}{2}\text{, and}\\
        &\mathbb{E}[f(X_{1,k},\cdots,X_{N,k})] < b-a < \mathbb{E}[f(Y_{1,k},\cdots,Y_{N,k})]\text{, hence}\\
        &m\left(\left\{y \in [0,1)\ |\ \left|\frac{1}{N}\sum_{n = 1}^N\mathbbm{1}_{(a,b)}(M_{Kn-K+k}y)-(b-a)\right| > \epsilon\right\}\right)\\
        \le& m\bigg(\bigg\{y \in [0,1)\ |\ \text{max}\bigg(\bigg|\frac{1}{N}\sum_{n = 1}^NX_{Kn-K+k}(y)-\mathbb{E}[f(Y_{1,k},\cdots,Y_{N,k})]\bigg|,\\
        &\textcolor{white}{m\bigg(\bigg\{y \in [0,1)\ |\ \text{max}\bigg(}\bigg|\frac{1}{N}\sum_{n = 1}^NY_{Kn-K+k}(y)-\mathbb{E}[f(X_{1,k},\cdots,X_{N,k})]\bigg|\bigg) > \epsilon\bigg\}\bigg)\\
        \le &m\bigg(\left\{\left|f_N(X_{1,k},\cdots,X_{N,k})-\mathbb{E}[f(X_{1,k},\cdots,X_{N,k})]\right| > \frac{\epsilon}{2}\right\}\\
        &\textcolor{white}{m\bigg(}\cup\left\{\left|f_N(Y_{1,k},\cdots,Y_{N,k})-\mathbb{E}[f(Y_{1,k},\cdots,Y_{N,k})]\right| > \frac{\epsilon}{2}\right\}\bigg) \le 4\text{exp}\left(-\frac{1}{2}N\epsilon^2\right)\text{, and}\\
        &m\left(\left\{y \in [0,1)\ |\ \left|\frac{1}{KN}\sum_{n = 1}^{KN}\mathbbm{1}_{(a,b)}(M_ny)-(b-a)\right| > \epsilon\right\}\right) \le m(E') < 4K\text{exp}\left(-\frac{1}{2}N\epsilon^2\right).
    \end{alignat*}
    Lastly, we observe that $f_N(X_{1,k}(x),\cdots,X_{N,k}(x))$ and $f_N(Y_{1,k}(x),\cdots,Y_{N,k}(x))$ are completely determined by the first $KN$ digits of the base $Q$ expansion of $x$, which is why $E'$ is a union of intervals of length $M_{KN}^{-1}$. 
\end{proof}

For the proofs of our next two results, we use big O notation as follows.
We write $O(x)$ to denote a quantity $y$ satisfying $y \le Cx$ for some absolute constant $C$.
Usually we will write an equation such as $a+b = c+O(x)$, which is another way of saying $|a+b-c| < Cx$.

\begin{theorem}\label{DN(Q)ImpliesUDN(Q)}
    If $Q = (q_n)_{n = 1}^\infty$ is a deterministic dynamically generated basic sequence, then $\mathcal{DN}(Q) = \mathcal{UDN}(Q)$.
\end{theorem}

\begin{proof}
    It is clear that $\mathcal{UDN}(Q) \subseteq \mathcal{DN}(Q)$, so we need to show that $\mathcal{DN}(Q) \subseteq \mathcal{UDN}(Q)$.
    Let $Q$ be generated by $(X,\mathscr{B},\mu,T,f,x)$ and let us fix some $y \in \mathcal{DN}(Q)$.
    Let us fix a block of bases $B = (b_1,\cdots,b_\ell) = (q_n,\cdots,q_{n+\ell-1})$ for some $n \in \mathbb{N}$, and let us also fix a block of digits $D = (d_1,\cdots,d_\ell) \in \mathbb{N}_0^\ell$ with $D < B$.
    Using Theorem \ref{SkewProductCharacterizationsOfNormalities}(ii), it suffices to show that

    \begin{equation}
        \lim_{N\rightarrow\infty}\frac{1}{N}\sum_{n = 1}^N\mathbbm{1}_{E_{D,B}}((T\rtimes M)^n(x,y)) = \mu\times m(E_{D,B}) = \frac{\mu(E_B)}{b_1\cdots b_\ell}.
    \end{equation}
    Let $\epsilon = \epsilon(\ell) > 0$ be sufficiently small, let $K = \lceil-\log_2(\epsilon)\rceil+2$ and let $\delta = \frac{1}{32K}\epsilon^3$. 
    As in the proof of Theorem \ref{DeterministicSystemsProduceDeterministicSequences}, let $L_0 \in \mathbb{N}$ be so large that  for $L_1 := \ell+\left\lceil\frac{1}{\epsilon}KL_0\right\rceil$ the partition $\xi^{L_1}$ has a collection of cells $\left\{\xi_w^{L_1}\right\}_{w \in W_{L_1}}$\footnote{We remark that $E_B = \xi_B^\ell$, but we will continue to use both notations for the sake of presentation.} such that $\mu\left(\xi_w^{L_1}\right) > e^{-L_1\delta}$ for all $w \in W_{L_1}$, $|W_{L_1}| < e^{L_1\delta}$, $\mu\left(\bigcup_{w \in W_{L_1}}\xi_w^{L_1}\right) > 1-\delta$, and

    \begin{equation}
        2K\text{exp}\left(L_1\delta-\frac{1}{16}L_0\epsilon^2\right) < \frac{1}{2}\epsilon.
    \end{equation}
    Furthermore, we may assume without loss of generality that $L_0 > \frac{4}{\epsilon}$.
    Let $L_2 = L_1-\ell = \left\lceil\frac{1}{\epsilon}KL_0\right\rceil$.
    We see that 

    \begin{alignat*}{2}
        & \frac{1}{N}\sum_{n = 1}^N\mathbbm{1}_{E_{D,B}}\left((T\rtimes M)^n(x,y)\right) = \frac{1}{N}\sum_{n = 1}^N\frac{1}{L_2}\sum_{l = 1}^{L_2}\mathbbm{1}_{E_{D,B}}\left((T\rtimes M)^{n+l}(x,y)\right)+O\left(\frac{L_2}{N}\right).
    \end{alignat*}
    We now let $[1,N] = A_N\cup B_N\cup C_N$ as follows. 
    Let $A_N$ be the set of $n \in [1,N]$ for which

    \begin{equation}
        \frac{1}{L_2}\sum_{l = 1}^{L_2}\mathbbm{1}_{E_B}\left(T^{n+l}x\right) < \frac{KL_0}{L_2} < \frac{\epsilon}{2},
    \end{equation}
    let $B_N$ be the set of $n \in [1,N]\setminus A_N$ for which
    
    \begin{equation}
        \left|\frac{1}{L_2}\sum_{l = 1}^{L_2}\mathbbm{1}_{E_{D,B}}\left((T\rtimes M)^{n+l}(x,y)\right)-\frac{1}{L_2b_1\cdots b_\ell}\sum_{l = 1}^{L_2}\mathbbm{1}_{E_B}\left(T^{n+l}x\right),\right| \le \epsilon,
    \end{equation}
    and let $C_N = [1,N]\setminus(A_N\cup B_N)$. 
    Writing $a\overset{\epsilon}{=}b$ to denote $|a-b| \le \epsilon$, we see that

    \begin{alignat*}{2}
        S(A,N) :=&\frac{1}{N}\sum_{n \in A_N}\frac{1}{L_2}\sum_{l = 1}^{L_2}\mathbbm{1}_{E_{D,B}}\left((T\rtimes M)^{n+l}(x,y)\right) \overset{\epsilon}{=} \frac{1}{N}\sum_{n \in A_N}\frac{1}{L_2b_1\cdots b_\ell}\sum_{l = 1}^{L_2}\mathbbm{1}_{E_B}\left(T^{n+l}x\right),\\
        S(B,N) :=&\frac{1}{N}\sum_{n \in B_N}\frac{1}{L_2}\sum_{l = 1}^{L_2}\mathbbm{1}_{E_{D,B}}\left((T\rtimes M)^{n+l}(x,y)\right) \overset{\epsilon}{=} \frac{1}{N}\sum_{n \in B_N}\frac{1}{L_2b_1\cdots b_\ell}\sum_{l = 1}^{L_2}\mathbbm{1}_{E_B}\left(T^{n+l}x\right)\text{, and}\\
        S(C,N) :=&\frac{1}{N}\sum_{n \in C_N}\frac{1}{L_2}\sum_{l = 1}^{L_2}\mathbbm{1}_{E_{D,B}}\left((T\rtimes M)^{n+l}(x,y)\right) \le \frac{|C_N|}{N}.
    \end{alignat*}
    
    Now we will show that $\limsup_{N\rightarrow\infty}\frac{|C_N|}{N} < \epsilon$.
    Let $\mathcal{B}_{L_1} = \bigcup_{w \notin W_{L_1}}\xi_w^{L_1}$, and observe that $\mu\left(\mathcal{B}_{L_1}\right) < \delta$. 
    We see that for any $m \in \mathbb{N}$, the set $\mathcal{L}_m := \left\{1 \le l \le L_2\ |\ T^{l+m}x \in E_B\right\}$ is determined by the $w = w(m) \in \mathbb{N}_{\ge 2}^{L_1}$ for which $T^mx \in \xi_w^{L_1}$.
    It follows that whether or not a given $m \in [1,N]$ satisfies $m \in A_N$ is also determined by $w(m)$, so let $\mathcal{A}_{L_1} \in \mathscr{B}$ be such that $m \in A_N$ if and only if $\xi_{w(m)}^{L_1} \subseteq \mathcal{A}_{L_1}$.
    We see that for $D_N := \left\{1 \le n \le N\ |\ \xi_{w(n)}^{L_1} \subseteq \mathcal{B}_{L_1}\right\} = \left\{1 \le n \le N\ |\ T^nx \in \mathcal{B}_{L_1}\right\}$, we have $\lim_{N\rightarrow\infty}\frac{|D_N|}{N} = \mu\left(\mathcal{B}_{L_1}\right) < \delta$, so it suffices to show that $\limsup_{N\rightarrow\infty}\frac{|C_N\setminus D_N|}{N} < \frac{1}{2}\epsilon < \epsilon-\delta$. 
    Now let us suppose that $m \notin A_N$, so $|\mathcal{L}_m| \ge KL_0$. 
    Let $\mathcal{L}_m' \subseteq \mathcal{L}_m$ be such that $K$ divides $|\mathcal{L}_m'|$ and $|\mathcal{L}_m\setminus \mathcal{L}_m'| < K$.
    Using Lemma \ref{LemmaUsingMcdiarmid}, we see if $T^mx \in \xi_{w(m)}^{L_1}$, then

    \begin{alignat*}{2}
        &m\left(\underbrace{\left\{z\ |\ \left|\frac{1}{L_2}\sum_{l = 1}^{L_2}\mathbbm{1}_{E_{D,B}}\left((T\rtimes M)^l(T^mx,z)\right)-\frac{1}{L_2b_1\cdots b_\ell}\sum_{l = 1}^{L_2}\mathbbm{1}_{E_B}\left(T^{l+m}x\right)\right| > \epsilon\right\}}_{:= H_{w(m)}}\right)\\
        =&m\left(z\ |\ \left|\frac{1}{|\mathcal{L}_m|}\sum_{l \in \mathcal{L}_m}\mathbbm{1}_{E_{D,B}}\left((T\rtimes M)^l(T^mx,z)\right)-\frac{1}{|\mathcal{L}_m|b_1\cdots b_\ell}\sum_{l \in \mathcal{L}_m}\mathbbm{1}_{E_B}\left(T^{l+m}x\right)\right| > \frac{L_2}{|\mathcal{L}_m|}\epsilon\right)\\
        =&m\left(z\ |\ \left|\frac{1}{|\mathcal{L}_m|}\sum_{l \in L_m}\mathbbm{1}_{X\times I_{D,B}}\left((T\rtimes M)^l(T^mx,z)\right)-\frac{1}{b_1\cdots b_\ell}\right| > \frac{L_2}{|\mathcal{L}_m|}\epsilon\right)\\
        \le&m\left(z\ |\ \left|\frac{1}{|\mathcal{L}_m'|}\sum_{l \in \mathcal{L}_m'}\mathbbm{1}_{X\times I_{D,B}}\left((T\rtimes M)^l(T^mx,z)\right)-\frac{1}{b_1\cdots b_\ell}\right| > \frac{L_2}{|\mathcal{L}_m|}\epsilon-2\frac{K-1}{|\mathcal{L}_m|}\right)\\
        \le&m\left(z\ |\ \left|\frac{1}{|\mathcal{L}_m'|}\sum_{l \in \mathcal{L}_m'}\mathbbm{1}_{X\times I_{D,B}}\left((T\rtimes M)^l(T^mx,z)\right)-\frac{1}{b_1\cdots b_\ell}\right| > \frac{1}{2}\epsilon\right)\\
        < &2K\text{exp}\left(-\frac{1}{2}\frac{|\mathcal{L}_m'|}{K}\left(\frac{\epsilon}{2}\right)^2\right) \le 2K\text{exp}\left(-\frac{1}{8}L_0\epsilon^2\right).
    \end{alignat*}
    If $m \in C_N\setminus D_N$, then $w(m) \in W_{L_1}$, $m \notin A_N$, and $(T\rtimes M)^m(x,y) \in \xi_{w(m)}^{L_1}\times H_{w(m)}$. Since $y \in \mathcal{DN}(Q)$ and each $H_{w(m)}$ is a finite union of intervals, we see that

    \begin{alignat}{2}
        &\limsup_{N\rightarrow\infty}\frac{|C_N\setminus D_N|}{N} \le \limsup_{N\rightarrow\infty}\left|\bigcup_{w \in W_{L_1}\setminus\mathcal{A}_{L_1}}\left\{1 \le m \le N\ |\ (T\rtimes M)^m(x,y) \in \xi_w^{L_1}\times H_{w(m)}\right\}\right|\\
        \le&\limsup_{N\rightarrow\infty}\left|\bigcup_{w \in W_{L_1}\setminus\mathcal{A}_{L_1}}\left\{1 \le m \le N\ |\ (T\rtimes M)^n(x,y) \in X\times H_{w(m)}\right\}\right|\\
        \le& \sum_{w \in W_{L_1}\setminus\mathcal{A}_{L_1}}m(H_{w(m)}) \le |W_{L_1}\setminus\mathcal{A}_{L_1}|\cdot2K\text{exp}(-\frac{1}{8}L_0\epsilon^2) \le 2K\text{exp}\left(L_1\delta-\frac{1}{8}L_0\epsilon^2\right) < \frac{1}{2}\epsilon.
    \end{alignat}
    Since we have shown that $\limsup_{N\rightarrow\infty}\frac{|C_N|}{N} < \epsilon$, we see that
    
    \begin{equation}
        \frac{1}{N}\sum_{n \in C_N}\frac{1}{L_2}\sum_{l = 1}^{L_2}\mathbbm{1}_{E_{D,B}}\left((T\rtimes M)^{n+l}(x,y)\right) \overset{2\epsilon}{=} \frac{1}{N}\sum_{n \in C_N}\frac{1}{L_2b_1\cdots b_\ell}\sum_{l = 1}^{L_2}\mathbbm{1}_{E_B}\left(T^{n+l}x\right).
    \end{equation}
    Putting everything together, we see that

    \begin{alignat*}{2}
        &\frac{1}{N}\sum_{n = 1}^N\mathbbm{1}_{E_{D,B}}((T\rtimes M)^n(x,y)) = S(A,N)+S(B,N)+S(C,N)+O\left(\frac{L_2}{N}\right)\\
        \overset{4\epsilon}{=}&\frac{1}{NL_2b_1\cdots b_\ell}\sum_{n = 1}^N\sum_{l = 1}^{L_2}\mathbbm{1}_{E_B}\left(T^{n+l}x\right)+O\left(\frac{L_2}{N}\right)\\
        =& \frac{1}{Nb_1\cdots b_\ell}\sum_{n = 1}^N\mathbbm{1}_{E_B}(T^nx)+O\left(\frac{L_2}{N}\right)\overset{N\rightarrow\infty}{\longrightarrow}\frac{\mu(E_B)}{b_1\cdots b_\ell}.
    \end{alignat*}
\end{proof}

Theorem~\ref{StrongestHotSpotForDDGBasicSequences} can be viewed as a generalization of the Bergelson-Vandehey Hot Spot Theorem to a large class of deterministic dynamically generated basic sequences.
We recall that for any dynamically generated $Q$, we have $\mathcal{UN}(Q) = \mathcal{UDN}(Q)$, and that $P_D = \displaystyle\lim_{n\rightarrow\infty}\frac{Q_n(D)}{n}$.

\begin{proof}[Proof of Theorem~\ref{StrongestHotSpotForDDGBasicSequences}]
    We begin by proving (i). Using Theorems \ref{DN(Q)ImpliesUDN(Q)}, it suffices to show that $y \in \mathcal{DN}(Q)$.
    Let $\nu_{(a,b)}(y,N,w) := \nu_{(a,b)}(y,N,w;\emptyset)$, and let us fix some $0 \le c < d \le 1$ with the intent of showing that $\lim_{N\rightarrow\infty}\nu_{(c,d)}(y,N,Q) = d-c$. 
    Let $\epsilon > 0$ be arbitrary.
    Let $W > 0$ be such that $W^2\epsilon^2 > 8\left\lceil-\log_2(W\epsilon)\right\rceil+16$ and let $K := \left\lceil-\log_2(W\epsilon)\right\rceil+2$. 
    Since $Q$ is deterministic, let $L_0 \in \mathbb{N}$ be such that $p_{\epsilon}(KL,Q) < \text{exp}\left(L\epsilon^2\right)$ for all $L \ge L_0$, let $A_L$ have upper density at most $\frac{1}{3}\epsilon$ and be such that $|B_{KL}(Q;A_L)| < \text{exp}\left(L\epsilon^2\right)$.
    We observe that 
    
    \begin{alignat*}{2}
        &\mu\left(\bigcup_{B \in B_{KL}(Q;A_L)}E_B\right) = d(\{n \in \mathbb{N}\ |\ T^nx \in E_B\ \&\ B \in B_{KL}(Q;A_L)\})\\
        \ge &d(\{n \in \mathbb{N}\setminus A_L\ |\ (q_n,\cdots,q_{n+KL-1}) \in B_{KL}\left(Q;A_L\right)\}) = d\left(A_L^c\right) \ge 1-\frac{1}{3}\epsilon.
    \end{alignat*}
    We recall from Lemma \ref{LemmaUsingMcdiarmid} that for $w \in \mathbb{N}_{\ge 2}^{KL}$ and $E'(w) = E'(W\epsilon,c,d,w)$ we have $m\left(E'(w)\right) \le 4K\text{exp}\left(-\frac{1}{2}LW^2\epsilon^2\right)$.
    Since $\int_X\log(f)d\mu < \infty$, let $M > 0$ be such that $\mu\left(\left\{\tilde{f} \le M\right\}\right) > 1-\frac{1}{3}\epsilon$.
    Pick $\sigma \in (0,1)$ such that $(1-\sigma)(M+1) < 2-\sigma$.
    We can assume without loss of generality that $L$ is so large that $4K\text{exp}\big(-KL\epsilon^2\big) < \epsilon$ and that $2^{-L} < \delta$. 
    Let $\tilde{f} := \mathbb{E}[\log(f)|\mathcal{I}_T]$, i.e., $\tilde{f}$ is the conditional expectation of $f$ with respect to the $\sigma$-algebra $\mathcal{I}_T \subseteq \mathscr{B}$ of $T$-invariant sets. 
    Using The Mean Ergodic Theorem, we see that
    
    \begin{equation}
        \lim_{N\rightarrow\infty}\frac{1}{N}\sum_{n = 1}^NT^n\log(f) = \tilde{f},
    \end{equation}
    with convergence taking place in $L^1(X,\mu)$, hence convergence also takes place in measure, so we may also assume without loss of generality that $L$ is so large that

    \begin{equation}\label{EquationUsingErgodicityAndFiniteLnIntegral}
        \left|\frac{1}{KL}\sum_{n = 1}^{KL}\log(f(T^nx))-\tilde{f}(x)\right| < 1,
    \end{equation}
    on a set of measure at least $1-\frac{1}{3}\epsilon$.
    Furthermore, since the value of $f_L(x) := \frac{1}{KL}\sum_{n = 1}^{KL}\log(f(T^nx))$ depends only on the $B \in \mathbb{N}_{\ge 2}^{KL}$ for which $x \in E_B$, there is a collection $\mathcal{B} \subseteq \mathbb{N}_{\ge 2}^{KL}$ for which $X_\mathcal{B} := \left\{\left|f_L-\tilde{f}\right| < 1\right\} = \bigcup_{B \in \mathcal{B}}E_B$.
    It follows that for $\mathcal{B}' = B_{KL}(Q;A_L)\cap\mathcal{B}$ and $X_{\mathcal{B}'} := \bigcup_{B \in \mathcal{B}'}E_B$ we have $\mu(X_{\mathcal{B}'}) > 1-\epsilon$.
    We also recall from Lemma \ref{LemmaUsingMcdiarmid} that for each $w \in \mathbb{N}^{KL}$ the set $E'(w) = E'(\eta,c,d,w)$ is a union of intervals $\left\{I_{D_j,w}\right\}_{j = 1}^{J_w}$, and we see that if $w \in \mathcal{B}'$ and $x_w \in E_w$, then
    
    \begin{equation*}
        m\left(I_{D_j,w}\right) = \text{exp}\left(-KLf_L(x_w)\right) \in \left(\text{exp}\left(-KL\left(\tilde{f}(x_w)+1\right)\right),\text{exp}\left(-KL\left(\tilde{f}(x_w)-1\right)\right)\right).
    \end{equation*}
    We see that for $w \in \mathcal{B}'$ and $x_w \in E_w$ we have

    \begin{alignat*}{2}
        &J_w\text{exp}\left(-KL\left(\tilde{f}(x_w)+1\right)\right) \le m\left(E'(w)\right) \text{, hence }\\
        &J_w \le 4K\text{exp}\left(KL\left(\tilde{f}(x_w)+1\right)-\frac{1}{2}LW^2\epsilon^2\right).
    \end{alignat*}
    
    We now use \eqref{GeneralizedDDGHotspotAssumption} to pick a $\mathcal{N}_\sigma \subseteq \mathbb{N}$ with upper density at most $1-\sigma$ and pick $N$ so large that for all $w \in \mathcal{B}'$ and all $1 \le j \le J_w$ we have $\nu_{I_{D_j,w}}(y,N,Q;\mathcal{N}_\sigma) < CNm\left(I_{D_j,w}\right)^\sigma$. 
    Let $\mathcal{A} = \{n \in \mathbb{N}\ |\ T^nx \not\in X_{\mathcal{B}'}\}\cup\mathcal{N}_\sigma$ and observe that $d(\mathcal{A}) \le \epsilon+1-\sigma$. 
    Let
    
    \begin{equation*}
        B = \left\{n \in \mathbb{N}\setminus\mathcal{N}_\sigma\ |\ (T\rtimes M)^n(x,y) \in \bigcup_{w \in \mathcal{B}'}\bigcup_{j = 1}^{J_w}T(E_w)\times I_{D_j,w}\right\} \text{ and }G = \mathbb{N}\setminus(\mathcal{A}\cup B).
    \end{equation*}
   Intuitively, $\mathcal{A}\cup B$ is the set of bad $n$ and $G$ is the set of good $n$. 
   We see that

    \begin{alignat}{2}
        &\frac{|B\cap[1,N]|}{N} \le \sum_{w \in \mathcal{B}'}\sum_{j = 1}^{J_w}\frac{\nu_{I_{D_j,w}}\left(y,N,Q;\mathcal{N}_\sigma\right)}{N} < \sum_{w \in \mathcal{B}'}\sum_{j = 1}^{J_w}Cm\left(I_{D_j,w}\right)^\sigma\label{EquationThatCouldBeSimplified}\\
        \le &C\sum_{w \in \mathcal{B}'}J_w\text{exp}\left(-KL\left(\tilde{f}(x_w)-1\right)\right)^\sigma\notag\\
        \le &4CK\sum_{w \in \mathcal{B}'}\text{exp}\left(KL(1-\sigma)\tilde{f}(x_w)+KL(1+\sigma)-\frac{1}{2}LW^2\epsilon^2\right)\label{EquationNeedingFiniteIntegral}\\
        \le&4CK|\mathcal{B}'|\text{exp}\left(-KL\epsilon^2\right) < 4CK\epsilon,\text{ hence}\notag\\
        &\frac{\nu_{(c,d)}(y,N,Q)}{N} = \sum_{n \in [1,N]}\frac{\nu_{(c,d)}(M_ny,KL,w_n)}{NKL}+O\left(\frac{KL}{N}\right)\notag\\
        =& \sum_{n \in [1,N]\setminus\mathcal{A}}\frac{\nu_{(c,d)}(M_ny,KL,w_n)}{NKL}+O(1-\sigma+\epsilon)+O\left(\frac{KL}{N}\right)\notag\\
        =& \sum_{n \in B\cap[1,N]}\frac{\nu_{(c,d)}(M_ny,KL,w_n)}{NKL}+\sum_{n \in G\cap[1,N]}\frac{\nu_{(c,d)}(M_ny,KL,w_n)}{NKL}+O\left(1-\sigma+2\epsilon\right)+O\left(\frac{KL}{N}\right)\notag\\
        =&O\left(\frac{|B\cap[1,N]|}{N}\right)+\frac{|G\cap[1,N]|}{N}(d-c)+O(W\epsilon)+O(1-\delta+\epsilon)+O\left(\frac{KL}{N}\right)\notag\\
        =&O(C\epsilon)+(1-O(C\epsilon)-O(1-\sigma+\epsilon))(d-c)+O(W\epsilon)+O(1-\sigma+\epsilon)+O\left(\frac{KL}{N}\right).\notag
    \end{alignat}
    Since $\epsilon$ can be taken to be arbitrarily small, $\sigma$ can be taken to be arbitrarily close to $1$, and $N$ can be taken to be arbitrarily large (after $K$ and $L$ have been determined), the desired result follows.

    We now proceed to prove $(ii)$.
    We observe that in the proof of $(i)$, the assumption in Equation \eqref{GeneralizedDDGHotspotAssumption} was only used to show that $\frac{|B|}{N} < C\epsilon$. 
    Consequently, it suffices to once again show that $\frac{|B|}{N} < 2C\epsilon$ under the new assumption in Equation \eqref{GeneralizedDDGHotspotAssumption2} and the ergodicity of $(X,\mathscr{B},\mu,T)$.
    Our method of doing this will be to show that if we enlarge the sets $T(E_w)\times I_{D_j,w}$ to the sets $S_{D_j}$ (instead of the horizontal strips $X\times I_{D_j,w}$ that were used in part (i)), then the measure of the cover is still small. 
    To this end, we see that

    \begin{alignat*}{2}
        &\frac{|B|}{N} = \sum_{w \in \mathcal{B}'}\sum_{j = 1}^{J_w}\frac{N^Q_N(D_j,w,y;\mathcal{A})}{N} \le \sum_{w \in \mathcal{B}'}\sum_{j = 1}^{J_w}\frac{N^Q_N(D_j,y;\mathcal{A})}{N} \le \sum_{w \in \mathcal{B}'}\sum_{j = 1}^{J_w}CP_{D_j}.
    \end{alignat*}
    Since $(X,\mathscr{B},\mu,T)$ is ergodic, we may take $M = \int_X\log(f)d\mu$, and we see that for $w \in \mathcal{B}'$ we have $m\left(I_{D_j,w}\right) \in \left(\text{exp}(-KL(M+1)),\text{exp}(-KL(M-1))\right)$ and $J_w \le 4K\text{exp}\left(KL(M+1)-\frac{1}{2}LW^2\epsilon^2\right)$.
    For each block of digits $D' \in \mathbb{N}_0^{KL}$, let $P_{D'}(0) = \mu\times m(S_{D'}\cap X_{\mathcal{B}'}\times[0,1))$ and let $P_{D'}(1) = P_{D'}-P_{D'}(0)$.
    We observe that 

    \begin{alignat*}{2}
        &\sum_{D' \in \mathbb{N}_0^{KL}}P_{D'}(1) = \mu(X_{\mathcal{B}'}^c) < \epsilon\text{ and}\\
        &\sum_{w \in \mathcal{B}'}\sum_{j = 1}^{J_w}P_{D_j}(0) = \sum_{w \in \mathcal{B}'}\sum_{j = 1}^{J_w}\sum_{B \in \mathcal{B}'}m(I_{D_j,B})\mu(E_B)\\
        \le&\sum_{w \in \mathcal{B}'}\sum_{j = 1}^{J_w}\sum_{B \in \mathcal{B}'}\text{exp}(-KL(M-1))\mu(E_B) \le \sum_{w \in \mathcal{B}'}J_w\text{exp}(-KL(M-1))\\
        \le & 4K|\mathcal{B}'|\text{exp}(2KL-\frac{1}{2}LW^2\epsilon^2) \le 4K\text{exp}\left(L\left(\epsilon^2+2K-\frac{1}{2}W^2\epsilon^2\right)\right) \le \epsilon\text{, hence}\\
        &\sum_{w \in \mathcal{B}'}\sum_{j = 1}^{J_w}CP_{D_j} = \sum_{w \in \mathcal{B}'}\sum_{j = 1}^{J_w}CP_{D_j}(0)+\sum_{w \in \mathcal{B}'}\sum_{j = 1}^{J_w}CP_{D_j}(1) < 2C\epsilon.
    \end{alignat*}
\end{proof}

\begin{remark}\label{RemarkAboutHotSpot}
    A few remarks are in order regarding Theorem \ref{StrongestHotSpotForDDGBasicSequences}. 
    Firstly, it is natural to ask if the assumption that $\int_X\log(f)d\mu < \infty$ as well as the assumption that $(X,\mathscr{B},\mu,T)$ is ergodic in part (ii) are necessary assumptions, as we did not need them to prove Theorem \ref{DN(Q)ImpliesUDN(Q)}. 
    We do not currently have any examples demonstrating the necessity of these assumptions, and we conjecture that they are in fact not necessary. 
    The assumption that $\int_X\log(f)d\mu < \infty$ was used in our proof in order to control the rate of growth of the length of most of the intervals $I_{D,B}$ as a function of $\ell$, the length of the blocks $D$ and $B$.
    This control was necessary in Theorem \ref{StrongestHotSpotForDDGBasicSequences}(i) to produce and bound the term $(1-\sigma)\tilde{f}(x_w)$ in Equation \eqref{EquationNeedingFiniteIntegral}.
    If the assumption that $b-a < \delta$ is dropped, then this simplifies Equation \eqref{GeneralizedDDGHotspotAssumption} to 

    \begin{equation}\label{GeneralizedDDGHotspotAssumption3}
    \limsup_{N\rightarrow\infty}\frac{\nu_{(a,b)}\left(y,N,Q;\mathcal{N}_\sigma\right)}{N} < C(b-a),
\end{equation}
and in this case Theorem \ref{StrongestHotSpotForDDGBasicSequences}(i) can be proven even when $\int_X\log(f)d\mu = \infty$.
In particular, the right hand side of Equation \eqref{EquationThatCouldBeSimplified} can easily be bounded above, even when $\int_X\log(f)d\mu = \infty$, as long as we do not have the exponent of $\sigma$.
However, this latter situation 
yields a result much closer to the Classical Hot Spot rather than the Bergelson-Vandehey Hot Spot.

In proving Theorem \ref{StrongestHotSpotForDDGBasicSequences}(ii), the main difficulty was being able to efficiently bound $\mu\times m(S_D)/\mu\times m(E_{D,B})$.
In general, the regions $S_D$ do not have a well understood shape (cf. Figure \ref{PartiallyFilledLevel3Grid}), but control on the lengths of the intervals $I_{D,B}$ allows us to obtain good bounds on $P_D(0) = \mu\times m(S_D\cap X_L)$.
Since we do not have a good understanding of the shape of $S_D$, we needed to assume the ergodicity of $(X,\mathscr{B},\mu,T)$ in order to have even more control on the length of $I_{D,B}$ than what we had in part (i).
Furthermore, we currently don't have efficient bounds on $P_D(1)/P_D(0)$, which is why Equation \eqref{GeneralizedDDGHotspotAssumption2} only has $CP_D$ instead of $CP_D^\sigma$. 
Nonetheless, we conjecture that if $Q$ is a deterministic dynamically generated basic sequences, and $y \in [0,1)$ and $C > 0$ are such that for every $\sigma \in (0,1)$ there exists $\ell_\sigma \in \mathbb{N}$ and $\mathcal{N}_\sigma \subseteq \mathbb{N}$ of upper density at most $1-\sigma$ such that for all $\ell \ge \ell_\sigma$ and all $D \in \mathbb{N}_0^\ell$ we have

\begin{equation}\label{GeneralizedDDGHotspotAssumption4}
        \limsup_{n\rightarrow\infty}\frac{N_n^Q(D,y;\mathcal{N}_\sigma)}{n} \le CP_D^\sigma,
    \end{equation}
then $y \in \mathcal{UN}(Q)$. 
\end{remark}

\begin{proof}[{Proof of Theorem~\ref{MainCorollary}}]
    Putting together Theorems~\ref{UN=UDN}, \ref{DN(Q)ImpliesUDN(Q)}, and \ref{StrongestHotSpotForDDGBasicSequences}(ii) the desired result follows.
\end{proof}

\section{Ratio Normality}\labs{RatioNormality}
\begin{definition}
    Let $Q = (q_n)_{n = 1}^\infty \in \mathbb{N}_{\ge 2}$ be a basic sequence. A number $y \in [0,1)$ is \textbf{$Q$-ratio normal} if for any $\ell \in \mathbb{N}$ and any $D_1,D_2 \in \mathbb{N}_0^\ell$ satisfying $\displaystyle\lim_{n\rightarrow\infty}\text{min}(Q_n(D_1),Q_n(D_2)) = \infty$, we have

    \begin{equation}
        \lim_{n\rightarrow\infty}\frac{N_n^Q(D_1,y)/Q_n(D_1)}{N_n^Q(D_2,y)/Q_n(D_2)} = 1.
    \end{equation}
    The set of all $Q$-ratio normal numbers is denoted by $\mathcal{RN}(Q)$. 
    If $Q$ is dynamically generated, a number $y \in [0,1)$ is \textbf{Uniformly $Q$-ratio normal} if for any $\ell \in \mathbb{N}$, any $B_i = (B_{i,1},\cdots,B_{i,\ell}), D_i = (D_{i,1},\cdots,D_{i,\ell}) \in \mathbb{N}_0^\ell$, $i = 1,2$, satisfying $D_{i,j} < B_{i,j}$ for $i = 1,2$ and $1 \le j \le \ell$, we have

    \begin{alignat*}{2}
        &\lim_{n\rightarrow\infty}\frac{N_n^Q(D_1,B_1,y)/Q_n(D_1,B_1)}{N_n^Q(D_2,B_2,y)/Q_n(D_2,B_2)} = 1.
    \end{alignat*}
\end{definition}

\begin{theorem}\label{RN=N}
    If $Q = (q_n)_{n = 1}^\infty$ is a dynamically generated basic sequence, then $\mathcal{RN}(Q) = \mathcal{N}(Q)$.
\end{theorem}

\begin{proof}
    It is clear that $\mathcal{N}(Q) \subseteq \mathcal{RN}(Q)$, so we proceed to show that $\mathcal{RN}(Q) \subseteq \mathcal{N}(Q)$. To this end, let $y \in \mathcal{RN}(Q)$ and $\ell \in \mathbb{N}$ both be arbitrary, let $0_\ell = \underbrace{(0,\cdots,0)}_{\ell}$, and let $y = 0.y_1y_2\cdots y_n\cdots_Q$. 
    Since $Q$ is dynamically generated, Lemma \ref{LimitsExistForDGBS} tells us that for $D \in \mathbb{N}_0^\ell$ we may define 

    \begin{equation}
        P_D = \lim_{n\rightarrow\infty}\frac{Q_n(D)}{n}\text{, and that }\sum_{D \in \mathbb{N}_0^\ell}P_D = 1.
    \end{equation}
    Since $y \in \mathcal{RN}(Q)$, we see that for any $D \in \mathbb{N}_0^\ell$, we may define

    \begin{equation}
        T_D := \lim_{N\rightarrow\infty}\frac{|\{ 1 \le n \le N\ |\ D = (y_n,y_{n+1},\cdots,y_{n+\ell-1})\}|}{|\{ 1 \le n \le N\ |\ 0_\ell = (y_n,y_{n+1},\cdots,y_{n+\ell-1})\}|} = \lim_{n\rightarrow\infty}\frac{Q_n(D)}{Q_n(0_\ell)} = \frac{P_D}{P_{0_\ell}},
    \end{equation}
    and we see that for any $(H_q)_{q = 1}^\infty \subseteq \mathbb{N}$ we have 

    \begin{alignat*}{2}
        &\lim_{q\rightarrow\infty}\frac{1}{H_q}|\{1 \le n \le H_q\ |\ D = (y_n,y_{n+1},\cdots,y_{n+\ell-1})\}|\\
        =& T_D\lim_{q\rightarrow\infty}\frac{1}{H_q}|\{1 \le n \le H_q\ |\ 0_\ell = (y_n,y_{n+1},\cdots,y_{n+\ell-1})\}|,
    \end{alignat*}
    provided that at least one of the limits exists. 
    Now let us assume for the sake of contradiction that there exists $0 \le a < b \le 1$ and $(N_q)_{q = 1}^\infty,(M_q)_{q = 1}^\infty \subseteq \mathbb{N}$ such that

    \begin{alignat*}{2}
        a& = \liminf_{N\rightarrow\infty}\frac{1}{N}|\{ 1 \le n \le N\ |\ 0_\ell = (y_n,\cdots,y_{n+\ell-1})\}|\\
        &= \lim_{q\rightarrow\infty}\frac{1}{N_q}|\{ 1 \le n \le N_q\ |\ 0_\ell = (y_n,\cdots,y_{n+\ell-1})\}|\text{, and}\\
        b& = \limsup_{N\rightarrow\infty}\frac{1}{N}|\{ 1 \le n \le N\ |\ 0_\ell = (y_n,\cdots,y_{n+\ell-1})\}|\\
        &= \lim_{q\rightarrow\infty}\frac{1}{M_q}|\{ 1 \le n \le M_q\ |\ 0_\ell = (y_n,\cdots,y_{n+\ell-1})\}|.
    \end{alignat*}
    Since $Q$ is dynamically generated, let $M \in\mathbb{N}_{\ge 2}$ be such that

    \begin{equation}
        \lim_{N\rightarrow\infty}\frac{1}{N}|\{1 \le n \le N\ |\ q_{n+i} \le M\ \forall\ 1 \le i \le \ell\}| > \frac{a+b}{2b}.
    \end{equation}
    We now see that

    \begin{alignat*}{2}
        \frac{a+b}{2b} &\le \sum_{w \in [0,M)^\ell}\lim_{q\rightarrow\infty}\frac{1}{N_q}|\{1 \le n \le N_q\ |\ w = (y_n,\cdots,y_{n+\ell-1})\}|\\
        &= \sum_{w \in [0,M)^\ell}T_w\lim_{q\rightarrow\infty}\frac{1}{N_q}|\{1 \le n \le N_q\ |\ 0_\ell = (y_n,\cdots,y_{n+\ell-1})\}|\\
        &= \sum_{w \in [0,M)^\ell}T_wa\text{, hence }a > 0\text{ and}\\
        1 < \frac{a+b}{2a} &\le \sum_{w \in [0,M)^\ell}T_wb = \sum_{w \in [0,M)^\ell}T_w\lim_{q\rightarrow\infty}\frac{1}{M_q}|\{1 \le n \le M_q\ |\ 0_\ell = (y_n,\cdots,y_{n+\ell-1})\}|\\
        &= \sum_{w \in [0,M)^\ell}\lim_{q\rightarrow\infty}\frac{1}{M_q}|\{1 \le n \le M_q\ |\ w = (y_n,\cdots,y_{n+\ell-1})\}|,
    \end{alignat*}
    which yields the desired contradiction. 
    It follows that

    \begin{equation}
        a := \lim_{N\rightarrow\infty}\frac{1}{N}|\{1 \le n \le N\ |\ 0_\ell = (y_n,\cdots,y_{n+\ell-1})\}|
    \end{equation}
    is well defined, and that for any $D \in \mathbb{N}_0^\ell$, we have

    \begin{alignat}{2}
        &\lim_{N\rightarrow\infty}\frac{1}{N}|\{1 \le n \le N\ |\ D = (y_n,\cdots,y_{n+\ell-1})\}| = T_Da.\label{ConsequenceOfRationNormalityEquation}
    \end{alignat}
    Now let us assume for the sake of contradiction that $\epsilon := |a-P_{0_\ell}| > 0$. 
    Let us first consider the case in which $\epsilon = P_{0_\ell}-a$. 
    We observe that for any $M \in\mathbb{N}_{\ge 2}$ we have

    \begin{alignat}{2}
        &  \lim_{N\rightarrow\infty}\frac{1}{N}|\{1 \le n \le N\ |\ q_{n+i} \le M\ \forall\ 1 \le i \le \ell\}|\label{BeginningOfRatioNormalityCalculation}\\
        \le& \sum_{w \in [0,M)^{\ell}}\lim_{N\rightarrow\infty}\frac{1}{N}|\{1 \le n \le N\ |\ w = (y_n,\cdots,y_{n+\ell-1})\}| = \sum_{w \in [0,M)^\ell}T_wa\\
        =& -\epsilon\sum_{w \in [0,M)^\ell}T_w+\sum_{w \in [0,M)^\ell}T_wP_{0_\ell} = -\epsilon\sum_{w \in [0,M)^\ell}\frac{P_w}{P_{0_\ell}}+\sum_{w \in [0,M)^\ell}P_w.\label{EndOfRatioNormalityCalcuation}
    \end{alignat}
    We see that as $D'\rightarrow\infty$, the quantity in Equation \eqref{BeginningOfRatioNormalityCalculation} approaches $1$ and the quantities in Equation \eqref{EndOfRatioNormalityCalcuation} approaches $1-\epsilon P_{0_\ell}^{-1}$, which yields the desired contradiction in this case. 
    Next, let us consider the case in which $\epsilon = a-P_{0_\ell}$. 
    We see that for any $M \in \mathbb{N}_{\ge 2}$ we have
    
    \begin{alignat*}{2}
        1& \ge \lim_{M\rightarrow\infty}\sum_{w \in [0,M)^\ell}\lim_{N\rightarrow\infty}\frac{1}{N}|\{1 \le n \le N\ |\ w = (y_n,\cdots,y_{n+\ell-1})\}|\\
        &= \lim_{M\rightarrow\infty}\sum_{w \in [0,M)^\ell}P_w\frac{a}{P_{0_\ell}} = \frac{a}{P_{0_\ell}} > 1,
    \end{alignat*}
    which yields the desired contradiction in this case as well. It follows that we must have $a = P_{0_\ell}$.
    The desired result now follows from Equation \eqref{ConsequenceOfRationNormalityEquation}.
\end{proof}

The proof of Theorem \ref{URN=UN} is almost identical to that of Theorem \ref{RN=N}, but we include it for the sake of completeness.

\begin{theorem}\label{URN=UN}
    If $Q = (q_n)_{n = 1}^\infty$ is a dynamically generated basic sequence, then $\mathcal{URN}(Q) = \mathcal{UN}(Q)$.
\end{theorem}

\begin{proof}
    It is clear that $\mathcal{UN}(Q) \subseteq \mathcal{URN}(Q)$, so we proceed to show that $\mathcal{URN}(Q) \subseteq \mathcal{UN}(Q)$.
    To this end, let $y \in \mathcal{URN}(Q)$ and $\ell \in \mathbb{N}$ both be arbitrary, let $0_\ell = \underbrace{(0,\cdots,0)}_{\ell}$, let $B_0 = (q_1,\cdots,q_\ell)$, and let $y = 0.y_1y_2\cdots y_n\cdots_Q$. 
    Since $Q$ is dynamically generated, Lemma \ref{LimitsExistForDGBS} tells us that for any $D \in \mathbb{N}_0^\ell$ and $B \in \mathbb{N}_{\ge 2}^\ell$ we may define 

    \begin{equation}
        P_{D,B} = \lim_{n\rightarrow\infty}\frac{Q_n(D,B)}{n}\text{, and that }\sum_{B \in \mathbb{N}_{\ge 2}^\ell}P_{D,B} = P_D.
    \end{equation}
    Since $x \in \mathcal{RN}(Q)$, we see that for any $D \in \mathbb{N}_0^\ell$ and $B \in \mathbb{N}_{\ge 2}^\ell$ we may define

    \begin{equation}
        T_{D,B} := \lim_{n\rightarrow\infty}\frac{N_n^Q(D,B,y)}{N_n^Q(0_\ell,B_0,y)} = \lim_{n\rightarrow\infty}\frac{Q_n(D,B)}{Q_n(0_\ell,B_0)} = \frac{P_{D,B}}{P_{0_\ell,B_0}},
    \end{equation}
    and we see that for any $(H_q)_{q = 1}^\infty \subseteq \mathbb{N}$ we have

    \begin{equation}
        \lim_{q\rightarrow\infty}N_{H_q}^Q(D,B,y) = T_{D,B}\lim_{q\rightarrow\infty}N_{H_q}^Q(0_\ell,B_0,y),
    \end{equation}
    provided that at least one of the limits exists.
    Now let us assume for the sake of contradiction that there exists $0 \le a < b \le 1$ and $(G_q)_{q = 1}^\infty,(M_q)_{q = 1}^\infty \subseteq \mathbb{N}$ such that 

    \begin{alignat*}{2}
        &a = \liminf_{n\rightarrow\infty}N_{n}^Q(0_\ell,B_0,y) = \lim_{q\rightarrow\infty}N_{G_q}^Q(0_\ell,B_0,y)\text{, and}\\
        &b = \limsup_{n\rightarrow\infty}N_{n}^Q(0_\ell,B_0,y) = \lim_{q\rightarrow\infty}N_{M_q}^Q(0_\ell,B_0,y).
    \end{alignat*}
    Since $Q$ is dynamically generated, let $U \in \mathbb{N}_{\ge 2}$ be such that

    \begin{equation}
        \lim_{N\rightarrow\infty}\frac{1}{N}|\{1 \le n \le N\ |\ q_{n+i} \le U\ \forall\ 1 \le i \le \ell\}| > \frac{a+b}{2b}.
    \end{equation}
    We now see that

    \begin{alignat*}{2}
        \frac{a+b}{2b} &\le \sum_{D \in [0,U)^\ell}\sum_{B \in [2,U]^\ell}\lim_{q\rightarrow\infty}N_{G_q}^Q(D,B,y) = \sum_{D \in [0,U)^\ell}\sum_{B \in [2,U]^\ell}T_{D,B}a\text{, hence }a > 0\text{, and}\\
        1 < \frac{a+b}{2a}&\le \sum_{D \in [0,U)^\ell}\sum_{B \in [2,U]^\ell}T_{D,B}b = \sum_{D \in [0,U)^\ell}\sum_{B \in [2,U]^\ell}\lim_{q\rightarrow\infty}N_{M_q}^Q(D,B,y),
    \end{alignat*}
    which yields the desired contradiction.
    It follows that 

    \begin{equation}
        a = \lim_{n\rightarrow\infty}N_n^Q(0_\ell,B_0,y)
    \end{equation}
    is well defined, and that for any $D \in \mathbb{N}_0^\ell$ and $B \in \mathbb{N}_{\ge 2}^\ell$, we have

    \begin{equation}\label{URNImpliesUNEquation}
        \lim_{n\rightarrow\infty}N_n^Q(D,B,y) = T_{D,B}a.
    \end{equation}
    We now see that for each $U \ge 2$ we have

    \begin{alignat}{2}
        &\lim_{N\rightarrow\infty}\frac{1}{N}|\{1 \le n \le N\ |\ q_{n+i} \le U\ \forall\ 1\le i \le \ell\}|\label{Equation1ForURN=UN}\\
        =&\sum_{D \in [0,U)^\ell}\sum_{B \in [2,U]^\ell}\lim_{n\rightarrow\infty}N_n^Q(D,B,y) = \frac{a}{P_{0_\ell,B_0}}\sum_{D \in [0,U)^\ell}\sum_{B \in [2,U]^\ell}P_{D,B}.\label{Equation2ForURN=UN}
    \end{alignat}
    Since 

    \begin{equation}
        \sum_{D \in \mathbb{N}_0^\ell}\sum_{B \in \mathbb{N}_{\ge 2}^\ell}P_{D,B} = 1, 
    \end{equation}
    We see that as $U\rightarrow\infty$, the quantity in Equation \eqref{Equation1ForURN=UN} approaches 1, while the rightmost quantity in Equation \eqref{Equation2ForURN=UN} approaches $\frac{a}{P_{0_\ell,B_0}}$, hence $a = P_{0_\ell,B_0}$.
    The desired result now follows from Equation \eqref{URNImpliesUNEquation}.
\end{proof}

\section{$g$-power basic sequences}\labs{gpower}
In the previous section we found some classes of basic sequences $Q$ for which $\mathcal{N}(Q) = \mathcal{DN}(Q) = \mathcal{UN}(Q)$. In this section we will assume that $Q$ is a $g$-power sequence and establish some relations between $\mathcal{N}(Q),\mathcal{DN}(Q),\mathcal{UN}(Q)$, and $\mathcal{N}_g$.

\begin{lemma}\label{DensityCalculationLemma}
    Let $g \in \mathbb{N}_{\ge 2}$ and let $Q = (q_n)_{n = 1}^\infty \in \left(\{g^n\}_{n = 1}^\infty\right)^\mathbb{N}$ be generated by $(X,\mathscr{B},\mu,T,f,x)$ with $I := \int_X\log_gfd\mu < \infty$. Consider the set $A = (a_n)_{n = 0}^\infty$ determined by $a_0 = 0$ and $a_n = a_{n-1}+\log_g(q_n)$. For any $k \in \mathbb{N}_0$, the set $A_k := A\cup(A+1)\cup\cdots\cup(A+k)$ has natural density

    \begin{equation}
        d(A_k) = \frac{1}{I}\int_X\text{min}(k+1,\log_g(f))d\mu = 1-\frac{1}{I}\int_{\log_g(f) > k+1}(\log_g(f)-k-1)d\mu.
    \end{equation}
\end{lemma}

\begin{proof}
    Consider the m.p.s. $\mathcal{X}' = (X',\mathscr{B}',\mu',T')$ that is the (discrete) flow of $\mathcal{X}$ under the function $\log_g(f)$. 
    To be more precise, let $E_m = f^{-1}(\{g^m\})$ so that we have

    \begin{alignat*}{2}
        &X' = \bigsqcup_{m = 1}^\infty\bigsqcup_{j = 1}^mE_m\times\{j\},\ \mu'(E_m\times\{j\}) = \frac{1}{m}\mu(E_m)\text{ and}\\
        &T'(x',j) = \begin{cases}
            (Tx',1) &\text{ if }x \in E_m\times\{m\}\text{ for some }m\\
            (x',j+1) &\text{ else}.
        \end{cases},
    \end{alignat*}
    and $\mathscr{B}'$ is the natural $\sigma$-algebra. 
    Letting $A_k = (a_{k,n})_{n = 1}^\infty$, $E = \bigcup_{m = 1}^\infty E_{m,1}$, and $x \in E$ correspond to $x \in X$ (by abuse of notation), we see that $a \in A_k$ if and only if
    
    \begin{equation}
        T^ax \in \bigsqcup_{m = 1}^\infty\bigsqcup_{j = 1}^{\text{min}(m,k+1)}E_{m,j} =: E(k).
    \end{equation}
    We also see that $x$ is a generic point for each $\mathbbm{1}_{E_k}$, and hence for each $\mathbbm{1}_{E(k)}$ as well. 
    The genericity of $x$ shows us that

    \begin{equation*}
        d(\{m \in \mathbb{N}\ |\ T^mx \in E(k)\}) = \lim_{M\rightarrow\infty}\frac{1}{M}\sum_{m = 1}^M\mathbbm{1}_{E(k)}\left(T^mx\right) = \mu'(E(k)) = \frac{1}{I}\int_X\text{min}(k+1,\log_g(f))d\mu.
    \end{equation*}
\end{proof}

\begin{lemma}\label{DistributionNormalImpliesgNormal}
    Let $g \in \mathbb{N}_{\ge 2}$ and let $Q = (q_n)_{n = 1}^\infty \in \left(\{g^n\}_{n = 1}^\infty\right)^{\mathbb{N}}$ be a basic sequence generated by $(X,\mathscr{B},\mu,T,f,x)$ with $\int_X\log(f)d\mu < \infty$.
    We have $\mathcal{DN}(Q) \subseteq \mathcal{N}_g$.
\end{lemma}

\begin{proof}
    Consider the c.m.p.s. $\mathcal{X}' = (X',\mathscr{B}',\mu',T')$ that is the flow of $\mathcal{X}$ under the function $\log_g(f)$, and let $A_k = (a_{k,n})_{n = 1}^\infty$ as in Lemma \ref{DensityCalculationLemma}.
    Let $\sigma > 0$ be arbitrary, let $k_0 = k_0(\sigma) \in \mathbb{N}$ be such that $d(A_k) > 1-\sigma$ and $b^{\sigma k} \ge k$ for all $k \ge k_0$, and let $\mathcal{N}_{\sigma,k} = A_k^c$. 
    We see that $\left(M_g^{a_{1,n}}y\right)_{n = 1}^\infty$ is uniformly distributed since $y$ was assumed to be $Q$-distribution normal, so $\left(M_g^{a_{1,n}+t}y\right)_{n = 1}^\infty$ is also uniformly distributed for any $t \in \mathbb{N}$ since $M_g$ maps uniformly distributed sequences to uniformly distributed sequences.
    Now let $s = [d_1,d_2,\cdots,d_k]$ be an arbitrary word of length $k \ge k_0$ and observe that for each $y \in \mathcal{DN}(Q)$ we have

    \begin{alignat*}{2}
        &\limsup_{N\rightarrow\infty}\frac{\nu_s\left(y,N;\mathcal{N}_{\sigma,k}\right)}{N} = \limsup_{N\rightarrow\infty}\frac{1}{N}\left|\left\{n \in [1,N]\setminus\mathcal{N}_{\sigma,k}\ |\ M_g^ny \in \Bigg[\sum_{i = 1}^k\frac{d_i}{g^i},\sum_{i = 1}^k\frac{d_i}{g^i}+\frac{1}{g^k}\Bigg)\right\}\right|\\
        \le &\sum_{t = 1}^k\limsup_{N\rightarrow\infty}\frac{1}{N}\left|\left\{1 \le n \le N\ |\ M_g^{a_{1,n}+t}y \in \Bigg[\sum_{i = 1}^k\frac{d_i}{g^i},\sum_{i = 1}^k\frac{d_i}{g^i}+\frac{1}{g^k}\Bigg)\right\}\right| = \frac{k}{g^k} \le \frac{g^{\sigma k }}{g^k},
    \end{alignat*}
    so the desired result now follows from Theorem \ref{StrongestHotSpot}.
\end{proof}

\begin{remark}
    We see that the converse to Lemma \ref{DistributionNormalImpliesgNormal} is not true as a consequence of Example \ref{QWithoutHotSpot}.
\end{remark}

\begin{proof}[Proof of Theorem~\ref{MainTheoremForPowersOf2}]
    Theorem \ref{UN=UDN} tells us that $\mathcal{UN}(Q) = \mathcal{UDN}(Q)$. 
    Theorem \ref{DN(Q)ImpliesUDN(Q)} tells us that $\mathcal{DN}(Q) = \mathcal{UN}(Q)$.
    Theorem \ref{StrongestHotSpotForDDGBasicSequences}(ii) tells us that $\mathcal{UN}(Q) = \mathcal{N}(Q)$ when we have ergodicity. 
    Lemma \ref{DistributionNormalImpliesgNormal} tells us that $\mathcal{DN}(Q) \subseteq \mathcal{N}_g$, so it only remains to show that $\mathcal{N}_g \subseteq \mathcal{DN}(Q)$. 
    Let $\mathcal{X}' = (X',\mathscr{B}',\mu',T')$ and $E \in \mathscr{B}'$ both be as in Lemma \ref{DensityCalculationLemma}, and as before let us identify $X$ with $E$.
    Since $Q$ is deterministic, $\mathcal{X}$ has 0 entropy, so $\mathcal{X}'$ also has 0 entropy by Abramov's formula.
    Let $y \in \mathcal{N}_g$ be arbitrary and let $\nu$ be any weak$^*$-limit of the sequence $\left\{\frac{1}{N}\sum_{n = 1}^N\delta_{T'^n(x),M_g^n(y)}\right\}_{N = 1}^\infty$. 
    We see that $\nu$ is a joining of $\mu'$ and $m$, where $m$ denote the Lebesgue measure on $[0,1)$. 
    in Since $M_g$ is Bernoulli, it is disjoint (in the sense of Furstenberg) from $T'$, hence $\nu = \mu'\times m$. 
    Since $\nu$ was an arbitrary weak$^*$ limit, we see that $(x,y)$ is a generic point for $T'\times M_g$ with respect to $\mu\times m$.
    To see that $y \in \mathcal{DN}(Q)$, let $0 \le a < b \le 1$ be arbitrary, recall that $E$ is clopen and observe that
    
    \begin{alignat*}{2}
        &\lim_{N\rightarrow\infty}\frac{1}{N}\left|\left\{1 \le n \le N\ |\ q_n\cdots q_1y \in (a,b)\right\}\right|\\
        =& \lim_{N\rightarrow\infty}\frac{1}{N\mu'(E)}\left|\left\{1 \le n \le N\ |\ \left(T'\times M_g\right)^n(x,y) \in E\times(a,b)\right\}\right| = b-a.
    \end{alignat*}
\end{proof}

\section{Questions and Conjectures}\labs{Conjectures}

The present paper is only the beginning of the theory of dynamically generated basic sequences. 
In this section we give many question and conjectures to illustrate how much more there is to be done.
Our first conjecture was discussed in detail in Remark \ref{RemarkAboutHotSpot}, so here we only state an abbreviated version.

\begin{conjecture}\label{HotSpotConjecture}
    Theorem \ref{StrongestHotSpotForDDGBasicSequences} holds for any determinstic dynamically generated basic sequence $Q$.
\end{conjecture}

A positive answer to Conjecture \ref{HotSpotConjecture} would also yield a positive answer to our next conjecture, but we state it nonetheless because it is of independent interest.

\begin{conjecture}\label{ZeroEntropyConjecture}
    If $Q$ is a determinstic dynamically generated basic sequence, then $\mathcal{N}(Q) = \mathcal{UN}(Q) = \mathcal{UDN}(Q) = \mathcal{DN}(Q)$.
\end{conjecture}

For a converse to Conjecture \ref{ZeroEntropyConjecture} we choose to be even bolder.

\begin{conjecture}\label{PositiveEntropyConjecture}
    If $Q$ is a non-deterministic dynamically generated basic sequence, then the following hold:
    \begin{enumerate}[(i)]
        \item $(\mathcal{N}(Q)\cap\mathcal{DN}(Q))\setminus\mathcal{UN}(Q)$ has Hausdorff dimension 1 and is $D_2\left(\pzt\right)$-complete.
        \item $\mathcal{DN}(Q)\setminus\mathcal{N}(Q)$ has Hausdorff dimension 1 and is $D_2\left(\pzt\right)$-complete.
        \item $\mathcal{N}(Q)\setminus\mathcal{DN}(Q)$ has Hausdorff dimension 1 and is $D_2\left(\pzt\right)$-complete.
    \end{enumerate}
\end{conjecture}

It can be checked with a more detailed analysis that for the $Q$ appearing in Example \ref{DN(Q)ButNOtN(Q)} we have that $\mathcal{DN}(Q)\setminus\mathcal{N}(Q)$ has Hausdorff dimension 1 and is $D_2\left(\pzt\right)$-complete.
Similarly, it can be checked with a more detailed analysis that for the $Q$ appearing in Example \ref{N(Q)ButNotDN(Q)} we have that $\mathcal{N}(Q)\setminus\mathcal{DN}(Q)$ has Hausdorff dimension 1 and is $D_2\left(\pzt\right)$-complete.
We note that for any dynamically generated basic sequence $Q$, the sets $\NQ$, $\RNQ$, and $\DNQ$ are all $\bP_3^0$-complete by \cite{AireyJacksonManceComplexityCantorSeries}.

As was mentioned in the introduction, two classical results in the theory of normality that make use of Theorem \reft{wall} are Rauzy's Theorem \cite{NormalityPreservationByAddition} on normality preservation through addition, as well as the Kamae-Weiss Selection rule \cite{SelectionRules1,SelectionRules2}.
Consequently, we are left with the following questions.

\begin{question}\label{SelectionRuleQuestion}
    Given a dynamically generated basic sequence $Q$, what are the selection rules for $Q$?
\end{question}

We point out that selection rules for a dynamically generated basic sequence $Q$ are much more delicate than the selection rules for base $g$ expansions.
When working with base $g$ expansions, it suffices to select a sequence of digits and use them to construct a new number.
However, when working with base $Q$ expansion, we must select pairs of digits and bases together, and this introduces many additional complications.

\begin{question}\label{NormalityPreservationUnderAdditionQuestion}
    Let $Q$ be a dynamically generated basic sequence.
    \begin{enumerate}[(i)]
        \item When is $\mathcal{N}(Q)$ preserved under addition?

        \item When is $\mathcal{DN}(Q)$ preserved under addition?

        \item When is $\mathcal{UN}(Q)$ preserved under addition?
    \end{enumerate}
\end{question}

There are also results that are not too difficult to prove for base $g$ expansions, but become much harder when dealing with Cantor series.
For example, it is simple to show that base $g$ normality is preserved under integer multiplication, and with a little more work, it can be shown that it is even preserved under rational multiplication.
While it is clear that distribution normality is preserved under integer multiplication for any basic sequence $Q$, the situation is not so simple for $Q$-normality. Let 
$$
\Xi(Q)=\{x \in \NQ : nx \notin \NQ \forall n \in \N_{\geq 2}\}.
$$
It was observed by Mance in \cite{ManceFractalFunctions} that there exists a basic sequence $Q$ for which $\Xi(Q)$ is non-empty. Furthermore, there is a large class of basic sequences $Q$ where $\dimh{\Xi(Q)}=1$. This follows from a small modification of the proof of Theorem~3.5 in the paper of Airey and Mance \cite{AireyManceHDDifference}. In order to do this, we instead let $\xi=E_0.E_1E_2\ldots \in \mathcal{N}(P)$ satisfy the property that $E_n \leq \log p_n$. This is possible by the proof of Theorem~3.12 of \cite{ManceFractalFunctions}. After this change, Theorem~3.5 of \cite{AireyManceHDDifference} describes a large class of $Q$ where
$$
\dimh{\Xi(Q) \cap \NQ \setminus \DNQ}=1.
$$
Furthermore, for $Q$ where $q_n \to \infty$, it was shown by Airey and Mance \cite{AireyManceNormalPreserves} that for every $r \in \Q \setminus \Z$ $\dimh{\{x \in \DNQ : rx \notin \DNQ\} }=1$. It was also shown by Airey, Mance, and Vandehey \cite{AireyManceVandehey} that there exists a $Q$ such that
$$
\dimh{\{x \in [0,1) : rx+s \in \NQ \setminus \DNQ \ \forall r \in \Q \setminus \{0\}, \ \forall s \in \Q \}}=1.
$$

In this direction we have proven Theorem \ref{NormalityPreservationUnderRationalMultiplication}, but we are still left with the following conjecture and questions.

\begin{conjecture}
    If $Q$ is a dynamically generated basic sequence, then $\mathcal{UN}(Q)$ is preserved under rational multiplication.
\end{conjecture}

\begin{question}
Let $Q$ be a dynamically generated basic sequence.
\begin{enumerate}[(i)]
    \item When is $\mathcal{DN}(Q)$ preserved under rational multiplication?

    \item When is $\mathcal{N}(Q)$ preserved under rational multiplication?
\end{enumerate}
\end{question}

Both having a Hot Spot Theorem and equivalence of normality and distribution normality appear to be essential for developing a theory of normality for a basic sequence $Q$ that resembles that of a base $g$ expansion. It is not clear whether it is possible to have one of these and not the other and what effect this would have. Thus, we ask the following.

\begin{question}
\
\begin{enumerate}[(i)]
    \item Does there exist a basic sequence $Q$ for which $\mathcal{DN}(Q) = \mathcal{N}(Q)$, but for which $Q$ does not admit a Hot Spot Theorem?

    \item Does there exist a basic sequence $Q$ for which $\mathcal{DN}(Q) \neq \mathcal{N}(Q)$, but for which $Q$ admits a Hot Spot Theorem?
\end{enumerate}
\end{question}

\begin{remark}\label{RemarkAboutJoinings}
While we have briefly mentioned joinings and disjointness of dynamical systems in the proof of Theorem \ref{MainTheoremForPowersOf2}, we have not initiated a deeper study of the dynamical properties of the skew product system $\mathcal{X}^f$. 
In personal communications, Tim Austin proved for us that $\mathcal{X}^f$ is an intermediate factor between $\mathcal{X}$ and $\mathcal{X}\times\mathcal{B}$ where $\mathcal{B}$ is a Bernoulli shift of infinite entropy.
It can also be checked that the algebra of sets $S_D$ (discussed in Section \ref{DynamicallyGeneratedBasicSequencesSection}) in $X\times[0,1)$ naturally generate a factor $\mathcal{X}^{\mathcal{N}(Q)}$ of $\mathcal{X}^f$ that carries all of the dynamical information relating to $Q$-normality.
A deeper understanding of the following diagram of factors may help resolve the questions and conjectures posed in this section.

\begin{equation}
    \begin{tikzcd}
  \mathcal{X}\times\mathcal{B} \arrow[d] & \\%
\mathcal{X}^f \arrow[d] \arrow[dr] &  \\
\mathcal{X}& \mathcal{X}^{\mathcal{N}(Q)} 
\end{tikzcd}
\end{equation}

However, we warn the reader that if $\lambda$ is a joining measure on $\mathcal{X}\times\mathcal{B}$, then it does not necessarily project to the measure $\mu\times m$ on $\mathcal{X}^f$, otherwise we would have easily been able to show that $\mathcal{DN}(Q) \neq \mathcal{UDN}(Q)$ when $Q$ is a non-deterministic dynamically generated basic sequence.  
It is also worth noting that if $\mathcal{X}$ has zero entropy, then the main result of \cite{SplittingForDeterministicTimesBernoulli} tells us that $\mathcal{X}^f$ is measurable isomorphic to $\mathcal{X}\times\mathcal{B}'$, where $\mathcal{B}' = ([0,1),\mathscr{L},m,S)$ is Bernoulli.
The reason that we cannot currently make use of this seemingly relevant fact, is that the isomorphism is measurable and not topological.
More concretely, if $Q$ is a basic sequence generated by $(X,\mathscr{B},\mu,T,f,x)$, then a measurable isomorphism may lose all information related to the point $x$.
\end{remark}

\bibliographystyle{abbrv}
\begin{center}
	\bibliography{references}

@article {PeriodicBasicSequences,
    AUTHOR = {Airey, Dylan and Mance, Bill},
     TITLE = {Normal equivalencies for eventually periodic basic sequences},
   JOURNAL = {Indag. Math. (N.S.)},
  FJOURNAL = {Koninklijke Nederlandse Akademie van Wetenschappen.
              Indagationes Mathematicae. New Series},
    VOLUME = {26},
      YEAR = {2015},
    NUMBER = {3},
     PAGES = {476--484},
      ISSN = {0019-3577},
   MRCLASS = {11K16 (11K55)},
  MRNUMBER = {3341809},
MRREVIEWER = {Val\'{e}rie Berth\'{e}},
       DOI = {10.1016/j-1in-102eb17ax508c.han.amu.edu.pldag.2015.02.002},
       URL = {https://doi-1org-102eb17ax508c.han.amu.edu.pl/10.1016/j-1in-102eb17ax508c.han.amu.edu.pldag.2015.02.002},
}

@Book{BugeaudBook,
  Title                    = {Distribution modulo one and {D}iophantine approximation},
  Author                   = {Bugeaud, Y.},
  Publisher                = {Cambridge University Press},
  Year                     = {2012},

  Address                  = {Cambridge}
}

@article {OriginalHotSpot,
    AUTHOR = {\v{S}apiro-Pyatecki\u{\i}, I. I.},
     TITLE = {On the laws of distribution of the fractional parts of an
              exponential function},
   JOURNAL = {Izv. Akad. Nauk SSSR Ser. Mat.},
  FJOURNAL = {Izvestiya Akademii Nauk SSSR. Seriya Matematicheskaya},
    VOLUME = {15},
      YEAR = {1951},
     PAGES = {47--52},
      ISSN = {0373-2436},
   MRCLASS = {10.0X},
  MRNUMBER = {43145},
MRREVIEWER = {R. A. Rankin},
}

@article {AStrongHotSpot,
    AUTHOR = {Bailey, David H. and Misiurewicz, Micha\l },
     TITLE = {A strong hot spot theorem},
   JOURNAL = {Proc. Amer. Math. Soc.},
  FJOURNAL = {Proceedings of the American Mathematical Society},
    VOLUME = {134},
      YEAR = {2006},
    NUMBER = {9},
     PAGES = {2495--2501},
      ISSN = {0002-9939},
   MRCLASS = {11K16 (37A30 37A45)},
  MRNUMBER = {2213726},
MRREVIEWER = {Wolfgang Steiner},
       DOI = {10.1090/S0002-9939-06-08551-0},
       URL = {https://doi-1org-102eb17am0c67.han.amu.edu.pl/10.1090/S0002-9939-06-08551-0},
}

@article {AStrongerHotSpot,
    AUTHOR = {Postnikov, A. G.},
     TITLE = {On the distribution of the fractional parts of the exponential
              function},
   JOURNAL = {Doklady Akad. Nauk SSSR (N.S.)},
  FJOURNAL = {Doklady Akad. Nauk SSSR (N.S.)},
    VOLUME = {86},
      YEAR = {1952},
     PAGES = {473--476},
   MRCLASS = {10.0X},
  MRNUMBER = {50637},
MRREVIEWER = {R. A. Rankin},
}

@article {StrongestHotSpotTheorem,
    AUTHOR = {Pjatecki\u{\i}-\v{S}apiro, I. I.},
     TITLE = {On the distribution of the fractional parts of the exponential
              function},
   JOURNAL = {Moskov. Gos. Ped. Inst. U\v{c}. Zap.},
    VOLUME = {108},
      YEAR = {1957},
     PAGES = {317--322},
   MRCLASS = {10.00},
  MRNUMBER = {0113865},
MRREVIEWER = {I. P. Kubilyus},
}

@article {HotSpotProofOfGeneralizedWallTheorem,
    AUTHOR = {Bergelson, Vitaly and Vandehey, Joseph},
     TITLE = {A hot spot proof of the generalized {W}all theorem},
   JOURNAL = {Amer. Math. Monthly},
  FJOURNAL = {American Mathematical Monthly},
    VOLUME = {126},
      YEAR = {2019},
    NUMBER = {10},
     PAGES = {876--890},
      ISSN = {0002-9890,1930-0972},
   MRCLASS = {11K16},
  MRNUMBER = {4033223},
MRREVIEWER = {J\"{o}rg\ Neunh\"{a}userer},
       DOI = {10.1080/00029890.2019.1651166},
       URL = {https://doi.org/10.1080/00029890.2019.1651166},
}

@article {NormalEquivalencesForEventuallyPeriodicSequences,
    AUTHOR = {Airey, Dylan and Mance, Bill},
     TITLE = {Normal equivalencies for eventually periodic basic sequences},
   JOURNAL = {Indag. Math. (N.S.)},
  FJOURNAL = {Koninklijke Nederlandse Akademie van Wetenschappen.
              Indagationes Mathematicae. New Series},
    VOLUME = {26},
      YEAR = {2015},
    NUMBER = {3},
     PAGES = {476--484},
      ISSN = {0019-3577},
   MRCLASS = {11K16 (11K55)},
  MRNUMBER = {3341809},
MRREVIEWER = {Val\'{e}rie Berth\'{e}},
       DOI = {10.1016/j-1in-102eb17fk0024.han.amu.edu.pldag.2015.02.002},
       URL = {https://doi-1org-102eb17fk0024.han.amu.edu.pl/10.1016/j-1in-102eb17fk0024.han.amu.edu.pldag.2015.02.002},
}

@book {SingleOrbitDynamics,
    AUTHOR = {Weiss, Benjamin},
     TITLE = {Single orbit dynamics},
    SERIES = {CBMS Regional Conference Series in Mathematics},
    VOLUME = {95},
 PUBLISHER = {American Mathematical Society, Providence, RI},
      YEAR = {2000},
     PAGES = {x+113},
      ISBN = {0-8218-0414-6},
   MRCLASS = {37-02 (28D05 37A25 37A50 37B99 60G10)},
  MRNUMBER = {1727510},
MRREVIEWER = {Thomas Ward},
       DOI = {10.1090/cbms/095},
       URL = {https://doi-1org-102eb174r0034.han.amu.edu.pl/10.1090/cbms/095},
}

@article{AireyJacksonManceComplexityCantorSeries,
    AUTHOR = {Airey, Dylan and Jackson, Steve and Mance, Bill},
     TITLE = {Descriptive complexity in {C}antor series},
   JOURNAL = {J. Symb. Log.},
  FJOURNAL = {The Journal of Symbolic Logic},
    VOLUME = {87},
      YEAR = {2022},
    NUMBER = {3},
     PAGES = {1023--1045},
      ISSN = {0022-4812,1943-5886},
   MRCLASS = {03E15 (11K16 11U99)},
  MRNUMBER = {4472523},
MRREVIEWER = {Jacek\ Tryba},
       DOI = {10.1017/jsl.2021.77},
       URL = {https://doi.org/10.1017/jsl.2021.77},
}

@article{AireyManceHDDifference,
 author = {Airey, Dylan and Mance, Bill},
 title = {The {Hausdorff} dimension of sets of numbers defined by their {{\(Q\)}}-{Cantor} series expansions},
 fjournal = {Journal of Fractal Geometry},
 journal = {J. Fractal Geom.},
 issn = {2308-1309},
 volume = {3},
 number = {2},
 pages = {163--186},
 year = {2016},
 language = {English},
 doi = {10.4171/JFG/33},
 zbMATH = {6593688},
 Zbl = {1350.28007}
}

@Article{AireyManceNormalOrders,
  Title                    = {Normality of different orders for {C}antor series expansions},
  Author                   = {Airey, D. and Mance, B.},
  Journal                  = {Nonlinearity},
  Year                     = {2017},
  Pages                    = {3719--3742},
  Volume                   = {30},
  Number                   = {10}
}

@article {BesicovitchSquares,
    AUTHOR = {Besicovitch, A. S.},
     TITLE = {The asymptotic distribution of the numerals in the decimal
              representation of the squares of the natural numbers},
   JOURNAL = {Math. Z.},
  FJOURNAL = {Mathematische Zeitschrift},
    VOLUME = {39},
      YEAR = {1935},
    NUMBER = {1},
     PAGES = {146--156},
      ISSN = {0025-5874,1432-1823},
   MRCLASS = {99-04},
  MRNUMBER = {1545494},
       DOI = {10.1007/BF01201350},
       URL = {https://doi.org/10.1007/BF01201350},
}

@incollection {McDiardmidsInequality,
    AUTHOR = {McDiarmid, Colin},
     TITLE = {On the method of bounded differences},
 BOOKTITLE = {Surveys in combinatorics, 1989 ({N}orwich, 1989)},
    SERIES = {London Math. Soc. Lecture Note Ser.},
    VOLUME = {141},
     PAGES = {148--188},
 PUBLISHER = {Cambridge Univ. Press, Cambridge},
      YEAR = {1989},
      ISBN = {0-521-37823-0},
   MRCLASS = {05C80 (60E15 60F10 60G42)},
  MRNUMBER = {1036755},
MRREVIEWER = {Alan\ M.\ Frieze},
}

@book {Kuipers&Niederreiter,
    AUTHOR = {Kuipers, L. and Niederreiter, H.},
     TITLE = {Uniform distribution of sequences},
    SERIES = {Pure and Applied Mathematics},
 PUBLISHER = {Wiley-Interscience [John Wiley \& Sons], New
              York-London-Sydney},
      YEAR = {1974},
     PAGES = {xiv+390},
   MRCLASS = {10K05 (22D99)},
  MRNUMBER = {0419394},
MRREVIEWER = {P. Gerl},
}

@book{Kechris,
author = {Kechris, A.},
title = {Classical descriptive set theory},
publisher = {Springer-Verlag},
address = {New York},
series = {Graduate Texts in Mathematics},
volume = {156},
year = {1995}
}

@article {SelectionRules2,
    AUTHOR = {Kamae, Teturo and Weiss, Benjamin},
     TITLE = {Normal numbers and selection rules},
   JOURNAL = {Israel J. Math.},
  FJOURNAL = {Israel Journal of Mathematics},
    VOLUME = {21},
      YEAR = {1975},
    NUMBER = {2-3},
     PAGES = {101--110},
      ISSN = {0021-2172},
   MRCLASS = {10K25},
  MRNUMBER = {401695},
MRREVIEWER = {Keith M. Wilkinson},
       DOI = {10.1007/BF02760789},
       URL = {https://doi-1org-102eb17h062be.han.amu.edu.pl/10.1007/BF02760789},
}

@article {SelectionRules1,
    AUTHOR = {Kamae, Teturo},
     TITLE = {Subsequences of normal sequences},
   JOURNAL = {Israel J. Math.},
  FJOURNAL = {Israel Journal of Mathematics},
    VOLUME = {16},
      YEAR = {1973},
     PAGES = {121--149},
      ISSN = {0021-2172},
   MRCLASS = {28A70},
  MRNUMBER = {338321},
MRREVIEWER = {Michael Keane},
       DOI = {10.1007/BF02757864},
       URL = {https://doi-1org-102eb17h062be.han.amu.edu.pl/10.1007/BF02757864},
}

@article {NormalityPreservationByAddition,
    AUTHOR = {Rauzy, G.},
     TITLE = {Nombres normaux et processus d\'{e}terministes},
   JOURNAL = {Acta Arith.},
  FJOURNAL = {Polska Akademia Nauk. Instytut Matematyczny. Acta Arithmetica},
    VOLUME = {29},
      YEAR = {1976},
    NUMBER = {3},
     PAGES = {211--225},
      ISSN = {0065-1036},
   MRCLASS = {10K25 (28A65)},
  MRNUMBER = {404192},
MRREVIEWER = {M. Mend\`es France},
       DOI = {10.4064/aa-29-3-211-225},
       URL = {https://doi-1org-102eb17h062be.han.amu.edu.pl/10.4064/aa-29-3-211-225},
}

@inproceedings{NonInvertibleAbramovRokhlinFormula,
  title={The {A}bramov-{R}okhlin formula},
  author={Bogensch{\"u}tz, Thomas and Crauel, Hans},
  booktitle={Ergodic Theory and Related Topics III: Proceedings of the International Conference held in G{\"u}strow, Germany, October 22--27, 1990},
  pages={32--35},
  year={1992},
  organization={Springer}
}

@article{SplittingForDeterministicTimesBernoulli,
  title={Une classe de systemes pour lesquels la conjecture de Pinsker est vraie},
  author={Thouvenot, Jean-Paul},
  journal={Israel Journal of Mathematics},
  volume={21},
  pages={208--214},
  year={1975},
  publisher={Springer}
}

@book{host2018nilpotent,
  title={Nilpotent structures in ergodic theory},
  author={Host, Bernard and Kra, Bryna},
  volume={236},
  year={2018},
  publisher={American Mathematical Soc.}
}

@article{queffelec1987substitution,
  title={Substitution dynamical systems: spectral analysis},
  author={Queffelec, M},
  journal={Lecture notes in mathematics},
  volume={1294},
  year={1987}
}

@article{adler1981construction,
  title={A construction of a normal number for the continued fraction transformation},
  author={Adler, Roy and Keane, Michael and Smorodinsky, Meir},
  journal={Journal of Number Theory},
  volume={13},
  number={1},
  pages={95--105},
  year={1981},
  publisher={Elsevier}
}

@article{copeland1946note,
  author = {Copeland, Arthur H. and Erd{\H{o}}s, Paul},
 title = {Note on normal numbers},
 fjournal = {Bulletin of the American Mathematical Society},
 journal = {Bull. Am. Math. Soc.},
 issn = {0002-9904},
 volume = {52},
 pages = {857--860},
 year = {1946},
 language = {English},
 doi = {10.1090/S0002-9904-1946-08657-7},
 zbMATH = {3103772},
 Zbl = {0063.00962}
}

@article{champernowne1933construction,
  title={The construction of decimals normal in the scale of ten},
  author={Champernowne, David G},
  journal={Journal of the London Mathematical Society},
  volume={1},
  number={4},
  pages={254--260},
  year={1933},
  publisher={Oxford University Press}
}

@article{davenport1952note,
  title={Note on normal decimals},
  author={Davenport, Harold and Erd{\"o}s, Paul},
  journal={Canadian Journal of Mathematics},
  volume={4},
  pages={58--63},
  year={1952},
  publisher={Cambridge University Press}
}

@article{vandehey2016new,
  title={New normality constructions for continued fraction expansions},
  author={Vandehey, Joseph},
  journal={Journal of Number Theory},
  volume={166},
  pages={424--451},
  year={2016},
  publisher={Elsevier}
}

@Article{RenyiProbability,
  Title                    = {On a new axiomatic theory of probability},
  Author                   = {R\'{e}nyi, A.},
  Journal                  = {Acta Math. Acad. Sci. Hungar.},
  Year                     = {1955},
  Pages                    = {329--332},
  Volume                   = {6}
}

@Article{Renyi,
  Title                    = {On the distribution of the digits in {C}antor's series},
  Author                   = {R\'{e}nyi, A.},
  Journal                  = {Mat. Lapok},
  Year                     = {1956},
  Pages                    = {77--100},
  Volume                   = {7}
}

@Article{RenyiSurvey,
  Title                    = {Probabilistic methods in number theory},
  Author                   = {R\'{e}nyi, A.},
  Journal                  = {Shuxue Jinzhan},
  Year                     = {1958},
  Pages                    = {465--510},
  Volume                   = {4}
}

@Article{ErdosRenyiConvergent,
  Title                    = {On {C}antor's series with convergent $\sum 1/q_n$},
  Author                   = {Erd\H{o}s, P. and R\'{e}nyi, A.},
  Journal                  = {Annales Universitatis L. E\"{o}tv\"{o}s de Budapest, Sect. Math.},
  Year                     = {1959},
  Pages                    = {93--109},
  Volume                   = {2}
}

@Article{ErdosRenyiFurther,
  Title                    = {Some further statistical properties of the digits in {C}antor's series},
  Author                   = {Erd\H{o}s, P. and R\'{e}nyi, A.},
  Journal                  = {Acta Math. Acad. Sci. Hungar},
  Year                     = {1959},
  Pages                    = {21--29},
  Volume                   = {10}
}

@Article{Turan,
  Title                    = {On the distribution of ``digits'' in {C}antor systems},
  Author                   = {Tur\'{a}n, P.},
  Journal                  = {Mat. Lapok},
  Year                     = {1956},
  Pages                    = {71--76},
  Volume                   = {7}
}

@Article{Cantor,
  Title                    = {\"{U}ber die einfachen {Z}ahlensysteme},
  Author                   = {Cantor, G.},
  Journal                  = {Zeitschrift f\"{u}r Math. und Physik},
  Year                     = {1869},
  Pages                    = {121--128},
  Volume                   = {14}
}

@Book{Galambos,
  Title                    = {Representations of real numbers by infinite series},
  Author                   = {Galambos, J.},
  Publisher                = {Springer-Verlag},
  Year                     = {1976},
  Address                  = {Berlin, Hiedelberg, New York},
  Series                   = {Lecture Notes in Math.},
  Volume                   = {502}
}

@article{AireyManceVandehey,
author = {Airey, D. and Mance, B. and Vandehey, J.},
title = {Normality preserving operations for {C}antor series expansions and associated fractals, {II}},
journal = {New York J. Math.},
year = {2015},
pages = {1311--1326},
volume = {21},
note = {MR3441645}
}

@Article{Kifer,
  Title                    = {Fractal dimensions and random transformations},
  Author                   = {Kifer, Y.},
  Journal                  = {Trans. Amer. Math. Soc.},
  Year                     = {1996},
  Pages                    = {2003--2038},
  Volume                   = {5}
}

@PhdThesis{Lafer,
  Title                    = {Normal numbers with respect to {C}antor series representation},
  Author                   = {Lafer, P.},
  School                   = {Washington State University},
  Year                     = {1974},

  Address                  = {Pullman, Washington}
}

@article{MoshchevitinShkredov,
 author = {Moshchevitin, N. G. and Shkredov, I. D.},
 title = {On the {Piatetskii}-{Shapiro} criterion of normality},
 fjournal = {Mathematical Notes},
 journal = {Math. Notes},
 issn = {0001-4346},
 volume = {73},
 number = {4},
 pages = {539--550},
 year = {2003},
 language = {English},
 doi = {10.1023/A:1023263305857},
 zbMATH = {2115599},
 Zbl = {1173.37302}
}

@article{MoshchevitinShkredovCorrection,
 author = {Airey, D. and Mance, B.},
 title = {Hotspot lemmas for noncompact spaces},
 fjournal = {Mathematical Notes},
 journal = {Math. Notes},
 issn = {0001-4346},
 volume = {108},
 number = {3},
 pages = {434--439},
 year = {2020},
 language = {English},
 doi = {10.1134/S0001434620090126},
 zbMATH = {7270840},
 Zbl = {1458.37003}
}

@article{AireyJacksonKwietniakManceSpecificationComplexity,
 author = {Airey, Dylan and Jackson, Steve and Kwietniak, Dominik and Mance, Bill},
 title = {Borel complexity of sets of normal numbers via generic points in subshifts with specification},
 fjournal = {Transactions of the American Mathematical Society},
 journal = {Trans. Am. Math. Soc.},
 issn = {0002-9947},
 volume = {373},
 number = {7},
 pages = {4561--4584},
 year = {2020},
 language = {English},
 doi = {10.1090/tran/8001},
 zbMATH = {7215109},
 Zbl = {1484.03088}
}

@article{AireyJacksonManceNoise,
 author = {Airey, Dylan and Jackson, Steve and Mance, Bill},
 title = {Some complexity results in the theory of normal numbers},
 fjournal = {Canadian Journal of Mathematics},
 journal = {Can. J. Math.},
 issn = {0008-414X},
 volume = {74},
 number = {1},
 pages = {170--198},
 year = {2022},
 language = {English},
 doi = {10.4153/S0008414X20000723},
 zbMATH = {7473361},
 Zbl = {1552.11090}
}

@misc{BergDownNormalityPreservation,
 author = {Bergelson, Vitaly and Downarowicz, Tomasz},
 title = {On preservation of normality and determinism under arithmetic operations},
 year = {2025},
 howpublished = {Preprint, {arXiv}:2506.12929 [math.{DS}] (2025)},
 url = {https://arxiv.org/abs/2506.12929},
 arXiv = {arXiv:2506.12929}
}

@article{ManceFractalFunctions,
 author = {Mance, B.},
 title = {Number theoretic applications of a class of {Cantor} series fractal functions. {I}},
 fjournal = {Acta Mathematica Hungarica},
 journal = {Acta Math. Hung.},
 issn = {0236-5294},
 volume = {144},
 number = {2},
 pages = {449--493},
 year = {2014},
 language = {English},
 doi = {10.1007/s10474-014-0456-7},
 zbMATH = {6384136},
 Zbl = {1320.11069}
}

@article{AireyManceNormalPreserves,
 author = {Airey, Dylan and Mance, Bill},
 title = {Normality preserving operations for {Cantor} series expansions and associated fractals. {I}},
 fjournal = {Illinois Journal of Mathematics},
 journal = {Ill. J. Math.},
 issn = {0019-2082},
 volume = {59},
 number = {3},
 pages = {531--543},
 year = {2015},
 language = {English},
 zbMATH = {6640758},
 Zbl = {1361.11048}
}

@phdthesis{Wall,
author = {Wall, D. D.},
title = {Normal numbers},
school = {Univ. of California, Berkeley},
year = {1949},
address = {Berkeley, California}
}
\end{center}

\end{document}